\documentclass[a4paper,11pt]{article}
\usepackage{latexsym}
\usepackage{amsmath}		
\usepackage{amsthm}

\usepackage[dvipdfmx]{graphicx}
\usepackage{txfonts}
\usepackage{bm}	
\usepackage{color}
\usepackage{epic,eepic}

\usepackage{geometry}
\numberwithin{equation}{section}

\newtheorem{theorem}{Theorem}[section]
\newtheorem{proposition}[theorem]{Proposition}
\newtheorem{lemma}[theorem]{Lemma}
\newtheorem{remark}{Remark}[section]
\newtheorem{example}{Example}[section]
\newcommand{\RefPartI}[1]{{\color{red}#1}} 
 \renewcommand{\RefPartI}[1]{#1} 

\newcommand{\RR}{{\bf R}}
\newcommand{\ZZ}{{\bf Z}}

\newcommand{\dom}{{\rm dom\,}}

\newcommand{\argmax}{\arg \max}
\newcommand{\argmin}{\arg \min}
\newcommand{\suppp}{{\rm supp}\sp{+}}
\newcommand{\suppm}{{\rm supp}\sp{-}}
\newcommand{\finbox}{\hspace*{\fill}$\rule{0.17cm}{0.17cm}$}

\newcommand{\todaye}{\the\year/\the\month/\the\day}

\newcommand{\odotZ}{\overset{....}}

\begin{document}

\title{Discrete Decreasing Minimization, Part II:  \\
Views from Discrete Convex Analysis}

\author{Andr\'as Frank%
\thanks{MTA-ELTE Egerv\'ary Research Group,
Department of Operations Research, E\"otv\"os University, P\'azm\'any
P. s. 1/c, Budapest, Hungary, H-1117. 
e-mail:  {\tt frank\char'100 cs.elte.hu}. 
The research was partially supported by the
National Research, Development and Innovation Fund of Hungary
(FK\_18) -- No. NKFI-128673.
}
 \ \ and \ 
{Kazuo Murota%
\thanks{Department of Economics and Business Administration,
Tokyo Metropolitan University, Tokyo 192-0397, Japan, 
e-mail:  {\tt murota\char'100 tmu.ac.jp}. 
The research was supported by CREST, JST, Grant Number JPMJCR14D2, Japan, 
and JSPS KAKENHI Grant Number 26280004.  }}}

%%\author{Andr\'as Frank  \ \ and \ Kazuo Murota}

\date{August 2018 / May 2019 / July 2019 / June 2020}
%%\date{August 2018 (Version \today)}

%%\date{\today}
%%\date{August 2018}

\maketitle

\begin{abstract}
We continue to consider the discrete decreasing minimization problem on 
an integral base-polyhedron treated in Part~I.
The problem is to find a lexicographically minimal integral vector
in an integral base-polyhedron, 
where the components of a vector are arranged in a decreasing order.
This study can be regarded as a discrete counter-part of the work by Fujishige (1980) 
on the lexicographically optimal base 
and the principal partition of a base-polyhedron in continuous variables.
The objective of Part II is two-fold.
The first is to offer structural views from discrete convex analysis (DCA)
on the results of Part~I obtained by the constructive and algorithmic approach.
The second objective is to pave the way of DCA approach to 
discrete decreasing minimization on other discrete structures
such as the intersection of M-convex sets, flows, and submodular flows.

We derive the structural results in Part~I 
from fundamental facts on M-convex sets and M-convex functions in DCA.
The characterization of decreasing minimality
in terms of 1-tightening steps (exchange operations) 
is derived from the local condition of global minimality 
for M-convex functions, known as M-optimality criterion in DCA.
The min-max formulas, 
including the one for the square-sum of components, 
are derived as special cases of the Fenchel-type discrete duality in DCA.
A general result on the Fenchel-type discrete duality in DCA
offers a short alternative proof to  the statement that
the decreasingly minimal elements of an M-convex set form a matroidal M-convex set.

A direct characterization is given to the canonical partition,
which was constructed by an iterative procedure in Part~I.
This reveals the precise relationship 
between the canonical partition for the discrete case 
and the principal partition for the continuous case.
Moreover, this result entails a proximity theorem,
stating that every decreasingly minimal element 
is contained in the 
small box (unit box)
containing 
the (unique) fractional decreasingly minimal element 
(the minimum-norm point),
leading further to a continuous relaxation algorithm 
for finding a decreasingly minimal element of an M-convex set.
Thus the relationship between the continuous and discrete cases is completely clarified.

Furthermore, we present DCA min-max formulas  
for network flows, the intersection of two M-convex sets, and submodular flows.
\end{abstract}

{\bf Keywords}: \ 
base-polyhedra, 
discrete convex analysis, 
Fenchel-type min-max formula, 
lexicographically optimal, 
majorization, 
principal partition.
%%\subclass{52A41 \and 90C10}
%%52A41=Convex functions and convex programs
%%90C10=Integer programming
%%90C27 = combinatorial optimization
%%90C25 = convex programming
%%\memo{Keywords : 4 -- 6}

\newpage

\tableofcontents

%%\vfill

%%\pdfcreationdate

\newpage

%% murota 2018-08-25 /  2019-04-17 / 2019-05-22 / 2019-07-09 / 2020-06-30

\section{Introduction}
\label{SCintro}

We continue to consider discrete decreasing minimization on 
an integral base-polyhedron studied in Part~I.
The problem is to find a lexicographically minimal (dec-min) integral vector
in an integral base-polyhedron, 
where the components of a vector are arranged in a decreasing order
(see Section~\ref{SCdefnotat} for precise description of the problem).
While our present study deals with the discrete case,
the continuous case 
was investigated by Fujishige \cite{Fuj80} around 1980
under the name of lexicographically optimal bases of a base-polyhedron,
as a generalization of lexicographically optimal maximal flows
considered by Megiddo \cite{Meg74}.
Our study can be regarded as a discrete counter-part of 
the work by Fujishige 
\cite{Fuj80}, \cite[Section 9]{Fuj05book}
on the lexicographically optimal base 
and the principal partition of a base-polyhedron.
The objective of Part II is two-fold.
The first is to offer structural views from discrete convex analysis (DCA)
on the results of Part~I obtained by the constructive and algorithmic approach.
The second objective is to pave the way of DCA approach to 
discrete decreasing minimization on other discrete structures
such as the intersection of M-convex sets, flows, and submodular flows.
%%that we consider in Parts III and IV.

In Part I of this paper, we have shown the following:

\begin{itemize}
\setlength{\itemsep}{0pt}%
\item
 A characterization of decreasing minimality by 1-tightening steps (exchange operations),

\item
A (dual) characterization of decreasing minimality by the canonical chain,

\item
The structure of the dec-min elements as a matroidal M-convex set,

\item
A characterization of a dec-min element
as a minimizer of square-sum of components,

\item
A min-max formula for the square-sum of components,

\item
A strongly polynomial algorithm for finding a dec-min element and the canonical chain,

\item
Applications.
\end{itemize}

In contrast to the constructive and algorithmic approach in Part~I,
Part~II offers structural views from discrete convex analysis (DCA)
as well as from majorization.
The concept of majorization ordering offers a useful general framework
to discuss decreasing minimality.
The relevance of DCA to decreasing minimization is not surprising,
since an M-convex set is nothing but the set of integral points of an integral
base-polyhedron and a separable convex function on an M-convex set is an M-convex function.
In particular, the square-sum of components of a vector in an M-convex set is an M-convex function.
It will be shown that most of the important structural results
obtained in Part~I can be derived from the Fenchel-type discrete duality theorem,
which is a main characteristic of DCA
as compared with other theories of discrete functions such as \cite{Onn10book}.

In Section~\ref{SCmajorconn} of this paper
the basic facts about majorization are described.
In Section~\ref{SCseparconvmin} 
we derive the characterization of decreasing minimality
in terms of 1-tightening steps
 (exchange operations) 
from the local characterization of global minimality 
for M-convex functions, known as M-optimality criterion in DCA.
In Section~\ref{SCminmaxformula}, the min-max formulas, 
including the one for the square-sum of components, 
are derived as special cases of the Fenchel-type discrete duality in DCA.
We also show a novel min-max formula, which reinforces the 
link between the present study and the theory of majorization.
In Section~\ref{SCsetoptsols} we use a general result on the Fenchel-type discrete duality in DCA
for a short alternative proof to  the statement that
the decreasingly minimal elements of an M-convex set form a matroidal M-convex set.
The relationship between the continuous and discrete cases is clarified in Section~\ref{SCcompRZ}.
We reveal the precise relation between the canonical partition 
and the principal partition
by establishing an alternative direct characterization of the canonical partition,
which was constructed by an iterative procedure in Part~I.
The obtained result provides a 
proximity theorem,
stating that every dec-min element is contained in the 
small box (unit box)
containing the (unique) fractional dec-min element (the minimum-norm point),
and hence a continuous relaxation algorithm
for finding a decreasingly minimal element of an M-convex set.
In Section \ref{SCdcaflowM2}
we present DCA results relevant 
%%to Parts III and IV, where 
discrete decreasing minimization
%%is considered 
for the set of integral feasible flows,
the intersection of two M-convex sets, and 
the set of integral members of an integral submodular flow polyhedron.
In Appendix~\ref{SCprevwksurvey} we offer a brief survey of early papers and  books
related to decreasing minimization on base-polyhedra.
%% 2020-06-30

\subsection{Definition and notation}
\label{SCdefnotat}

We review some definitions and notations 
introduced in Part~I \cite{FM18part1}.

\subsection*{Decreasing minimality}

For a vector $x$, let $x {\downarrow}$ 
denote the vector obtained from $x$ by rearranging
its components in a decreasing order.  For example,
$x{\downarrow}=(5,5,4,2,1)$ when $x = (2,5,5,1,4)$.  
We call two vectors 
$x$ and $y$ (of same dimension) {\bf
value-equivalent} if 
$x{\downarrow}= y{\downarrow}$.  
For example, $(2,5,5,1,4)$ and
$(1,4,5,2,5)$ are value-equivalent while the vectors $(3,5,5,3,4)$ \
and \ $(3,4,5,4,4)$ \ are not.

A vector $x$ is {\bf decreasingly smaller} than vector $y$, 
in notation $x <_{\rm dec} y$,
 \ if $x{\downarrow}$ \ 
is lexicographically
smaller than  $y{\downarrow}$ in the sense that they are not
value-equivalent and 
$x{\downarrow}(j)<y{\downarrow}(j)$ 
for the smallest subscript $j$ for which $x{\downarrow}(j)$ and
$y{\downarrow}(j)$ differ.  
For example, $x = (2,5,5,1,4)$ is
decreasingly smaller than $y =(1,5,5,5,1)$ \ since 
$x{\downarrow}=(5,5,4,2,1)$ is \ lexicographically smaller than \
$y{\downarrow}=(5,5,5,1,1)$.
We write  $x\leq _{\rm dec} y$ to mean that
$x$ is decreasingly smaller than or value-equivalent to $y$.

For a set \ $Q$ \ of vectors, $x\in Q$ is 
 {\bf decreasingly minimal} ({\bf dec-min}, for short) 
if \ $x \leq _{\rm dec} y$ \ for every \ $y\in Q$.  
Note that the dec-min elements of $Q$ are value-equivalent.  
An element $m$ of $Q$ is dec-min if
its largest component is as small as possible, within this, its second
largest component (with the same or smaller value than the largest
one) is as small as possible, and so on.  
An element $x$ of $Q$ is said to be a {\bf max-minimized} element
 (a {\bf max-minimizer}, for short) 
if its largest component is as small as possible.

In an analogous way, for a vector $x$, we let $x {\uparrow}$ 
denote the vector obtained from $x$ by rearranging its
components in an increasing order.  
A vector $y$ is {\bf increasingly larger} than vector $x$, 
in notation $y >_{\rm inc} x$, if they are
not value-equivalent and 
$y{\uparrow}(j)> x{\uparrow}(j)$ holds 
for the smallest subscript $j$ for which 
$y{\uparrow}(j)$ and $x{\uparrow}(j)$ differ.  
We write $y \geq _{\rm inc} x$ if either $y >_{\rm inc} x$ 
or $x$ and $y$ are value-equivalent.  
Furthermore, we call an element $m$ of $Q$ 
{\bf increasingly maximal}
 ({\bf inc-max} for short) if its smallest component is as large as possible
over the elements of $Q$, within this its second smallest component is
as large as possible, and so on.

The {\bf decreasing minimization problem}
is to find a dec-min element of a given set $Q$ of vectors.
When the set $Q$ consists of integral vectors, we speak of
discrete decreasing minimization.
In Parts I and II of this series of papers, 
we deal with the case where 
the set $Q$ is an M-convex set,
i.e.,
the set of integral members of an integral base-polyhedron.
In Part III, the set $Q$ will be the integral feasible flows. 
%%In Part IV, 
The set $Q$ can be 
the intersection of two M-convex sets,
or more generally,
the set of integral members of an integral submodular flow polyhedron.
%% 2020-06-30

\subsection*{Base polyhedra}

Throughout the paper, $S$ denotes a finite nonempty ground-set.  
For a vector $m\in \RR\sp{S}$ (or function $m:S\rightarrow \RR$)
and a subset $X \subseteq S$, 
we use the notation
$\widetilde m(X)=\sum [m(v):  v\in X]$.  
The characteristic (or incidence) vector of a subset $Z \subseteq S$ is denoted by 
$\chi_{Z}$,  that is, 
$\chi_{Z}(v)=1$  if $v\in Z$ and 
$\chi_{Z}(v)=0$ otherwise.
For a polyhedron $B$, notation $\odotZ{B}$ 
(pronounced:  dotted $B$)
means the set of integral members
(elements, vectors, points) of $B$.

Let $b$ be a set-function for which 
$b(\emptyset)=0$ and
$b(X)=+\infty $ is allowed 
but $b(X)=-\infty $ is not.  
The submodular inequality for subsets 
$X,Y\subseteq S$ is defined by 
\begin{equation} \label{submodineq}
b(X) + b(Y) \geq b(X\cap Y) + b(X\cup Y).
\end{equation}
We say that $b$ is submodular if the submodular inequality holds
for every pair of subsets $X, Y\subseteq S$ with finite $b$-values.
A set-function $p$ is supermodular if $-p$ is submodular.  
A (possibly unbounded) 
{\bf base-polyhedron} $B$ in $\RR\sp{S}$ 
is defined by 
\begin{equation} \label{basepolysubmod}
B=B(b)=\{x\in \RR\sp{S}:  \widetilde x(S)=b(S), \
\widetilde x(Z)\leq b(Z) \ \hbox{ for every } \  Z\subset S\}.
\end{equation}
A nonempty base-polyhedron $B$ can also be defined by a supermodular
function $p$ for which $p(\emptyset )=0$ and $p(S)$ is finite as follows:  
\begin{equation} \label{basepolysupermod}
B=B'(p)=\{x\in \RR\sp{S}:  \widetilde x(S)=p(S), \
\widetilde x(Z)\geq p(Z) \ \hbox{ for every } \ Z\subset S\}.
\end{equation}

We call the set $\odotZ{B}$ 
of integral elements of an integral base-polyhedron $B$ an 
{\bf M-convex set}.  
Originally, this basic notion of 
discrete convex analysis was 
defined as a set of integral points in $\RR\sp{S}$ 
satisfying certain exchange axioms, 
and it has been known that 
the two properties are equivalent (\cite[Theorem 4.15]{Mdcasiam}).

\subsection*{Discrete convex functions}

For a function $\varphi: \ZZ \to \RR \cup \{ -\infty, +\infty \}$
the {\bf effective domain} of $\varphi$ is denoted as
$\dom \varphi = \{ k \in \ZZ : -\infty < \varphi(k) < +\infty \}$.
A function $\varphi: \ZZ \to \RR \cup \{ +\infty \}$
is called {\bf discrete convex} (or simply {\bf convex}) if 
\begin{equation}  \label{univarconvfndef}
  \varphi(k-1) + \varphi(k+1) \geq 2 \varphi(k)
\end{equation}
for all $k \in \dom \varphi$,
and  {\bf strictly convex}  if 
$\dom \varphi = \ZZ$ and
$\varphi(k-1) + \varphi(k+1) > 2 \varphi(k)$ 
for all $k \in \ZZ$.

A function 
$\Phi: \ZZ\sp{S} \to \RR \cup \{ +\infty \}$
of the form
\begin{equation}  \label{sepconvfndef}
  \Phi(x) = \sum [\varphi_{s}(x(s)):  s\in S] 
\end{equation}
is called a {\bf separable (discrete) convex function}
if, for each $s \in S$,
$\varphi_{s}: \ZZ \to \RR \cup \{ +\infty \}$
is a discrete convex function.
We call $\Phi$ a {\bf symmetric separable convex function} if 
$\varphi_{s}$ does not depend on $s$,
that is, if $\varphi_{s}=\varphi$ for all $s \in S$
for some discrete convex function $\varphi$.
We call $\Phi$ a
{\bf symmetric separable strictly convex function}
if  $\varphi$ is strictly convex.

%%% end file %%%

%% murota 2018-08-25 / 2019-05-22 / 2019-07-09

\section{Connection to majorization}
\label{SCmajorconn}

Majorization ordering (or dominance ordering)
is a well-established notion studied in diverse contexts including statistics and economics,
as described in Arnold--Sarabia \cite{AS18} and Marshall--Olkin--Arnold \cite{MOA11}.
In this section we describe the relevant results known in the literature of majorization,
and indicate a close relationship to
decreasing minimality investigated in our series of papers.

We have dual objectives in this section.
First, we intend to reinforce the connection 
between majorization and combinatorial optimization.
It is also hoped that this will lead to future applications
of our results in areas like statistics and economics, 
in addition to those areas related to graphs, networks, and matroids mentioned in 
the introduction of Part~I \cite{FM18part1}.
In economics, for example, egalitarian allocation for indivisible goods
can possibly be formulated and analyzed by means of discrete decreasing minimization.

Second, we point out substantial technical connections between majorization and our results in Part~I. 
We argue that some of our results can be derived from the combination of  
the classical results about majorization and 
the results of Groenevelt \cite{Gro91} for the minimization of separable convex functions
over the integer points in an integral base-polyhedron.
We also point out 
that some of the standard characterizations for least majorization
are associated with min-max duality relations 
in the case where the underlying set is
the integer points of an integral base-polyhedron
or the intersection of two integral base-polyhedra.

\subsection{Majorization ordering}
\label{SCmajorord}

We review standard results known in the literature of majorization
in a way suitable for our discussion.

Recall that $x{\downarrow}$ denotes the vector
obtained from a vector $x\in \RR\sp{n}$ by rearranging its
components in a decreasing order.  
Let $\overline{x}$ denote the vector 
whose $k$-th component $\overline{x}(k)$ is equal to the sum of the first $k$ components of
$x{\downarrow}$.
A vector $x$ is said to be {\bf majorized} by another vector $y$,
in notation $x \prec y$,
if $\overline{x} \leq \overline{y}$ and $\overline{x}(n) = \overline{y}(n)$.
It is easy to see  \cite[p.13]{MOA11} that
\begin{equation} \label{xymajorminusxy}
 x \prec y  \iff -x \prec -y.
\end{equation}
(At first glance, the equivalence in \eqref{xymajorminusxy} may look strange,
but observe that $x \prec y$ means that $x$ is more uniform than $y$,
which is equivalent to saying that $-x$ is more uniform than $-y$.) 
Majorization is discussed more often for real vectors,
but here we are primarily interested in integer vectors.

As an immediate adaptation of the standard results \cite[1.A.3 in p.14]{MOA11}, 
the following proposition gives equivalent conditions for majorization for integer vectors.
A {$T$-transform} (also called a {Robin Hood operation}) 
means a linear transformation of the form
$T = (1 - \lambda ) I +  \lambda Q$, 
where $0 \leq \lambda \leq 1$ and $Q$ is a permutation matrix that interchanges just two elements
(transposition).
In other words,  a $T$-transform is a mapping of the form
$x \mapsto x + \hat \lambda (\chi_{s}-\chi_{t})$
with $0 \leq \hat \lambda \leq x(t)-x(s)$.
It is noteworthy that this operation with $\hat \lambda =1$ 
corresponds to the basis exchange in an integral base-polyhedron.

\begin{proposition}  \label{PRmajorchar}
%%[{\cite[1.A.3 in p.14]{MOA11}}] 
The following conditions are equivalent for $x, y \in \ZZ\sp{n}:$

{\rm (i)}
 $x \prec y$ ($x$ is majorized by $y$), that is,
\begin{equation} \label{minklargestsum}
 \sum_{i=1}\sp{k} x{\downarrow}(i) \leq \sum_{i=1}\sp{k} y{\downarrow}(i)
\quad (k=1,\ldots,n-1),
\qquad
 \sum_{i=1}\sp{n} x{\downarrow}(i) = \sum_{i=1}\sp{n} y{\downarrow}(i).
\end{equation}

{\rm (ii)}
 $x = y P$ for some doubly stochastic matrix $P$,
where $x$ and $y$ are regarded as row vectors.

{\rm (iii)}
$x$ can be derived from y by successive applications of a finite
number of $T$-transforms.

{\rm (iv)}
$\displaystyle  \sum_{i=1}\sp{n} \varphi ( x(i) ) \leq \sum_{i=1}\sp{n} \varphi ( y(i) )$ \ 
for all discrete convex functions
$\varphi: \ZZ \to \RR$.

{\rm (v)}
$\displaystyle 
\sum_{i=1}\sp{n} x(i)= \sum_{i=1}\sp{n} y(i)$ \ 
and
$\displaystyle 
\sum_{i=1}\sp{n} (x(i) - a )\sp{+} \leq  \sum_{i=1}\sp{n} (y(i) - a )\sp{+}$  \ 
for all $a \in \ZZ$.
where $(z)\sp{+} = \max\{ 0, z \}$ for any $z \in \ZZ$.
\finbox 
\end{proposition}

Let $D$ be an arbitrary subset of $\ZZ\sp{n}$.
An element $x$ of $D$ is said to be {\bf least majorized} 
in $D$ 
if $x$ is majorized by all $y \in D$.
A least majorized element may not exist in general, as the following example shows.

\begin{example} \rm \label{EX}
Let $D= \{ (2, 0, 0, 0), \ (1, -1, 1, 1)\}$.
For $x = (2, 0, 0, 0)$ and $y = (1, -1, 1, 1)$ we have
$x{\downarrow} =(2,0,0,0)$ and $y{\downarrow} =(1,1,1,-1)$.
Therefore, $x = (2, 0, 0, 0)$ is increasingly maximal in $D$
and $y = (1, -1, 1, 1)$ is decreasingly minimal in $D$.
However, there exists no least majorized element in $D$,
since $\overline{x} = (2,2,2,2)$ and
$\overline{y} = (1,2,3,2)$,
for which neither
$\overline{x} \leq \overline{y}$
nor
$\overline{y} \leq \overline{x}$
holds.
We note that $D$ here arises 
from the intersection of two integral base-polyhedra
(see Section \RefPartI{3.4} of Part~I \cite{FM18part1}).
\finbox
\end{example}

\begin{remark} \rm \label{RMsubmajor}
In discussing the existence and properties of a least majorized element,
we are primarily concerned with a subset $D$ of $\ZZ\sp{n}$ 
whose elements have a constant component-sum.
If the component-sum is not constant on $D$, 
we need to introduce a more general notion \cite{Tami95}.
A vector $x$ is said  to be {\bf weakly submajorized}
by another vector $y$, denoted $x \prec_{\rm w} y$,
 if $\overline{x} \leq \overline{y}$.
An element $x$ of $D$ is said to be {\bf least weakly submajorized} 
in $D$ 
if $x$ is weakly submajorized by all $y \in D$.
The distinction of 
``weakly submajorized'' and  ``majorized'' is not necessary
for a base-polyhedron or the intersection of base-polyhedra,
whereas we have to distinguish these concepts for a g-polymatroid
and a submodular flow polyhedron.
\finbox
\end{remark}

\begin{remark} \rm \label{RMmajorminmax}
The characterization of a least majorized element
in (iv) in Proposition \ref{PRmajorchar} can be
associated with a min-max duality relation, 
which is given by 
(\ref{minmaxsymsep}) in Section \ref{SCfencsepar}
when the underlying set $D$ is 
an M-convex set (= the integer points of an integral base-polyhedron),
and by (\ref{minmaxsymsepInter}) in Section \ref{SCfencseparInter}
when $D$ is the intersection of two M-convex sets.
For an M-convex set,
the min-max formula associated with (v) in Proposition \ref{PRmajorchar} is given 
by \eqref{decmintruncsum}
in Theorem \ref{THtotalexcess} in Section \ref{SCtotalexcess}. 
\finbox
\end{remark}

\subsection{Majorization and decreasing-minimality}
\label{SCmajordecmin}

Majorization and decreasing-minimality are closely related,
as is explicit in Tamir \cite{Tami95}.

\begin{proposition}  \label{PRmajorisdec}
If $x \prec y$, then $x \leq_{\rm dec} y$ and $x \geq_{\rm inc} y$.
\end{proposition}
\begin{proof}
Suppose that $x \prec y$.
If $\overline{x} = \overline{y}$, then
$x{\downarrow} = y{\downarrow}$, and hence $x$ and $y$ are value-equivalent.
If $\overline{x} < \overline{y}$, then there exists an index $k$ 
 with $1 \leq k \leq n$ such that
$x{\downarrow}(i) = y{\downarrow}(i)$ for $i=1,\ldots, k-1$ and
$x{\downarrow}(k) <  y{\downarrow}(k)$.
This shows that $x$ is decreasingly smaller than $y$.
In either case, we have $x \leq_{\rm dec} y$.
Since $x \prec y$,  we have $-x \prec -y$ by (\ref{xymajorminusxy}).
By the above argument applied to $(-x,-y)$, we obtain $-x \leq_{\rm dec} -y $,
which is equivalent to $x \geq_{\rm inc} y$.
\end{proof}

\begin{remark} \rm \label{RMmajorisdecconverse}
The converse of Proposition \ref{PRmajorisdec} is not true.
That is, 
$x \prec y$ does not follow from 
$x \leq_{\rm dec} y \ \ \mbox{\rm and} \ \  x \geq_{\rm inc} y$.
For instance, for
$x = (2, 2, -2, -2)$ and  $y = (3, 0, 0, -3)$
we have $x \leq_{\rm dec} y$ and $x \geq_{\rm inc} y$, but $x \not\prec y$
since $\overline{x}=(2,4,2,0)$ and $\overline{y}=(3,3,3,0)$.
\finbox
\end{remark}

\begin{proposition} \label{PRmajorsameasdec}
Let $D$ be an arbitrary subset of $\ZZ\sp{n}$ 
and assume that $D$ admits a least majorized element.
For any $x \in D$ the following three conditions are equivalent.

\smallskip \noindent {\rm (A)}
$x$ is least majorized in $D$.

\smallskip \noindent {\rm (B)}
$x$ is decreasingly minimal in $D$.

\smallskip \noindent {\rm (C)}
$x$ is increasingly maximal in $D$.
\end{proposition}
\begin{proof}
(A)$\rightarrow$(B) \ 
By Proposition \ref{PRmajorisdec}, 
a least majorized element is decreasingly minimal.

(B)$\rightarrow$(A) \ 
Take a least majorized element $y$, 
which exists by the assumption.
By definition we have
$\overline{y} \leq \overline{x}$.
Since $x \leq_{\rm dec} y$, we have
either
$x{\downarrow}= y{\downarrow}$ 
or there exists an index $k$ with $1 \leq k \leq n$ such that
$x{\downarrow}(i) = y{\downarrow}(i)$ for $i=1,\ldots, k-1$ and
$x{\downarrow}(k) <  y{\downarrow}(k)$.
In the latter case we have
$\overline{x}(k) < \overline{y}(k)$,
which contradicts $\overline{y} \leq \overline{x}$.
Therefore we have
$x{\downarrow}= y{\downarrow}$,
which implies that $x$ is a least majorized element.

(A)$\leftrightarrow$(C) \ 
For any $y \in D$, we have
\[
x \prec y 
\iff -x \prec -y 
\iff  -x \leq_{\rm dec} -y 
\iff x \geq_{\rm inc} y 
\]
by (\ref{xymajorminusxy}) and
(A)$\leftrightarrow$(B) for $(-x,-y)$.
\end{proof}

\subsection{Majorization in integral base-polyhedra}
\label{SCmajorbasepoly}

In this section we consider majorization ordering 
for integer points in an integral base-polyhedron.
In discrete convex analysis,
the set of the integer points of an integral base-polyhedron 
is called an M-convex set.

The following fundamental fact has long been recognized by experts, though 
it was difficult for the present authors to identify its origin in the literature
(see Remark~\ref{RMleastmajorbase}).

\begin{theorem} \label{THmajorExistInBase}
The set of the integer points of an integral base-polyhedron 
admits a least majorized element.  
\finbox
\end{theorem}

This fact can be regarded as a corollary of the following fundamental result 
of Groenevelt \cite{Gro91},
which is already mentioned in 
Section \RefPartI{6} of Part~I \cite{FM18part1}.

\begin{proposition}[Groenevelt \cite{Gro91}; cf.~{\cite[Theorem 8.1]{Fuj05book}}]
 \label{PRgroenevelt}
Let $B$ be an integral base-polyhedron,
$\odotZ{B}$ be the set of its integral elements,
and $\Phi(x) = \sum [\varphi_{s}(x(s)):  s\in S]$ \ for $x\in \ZZ\sp{S}$, 
where $\varphi_{s}: \ZZ \to \RR \cup \{ +\infty \}$
 is a discrete convex function for each $s\in S$.
An element $m$ of $\odotZ{B}$ is a minimizer of $\Phi(x)$
if and only if
$\varphi_{s}(m(s)+1) + \varphi_{t}(m(t)-1) \geq \varphi_{s}(m(s)) + \varphi_{t}(m(t))$
whenever 
$m+\chi_{s}-\chi_{t} \in \odotZ{B}$.
\finbox
\end{proposition}

Theorem \ref{THmajorExistInBase} can be derived from the combination of 
Proposition \ref{PRgroenevelt} with Proposition \ref{PRmajorchar}.
Let  $m\in \odotZ{B}$ be a minimizer of the square-sum
$\sum [x(s)\sp{2}:  s\in S]$ over $\odotZ{B}$;
note that such $m$ exists.
Then,  by Proposition \ref{PRgroenevelt} (only-if part),
we have 
$(m(s)+1)\sp{2} + (m(t)-1)\sp{2}  \geq m(s)\sp{2}  + m(t)\sp{2}$
whenever 
$m+\chi_{s}-\chi_{t} \in \odotZ{B}$.
Here the inequality 
$(m(s)+1)\sp{2} + (m(t)-1)\sp{2}  \geq m(s)\sp{2}  + m(t)\sp{2}$
is equivalent to 
$m(s) - m(t) + 1 \geq 0$,
which implies
$\varphi(m(s)+1) + \varphi(m(t)-1) \geq \varphi(m(s)) + \varphi(m(t))$
for any discrete convex function $\varphi: \ZZ \to \RR$.
Therefore,  by Proposition \ref{PRgroenevelt} (if part),
$m$ is a minimizer of any symmetric separable convex function
$\sum [\varphi(x(s)):  s\in S]$ over $\odotZ{B}$.
By the equivalence of (i) and (iv) in Proposition \ref{PRmajorchar},
this element $m$ is a least majorized element of $\odotZ{B}$.

The combination of Theorem \ref{THmajorExistInBase} and
Proposition \ref{PRmajorsameasdec} implies the following.

\begin{theorem} \label{decminmajorB} 
Let $B$ be an integral base-polyhedron
and $\odotZ{B}$ be the set of its integral elements.
An element $m$ of $\odotZ{B}$ is decreasingly minimal if and only if
$m$ is least majorized in $\odotZ{B}$.
\finbox
\end{theorem}

\begin{remark} \rm  \label{RMdecminmajor}
In Theorem \RefPartI{3.5} of Part~I \cite{FM18part1} 
we have shown that a dec-min element of $\odotZ{B}$ has the property
\eqref{minklargestsum},
which is referred to as ``min $k$-largest-sum'' in \cite{FM18part1}.
This implies that any dec-min element of $\odotZ{B}$ 
is a least majorized element of $\odotZ{B}$.
Since a dec-min element always exists,
this theorem also implies 
the existence of a least majorized element in  $\odotZ{B}$.
\finbox
\end{remark}

\begin{remark} \rm  \label{RMleastmajorbase}
A variant of majorization concept, ``weak submajorization'' 
(cf., Remark \ref{RMsubmajor}), is investigated 
for integral g-polymatroids by Tamir \cite{Tami95}
and for jump systems by Ando \cite{And96}.
These results are a direct extension of Theorem \ref{THmajorExistInBase}.
Therefore, we may safely say that Theorem \ref{THmajorExistInBase}
with the above proof was known to experts before 1995.
\finbox
\end{remark} 

%%% end file %%%

%% murota 2018-08-25 / 2019-05-22  / 2019-07-09 / 2019-08-21 / 2019-09-19 / 

\section{Convex minimization and decreasing minimality}
\label{SCseparconvmin}

In this section we shed the light of discrete convex analysis  on 
the following results obtained in Part~I \cite{FM18part1}. 
More specifically, we derive these results from 
the optimality criterion for  M-convex functions,
which is described in Section \ref{SCdcaMmin}.

\begin{theorem}[{\cite[Theorem \RefPartI{3.3}, (A) \& (C1)]{FM18part1}}]
 \label{THnoTighten2}
An element $m$ of $\odotZ{B}$ is a dec-min element of $\odotZ{B}$ 
if and only if there is no 1-tightening step for $m$.
\finbox
\end{theorem}

\begin{theorem}[{\cite[Corollary \RefPartI{6.3}]{FM18part1}}]
 \label{THdecminPhistrict2}
Let $\Phi(x) = \sum [\varphi(x(s)):  s\in S]$  
be a symmetric separable convex function 
with $\varphi: \ZZ \to \RR$.
An element $m$ of $\odotZ{B}$ is a minimizer of $\Phi$ if $m$ 
is a dec-min element of $\odotZ{B}$,
and the converse is also true if, in addition, $\Phi$ is strictly  convex.
\finbox
\end{theorem}

It should be clear in the above that $\odotZ{B}$ denotes an M-convex set
(the set of integral points of an integral base-polyhedron),
and a {\bf 1-tightening step} for $m\in \odotZ{B}$ 
means the operation of replacing $m$ to $m+\chi_{s}-\chi_{t}$
for some $s, t \in S$ such that
$m(t)\geq m(s)+2$ and $m+\chi_{s}-\chi_{t} \in \odotZ{B}$.

\subsection{Convex formulation of decreasing minimality}
\label{SCdecminconvform}

A dec-min element can be characterized as a minimizer of
`rapidly increasing' convex function.
This characterization enables us to make use of 
discrete convex analysis 
in investigating decreasing minimality.

%%\memo{Major revision 2019-05-28; after arXiv Ver.2}

We say that a positive-valued function $\varphi: \ZZ \to \RR$ is $N$-increasing,
where $N>0$,  if
\begin{equation} \label{basevarphirapinc}
 \varphi(k+1) \geq  N  \  \varphi(k)  > 0 
\qquad (k \in \ZZ) .
\end{equation}
%%$\varphi(k+1) \geq  N \  \varphi(k)  > 0$ for all $k \in \ZZ$.
With the choice of a sufficiently large $N$, this concept formulates the intuitive notion that
$\varphi$ is ``rapidly increasing.'' 
%%For example, $\varphi (k) = n\sp k$ is rapidly
%%increasing in this sense, but $\varphi (k)=k\sp{2}$ is not.
An $N$-increasing function $\varphi$ with $N \geq 2$ is strictly convex, since
$\varphi(k-1) +  \varphi(k+1) > \varphi(k+1) \geq N \varphi(k) \geq 2 \varphi(k)$.

As is easily expected, $x <_{\rm dec} y$ is equivalent to $\Phi(x) < \Phi (y)$
defined by such $\varphi$, as follows.

\begin{proposition} \label{PRdecsmallvalsmall}
Assume $|S| \geq 2$ and that $\varphi$ is $|S|$-increasing.
A vector $x \in \ZZ\sp{S}$ is decreasingly-smaller than
a vector $y \in \ZZ\sp{S}$ 
if and only if $\Phi(x) < \Phi(y)$.
\end{proposition}
\begin{proof}
For $x \in \ZZ\sp{S}$ and $k \in \ZZ$, 
let $\Theta(x,k)$ denote the number of elements $s$ of $S$ with $x(s)=k$, 
i.e.,
$\Theta(x,k) = |\{ s \in S  :  x(s)=k \} | $.
Then we have
\begin{equation} \label{baseconvcosttheta}
 \Phi(x) = \sum_{k}  \Theta(x,k) \varphi(k) .
\end{equation}
Obviously, $\Phi(x) = \Phi(y)$ if $x$ and $y$ are value-equivalent.
Suppose that  $x$ is not value-equivalent to $y$,
and let $\hat k$ be the largest $k$ with 
$\Theta(x,k) \not= \Theta(y,k)$.
By definition,  $x$ is decreasingly-smaller than $y$
if and only if
$\Theta(x,\hat k) < \Theta(y, \hat k)$.

We show that
$\Theta(x,\hat k) < \Theta(y,\hat k)$ implies $\Phi(x) < \Phi(y)$.
Then the converse also follows from this (by exchanging the roles of $x$ and $y$).
Let 
$T:= \sum_{k > \hat k}  \Theta(x,k)  \varphi(k)
 = \sum_{k > \hat k}  \Theta(y,k) \varphi(k) $.
It follows from
\begin{align} 
\Phi(x)
& = T +  \Theta(x,\hat k)  \varphi(\hat k) 
 + \sum_{k < \hat k} \Theta(x,k)  \varphi(k) 
\notag \\
& \leq T +  \Theta(x,\hat k) \varphi(\hat k) 
 + \varphi(\hat k - 1) \sum_{k < \hat k} \Theta(x,k)   
\notag  \\
& \leq T +  \Theta(x,\hat k) \varphi(\hat k) 
 + \varphi(\hat k) \ \frac{1}{|S|} \sum_{k < \hat k} \Theta(x,k)   
\notag  \\
& \leq T +  ( \Theta(x,\hat k) + 1) \varphi(\hat k) ,
\label{basePhix}
%%%%%%%%%%%%%%%%%%
\\
\Phi(y)
& = T +  \Theta(y,\hat k)  \varphi(\hat k) 
 + \sum_{k < \hat k} \Theta(y,k)  \varphi(k) 
\notag \\
& \geq T +  \Theta(y,\hat k) \varphi(\hat k) 
\label{basePhiy}
\end{align}
that
\begin{align} 
\Phi(y) - \Phi(x)
& \geq 
 ( \Theta(y,\hat k) - \Theta(x, \hat k) -1 ) \varphi(\hat k)
\geq 0.
\label{basePhiyPhix}
\end{align}
Here we can exclude the possibility of equality.
Suppose we have equalities in \eqref{basePhiyPhix}.
This implies that
$\Theta(y,\hat k) = \Theta(x, \hat k) + 1$
and that we have equalities throughout 
\eqref{basePhix} and \eqref{basePhix}.
From \eqref{basePhix} we obtain
$\sum_{k < \hat k} \Theta(x,k) = |S|$, from which
$\Theta(x,k) = 0$ for all $k \geq \hat k$.
Therefore we have
$\Theta(y,k) = 0$ for all $k > \hat k$
and
$\Theta(y,\hat k) = 1$.
From \eqref{basePhiy}, on the other hand, 
we obtain $\Theta(y,k) = 0$ for all $k < \hat k$.
This contradicts the relation $\sum_{k} \Theta(y,k) = |S| \geq 2$. 
\end{proof}
%%\hfill \memo{Major revision ends here (2019-05-28; after arXiv Ver.2)} \\

By Proposition~\ref{PRdecsmallvalsmall} above,
the problem of finding a dec-min element
can be recast into a convex minimization problem.
It is emphasized that for this equivalence, the underlying set
may be any subset of $\ZZ\sp{S}$
(not necessarily an M-convex set).

\begin{proposition} \label{PRbaseegalconvmin}
Let $D$ be an arbitrary subset of $\ZZ\sp{S}$, where $|S| \geq 2$,
and assume that $\varphi$ is $|S|$-increasing.
An element $m$ of $D$ is decreasingly-minimal in $D$ 
if and only if it minimizes 
$\Phi(x) = \sum_{s \in S} \varphi ( x(s) )$ 
among all members of $D$.
\finbox
\end{proposition}

\begin{remark} \rm \label{RMdecminconvchar}
The characterization of a decreasingly-minimal elements
as a minimizer of a rapidly increasing convex function 
in Proposition~\ref{PRbaseegalconvmin}
is not particularly new.
Similar ideas are scattered in the literature
of related topics such as majorization
(Marshall--Olkin--Arnold \cite{MOA11})
and shifted optimization (Levin--Onn \cite{LO16shift}).
\finbox
\end{remark}

\begin{remark} \rm \label{RMmajdecminincmaxconv}
The relations of being majorized ($\prec$), weakly submajorized ($\prec_{\rm w}$), 
and decreasingly-smaller ($\leq_{\rm dec}$)
are characterized with reference to different classes 
of symmetric separable convex functions as follows
(Proposition~\ref{PRmajorchar}, \cite[4.B.2]{MOA11}, and
Proposition~\ref{PRdecsmallvalsmall}):
\begin{itemize}
\item
$x \prec y$ \ \ \ $\iff$ 
$\displaystyle  \sum_{i=1}\sp{n} \varphi ( x(i) ) \leq \sum_{i=1}\sp{n} \varphi ( y(i) )$ 
\ \ 
for all convex $\varphi$,

\item
$x \prec_{\rm w} y$ $\iff$ 
$\displaystyle  \sum_{i=1}\sp{n} \varphi ( x(i) ) \leq \sum_{i=1}\sp{n} \varphi ( y(i) )$ 
\ \ for all increasing (nondecreasing) convex $\varphi$,

\item
$x \leq_{\rm dec} y$ $\iff$ 
$\displaystyle  \sum_{i=1}\sp{n} \varphi ( x(i) ) \leq \sum_{i=1}\sp{n} \varphi ( y(i) )$ 
\ \ for all rapidly increasing convex $\varphi$.
\finbox
\end{itemize}
\end{remark}

\subsection{M-convex function minimization in discrete convex analysis}
\label{SCdcaMmin}

In this section we introduce
M-convex functions, a fundamental concept in 
discrete convex analysis 
\cite{Mdcasiam},
along with a local optimality condition
for a minimizer of an M-convex function.
Since a separable convex function on an M-convex set is an M-convex function
(cf.~Section \ref{SCseparDCA}),
this optimality criterion renders alternative proofs of Theorems
\ref{THnoTighten2} and \ref{THdecminPhistrict2}
about the dec-min elements of an M-convex set
(cf.~Section \ref{SCdcaproof}).

For a vector $z \in \RR\sp{S}$ in general, we define 
the positive and negative supports of $z$ as
\begin{equation} \label{vecsupportdef}
 \suppp(z) = \{ s  \in S  :  z(s) > 0 \},
\qquad 
 \suppm(z) = \{ t \in S  :  z(t) < 0 \}.
\end{equation}
For a function $f: \ZZ\sp{S} \to \RR \cup \{ -\infty, +\infty \}$,
the effective domain is defined as $\dom f = \{ x \in \ZZ\sp{S} : -\infty < f(x) < +\infty \}$.

A function
$f: \ZZ\sp{S} \to \RR \cup \{ +\infty \}$
with $\dom f \not= \emptyset$
is called {\bf M-convex} if,
for any $x, y \in \ZZ\sp{S}$ and $s \in \suppp(x-y)$, 
there exists some $t \in \suppm(x-y)$ such that
\begin{equation}  \label{mconvexZ}
f(x) + f(y)   \geq  f(x-\chi_{s}+\chi_{t}) + f(y+\chi_{s}-\chi_{t}) .
\end{equation}
In the above statement we may change
``for any $x, y \in \ZZ\sp{S}$'' to ``for any $x, y \in \dom f$''
 since if $x \not\in \dom f$ or $y \not\in \dom f$,
(\ref{mconvexZ}) trivially holds with $f(x) + f(y)  = +\infty$. 
We often refer to this defining property as the {\bf exchange property}
of an M-convex function.
It follows from this definition that $\dom f$ 
consists of the integer points of an integral base-polyhedron (an M-convex set).
A function $f$ is called {\bf M-concave} 
if $-f$ is {\rm M}-convex.
We remark that the exchange property \eqref{mconvexZ}
of an M-convex function is a quantitative extension of the
symmetric exchange property of matroid bases.

A function
$f: \ZZ\sp{S} \to \RR \cup \{ +\infty \}$
with $\dom f \not= \emptyset$
is called {\bf {\rm M}$\sp{\natural}$-convex} if,
for any $x, y \in \ZZ\sp{S}$ and $s \in \suppp(x-y)$, 
we have (i)
\begin{equation}  \label{mnatconvex1Z}
f(x) + f(y)  \geq  f(x - \chi_{s}) + f(y+\chi_{s})
\end{equation}
or (ii) there exists some $t \in \suppm(x-y)$ for which (\ref{mconvexZ}) holds.
It follows from this definition that the effective domain of an {\rm M}$\sp{\natural}$-convex 
function consists of 
the integer points of an integral g-polymatroid \cite{Fra11book};
such a set is called  {\bf {\rm M}$\sp{\natural}$-convex set} in DCA.
An {\rm M}-convex function is {\rm M}$\sp{\natural}$-convex.
A function $f$ is called {\bf M$\sp{\natural}$-concave} 
if $-f$ is {\rm M}$\sp{\natural}$-convex.

The following is a local characterization of global minimality
for M- or M$\sp{\natural}$-convex functions, called the M-optimality criterion.

\begin{theorem}[{\cite[Theorem 6.26]{Mdcasiam}}] \label{THmopt}
Let $f: \ZZ\sp{S} \to \RR \cup \{ +\infty \}$ be an M$\sp{\natural}$-convex function,
and $x\sp{*}  \in \dom f$.
Then $x\sp{*} $ is a minimizer of $f$ if and only if
it is locally minimal in the sense that
\begin{align} 
& f(x\sp{*}) \leq f(x\sp{*} + \chi_{s} - \chi_{t})  \quad  \mbox{\rm for all } \ s, t \in S ,
\label{mnatfnlocmin1}
\\
& f(x\sp{*}) \leq f(x\sp{*} + \chi_{s})  \quad\quad\quad  \mbox{\rm for all } \ s \in S ,
\label{mnatfnlocmin2}
\\
& f(x\sp{*}) \leq f(x\sp{*} - \chi_{t})  \quad\quad\quad  \mbox{\rm for all } \ t \in S .
\label{mnatfnlocmin3}
\end{align}
If $f$ is  M-convex,
$x\sp{*} $ is a minimizer of $f$ if and only if {\rm (\ref{mnatfnlocmin1})} holds.
\finbox
\end{theorem}

\subsection{Separable convex function minimization in discrete convex analysis}
\label{SCseparDCA}

Minimization of a separable convex function over
the set of integral points of an integral base-polyhedron 
can be treated successfully as a special case of M-convex function minimization
presented in Section \ref{SCdcaMmin}.

We consider a function 
$\Phi: \ZZ\sp{S} \to \RR \cup \{ +\infty \}$
of the form
\begin{equation} \label{gensepar}
  \Phi(x) = \sum [\varphi_{s}(x(s)):  s\in S] ,
\end{equation}
where, for each $s \in S$,  the function
$\varphi_{s}: \ZZ \to \RR \cup \{ +\infty \}$
is discrete convex
(i.e., $\varphi_{s}(k-1) + \varphi_{s}(k+1) \geq 2 \varphi_{s}(k)$ 
for all $k \in \dom \varphi_{s}$).
Such function $\Phi$ is called a separable (discrete) convex function. 
We call $\Phi$ symmetric if $\varphi_{s}=\varphi$ for all $s \in S$.

Let $\odotZ{B}$ be the set of integral points of an integral base-polyhedron $B$.
The problem we consider is:
\begin{equation} \label{sepminB1}
\mbox{Minimize } \  \Phi(x) = \sum [\varphi_{s}(x(s)):  s\in S] 
 \ \mbox{ subject to } \ 
x \in \odotZ{B}.
\end{equation}
Using the indicator function 
$\delta: \ZZ\sp{S} \to \RR \cup \{ +\infty \}$
of $\odotZ{B}$ defined as
\begin{equation} \label{indicBdef}
 \delta(x) = 
   \left\{  \begin{array}{ll}
    0   & (x \in \odotZ{B}),  \\
    +\infty  &  (\mbox{otherwise}),  \\
             \end{array}  \right.
\end{equation}
we can rewrite (\ref{sepminB1}) as 
\begin{equation} \label{sepminB2}
\mbox{Minimize } \  \Phi(x) + \delta(x).
\end{equation}

This problem is amenable to discrete convex analysis,
since the separable convex function  $\Phi$ is M$\sp{\natural}$-convex,
the indicator function $\delta$ of an M-convex set
is M-convex, and  moreover, the function $\Phi + \delta$ is M-convex.
Indeed it is easy to verify that these functions satisfy the defining exchange property.
In this connection it is noted that
the sum of an M-convex function and an M$\sp{\natural}$-convex function 
is not necessarily M$\sp{\natural}$-convex, but 
the sum of an M-convex function and a separable convex function is always M-convex
(cf.~Remark \ref{RMm2convex} in Section \ref{SCfencml}).

An application of the M-optimality criterion (Theorem~\ref{THmopt})
to our function $\Phi + \delta$ gives 
the important result due to Groenevelt \cite{Gro91}
shown in Proposition~\ref{PRgroenevelt}.
In the special case of symmetric separable convex functions,
with $\varphi_{s} = \varphi$ for all $s \in S$,
we can relate the condition
given in Proposition~\ref{PRgroenevelt}
to 1-tightening steps.
Recall that a 1-tightening step for $m\in \odotZ{B}$ 
means the operation of replacing $m$ to $m+\chi_{s}-\chi_{t}$
for some $s, t \in S$ such that
$m(t)\geq m(s)+2$ and $m+\chi_{s}-\chi_{t} \in \odotZ{B}$.

\begin{proposition}  \label{PRtighteningB}
For any symmetric separable discrete convex function 
$\Phi(x) = \sum [\varphi(x(s)):  s\in S]$
with $\varphi: \ZZ \to \RR \cup \{ +\infty \}$,
an element $m$ of $\odotZ{B}$ is a minimizer of $\Phi$ over $\odotZ{B}$ if 
there is no 1-tightening step for $m$.
The converse is also true if $\varphi$ is strictly convex.
\end{proposition}
\begin{proof}
By Proposition \ref{PRgroenevelt},
$m$ is a minimizer of $\Phi$ 
if and only if
\begin{equation*}  %%\label{tightenPrf1}
\varphi(m(s)+1) + \varphi(m(t)-1) \geq \varphi(m(s)) + \varphi(m(t))
\end{equation*}
for all $s, t \in S$ such that
$m+\chi_{s}-\chi_{t} \in \odotZ{B}$.
By the convexity of $\varphi$, we have this inequality
if $m(t)\leq m(s)+1$, and the converse is also true when $\varphi$ is strictly convex.
Finally we note that there is no 1-tightening step for $m$
if and only if
$m(t)\leq m(s)+1$ for all $s, t \in S$ such that
$m+\chi_{s}-\chi_{t} \in \odotZ{B}$.
\end{proof}

\subsection{DCA-based proofs of the theorems}
\label{SCdcaproof}

The combination of Proposition~\ref{PRtighteningB} 
with Proposition~\ref{PRbaseegalconvmin}
provides alternative proofs of Theorems \ref{THnoTighten2} and \ref{THdecminPhistrict2}.

\paragraph{Proof of Theorem~\ref{THnoTighten2}:}
Let $\Phi$ be a symmetric separable convex function with
rapidly increasing $\varphi$.
By Proposition~\ref{PRbaseegalconvmin}, 
$m$ is dec-min if and only if $m$ is a minimizer of $\Phi$.
On the other hand, since $\Phi$ is strictly convex, Proposition~\ref{PRtighteningB}
shows that
$m$ is a minimizer of $\Phi$
if and only if 
there is no 1-tightening step for $m$.
Therefore, $m$ is a dec-min element of $\odotZ{B}$ 
if and only if there is no 1-tightening step for $m$.

\paragraph{Proof of Theorem~\ref{THdecminPhistrict2}:}
Let $\Phi$ be a symmetric separable convex function.
By Proposition~\ref{PRtighteningB},
$m$ is a minimizer of $\Phi$
if there is no 1-tightening step for $m$;
and the converse is also true for strictly convex $\Phi$.
Theorem~\ref{THnoTighten2},
on the other hand, shows that 
there is no 1-tightening step for $m$
if and only if $m$ is a dec-min element.
Therefore,
$m$ is a minimizer of $\Phi$ if $m$ is a dec-min element of $\odotZ{B}$;
and the converse is also true for strictly convex $\Phi$.

\subsection{Extension to generalized polymatroids}
\label{SCgpolymKM}

%%\memo{Major revision 2019-08-20; after arXiv Ver.3}

In this section we shed a light of DCA on 
the majorization ordering and decreasing minimality
in generalized polymatroids (g-polymatroids).
Let $Q$ be an integral g-polymatroid on the ground set $S$ and
$\odotZ{Q}$ the set of its integral points;
see \cite{Fra11book} for the basic facts about g-polymatroids.
It is shown by Tamir \cite{Tami95} that
$\odotZ{Q}$ admits a least weakly submajorized element
(cf., Remark \ref{RMsubmajor} for this terminology).
By Remark \ref{RMmajdecminincmaxconv}
this is equivalent to saying that 
there exists an element of $\odotZ{Q}$ that simultaneously minimizes 
all symmetric separable functions 
$\sum_{s \in S} \varphi ( x(s) )$ 
defined by an
increasing discrete convex function $\varphi$.
A least weakly submajorized element of $\odotZ{Q}$ is
a decreasingly minimal element of $\odotZ{Q}$
(cf., Remark \ref{RMmajdecminincmaxconv}).

%%\marginpar{8/21}

G-polymatroids fit in the framework of DCA, because
the set $\odotZ{Q}$ of integral points of 
an integral g-polymatroid $Q$
is nothing but an {\rm M}$\sp{\natural}$-convex set, 
and accordingly,
the indicator function of $\odotZ{Q}$ is an {\rm M}$\sp{\natural}$-convex function.
An {\rm M}$\sp{\natural}$-convex set
is exactly the projection of an {\rm M}-convex set,
which is a classic result \cite{Fra11book,Fuj05book}
expressed in the language of DCA. 
See \cite{Mdcasiam} for more about {\rm M}$\sp{\natural}$-convexity.

The M-optimality criterion (Theorem~\ref{THmopt}) immediately implies 
the following generalization of Proposition~\ref{PRgroenevelt}.

\begin{proposition}  \label{PRseparlocQ}
Let $Q$ be an integral g-polymatroid and
$\odotZ{Q}$ be the set of its integral elements.
An element $m$ of $\odotZ{Q}$ is a minimizer of  
a separable convex function
$\Phi(x) = \sum [\varphi_{s}(x(s)):  s\in S]$ over $\odotZ{Q}$
if and only if
each of the following three conditions holds:
\begin{itemize}
\item
$\varphi_{s}(m(s)+1) + \varphi_{t}(m(t)-1) \geq \varphi_{s}(m(s)) + \varphi_{t}(m(t))$
whenever 
$m+\chi_{s}-\chi_{t} \in \odotZ{Q}$,

\item
$\varphi_{s}(m(s)+1) \geq \varphi_{s}(m(s))$
whenever 
$m+\chi_{s} \in \odotZ{Q}$, and

\item
$\varphi_{t}(m(t)-1) \geq  \varphi_{t}(m(t))$
whenever 
$m-\chi_{t} \in \odotZ{Q}$.
\finbox
\end{itemize}
\end{proposition}

Proposition~\ref{PRtighteningB}
for a symmetric separable convex function
$\Phi(x) = \sum [\varphi(x(s)):  s\in S]$
on base-polyhedra can be adapted to g-polymatroids
under the additional assumption of monotonicity of $\varphi$.

Let $B$ denote the set of minimal elements of an integral g-polymatroid $Q$,
and $\odotZ{B}$ the set of integral members of $B$.
When $Q$ is defined by a paramodular pair $(p,b)$
of an integer-valued supermodular function $p$ 
and an integer-valued submodular function $p$,
it has a minimal element precisely if $p(S)$ is finite
\cite[Chapter 14]{Fra11book}.
That is, $B$ is nonempty if and only if $p(S)$ is finite. 
If $B \not= \emptyset$, 
$B$ is an integral base-polyhedron and $\odotZ{B}$ is an M-convex set.
Note that
$\odotZ{B} \not= \emptyset$ if and only if $B \not= \emptyset$,
and $\odotZ{B}$ is the set of minimal elements of $\odotZ{Q}$.

\begin{proposition}  \label{PRtighteningQ}
Let $\Phi$ be a symmetric separable convex function represented as
$\Phi(x) = \sum [\varphi(x(s)):  s\in S]$
with monotone nondecreasing discrete convex $\varphi: \ZZ \to \RR \cup \{ +\infty \}$.
There exists a minimizer of $\Phi$ in $\odotZ{Q}$ if and only if $\odotZ{B}$ is nonempty.
An element $m$ of $\odotZ{Q}$ is a minimizer of $\Phi$ over $\odotZ{Q}$ if 
$m$ belongs to $\odotZ{B}$ and
$m(t)\leq m(s)+1$ 
whenever $m+\chi_{s}-\chi_{t}$ is in $\odotZ{B}$.
The converse is also true if $\varphi$ is strictly convex and strictly monotone increasing.
\finbox
\end{proposition}

Let $m$ be an element of $\odotZ{Q}$ that minimizes 
$\Phi(x) = \sum [\varphi(x(s)):  s\in S]$
for an arbitrarily chosen 
strictly convex and strictly monotone increasing $\varphi$.
Then Proposition~\ref{PRtighteningQ} implies
that $m$ is a universal minimizer of all such $\Phi(x)$,
since the condition 
\begin{equation}  \label{Mnat1tighten}
 m+\chi_{s}-\chi_{t} \in \odotZ{B} \  \Rightarrow \  m(t)\leq m(s)+1
\end{equation}
is independent of $\varphi$.
Therefore, $m$ is a least weakly submajorized element of $\odotZ{Q}$.

By adapting the above results to decreasing minimality,
we see that
$\odotZ{Q}$ has a dec-min element if and only if 
$\odotZ{B}$ is nonempty,
and that a member $m$ of
$\odotZ{Q}$ is decreasingly minimal in $\odotZ{Q}$ 
if and only if $m \in \odotZ{B}$ and \eqref{Mnat1tighten} holds,
which is equivalent, by Theorem \RefPartI{3.3} of Part~I, 
to $m$ being a dec-min element of $\odotZ{B}$.

%%\hfill \memo{Major revision ends here (2019-08-20; after arXiv Ver.3)} 

%%% end file %%%

%% murota 2018-08-25 / 2019-05-22 / 2019-09-19

\section{Min-max formulas}
\label{SCminmaxformula}

Key min-max formulas on discrete decreasing minimization,
established by constructive methods in Part~I \cite{FM18part1},
are derived here from the Fenchel-type discrete duality in discrete convex analysis.
These formulas can in fact be derived from a special case of
the Fenchel-type discrete duality 
where a separable convex function is minimized over an M-convex set.
This special case often provides interesting min-max relations in applications
and deserves particular attention.
The (general) Fenchel-type discrete duality is described in Section~\ref{SCfencml}
 and its special case for separable convex functions in Section~\ref{SCfencsepar}.

\subsection{Min-max formulas for decreasing minimization}
\label{SCminmaxformuladecmin}

In this section we treat the formulas
\eqref{minmaxSqSum-KM},
\eqref{(betamin-KM)},
\eqref{r1-KM}, and
\eqref{decmintruncsum0} below.
Recall that 
$p$ is an integer-valued (fully) supermodular function on the ground-set $S$
describing a base-polyhedron $B$
and $\hat p$ is the linear extension (Lov{\'a}sz extension) of $p$,
whose definition is given in \eqref{lovextdef} in Section \ref{SCfencsepar}.

\begin{itemize}
\item
\cite[Theorem \RefPartI{6.6}]{FM18part1} \ 
For the square-sum we have
\begin{equation}
\min \{ \sum_{s \in S} m(s)\sp{2} :  m\in \odotZ{B} \} 
= \max \{\hat p(\pi ) - \sum _{s\in S} 
 \left\lfloor {\pi (s) \over 2}\right\rfloor 
 \left\lceil {\pi (s) \over 2}\right\rceil : 
 \pi \in {\bf Z}\sp{S} \}.  
 \label{minmaxSqSum-KM} 
\end{equation}

\item
\cite[Theorem \RefPartI{4.1}]{FM18part1} \ 
For the largest component $\beta_{1}$ of a
max-minimizer of $\odotZ{B}$, we have
\begin{equation}
 \beta_{1}=\max \{ \left\lceil {p(X) \over \vert X\vert } \right\rceil : 
 \emptyset \not =X\subseteq S\}.
\label{(betamin-KM)} 
\end{equation}
Recall that $\beta_{1}$ is equal to the largest component
of any dec-min element of $\odotZ{B}$.

\item
\cite[Theorem \RefPartI{4.3}]{FM18part1} \ 
 For the minimum number $r_{1}$ of $\beta_{1}$-valued components of 
a $\beta_{1}$-covered member of $\odotZ{B}$, we have
\begin{equation}
 r_{1}= \max \{ p(X) - (\beta_{1}-1)\vert X\vert :  X\subseteq S\} .  
\label{r1-KM} 
\end{equation}
Recall that 
$r_{1} = | \{ s \in S : m(s) = \beta_{1} \} |$
for any dec-min element $m$ of $\odotZ{B}$.
\end{itemize}

Moreover, the following min-max formula will be established in Section \ref{SCtotalexcess}
as a generalization of (\ref{r1-KM}).
We refer to 
$\sum_{s \in S} (m(s) - a)\sp{+}$
in the  minimization  below as the {\bf total $a$-excess} of $m$.

\begin{itemize}
\item
For each integer $a$, we have 
\begin{equation} \label{decmintruncsum0}
 \min \{ \sum_{s \in S} (m(s) - a)\sp{+}  : m \in \odotZ{B} \}  
=  \max \{p(X) - a \vert X\vert :  X\subseteq S\}.
\end{equation}
\end{itemize}
Note that this formula (\ref{decmintruncsum0}) for $a=\beta_{1} - 1$
reduces to the formula (\ref{r1-KM}) for $r_{1}$.
It will be shown in Theorem \ref{THtotalexcess} that
an element of $\odotZ{B}$
is decreasingly minimal if and only if it is a minimizer of 
the left-hand side of (\ref{decmintruncsum0}) universally for all $a \in \ZZ$.
We remark that the minimization problem above is known to be most fundamental 
in the literature of majorization,
whereas the function 
$p(X) - a \vert X\vert$
to be maximized plays the pivotal role in characterizing
the canonical partition and the essential value-sequence
(cf., Section \ref{SCcanopat}).
Thus the min-max formula (\ref{decmintruncsum0}) 
reinforces the link between the present study and the theory of majorization.

\subsection{Fenchel-type discrete duality in discrete convex analysis}
\label{SCfencml}

In this section we describe an important result in DCA,
the Fenchel-type discrete duality theorem, which we use to derive the min-max formulas
related to dec-min elements.
The Fenchel-type discrete duality theorem in DCA originates in Murota \cite{Mstein96}
and is formulated for integer-valued functions in \cite{Mdca98,Mdcasiam}.

For any integer-valued functions
$f: \ZZ\sp{S} \to \ZZ \cup \{ +\infty \}$
and  
$h: \ZZ\sp{S} \to \ZZ \cup \{ -\infty \}$,
we define their (convex and concave) {\bf conjugate functions} by 
\begin{eqnarray}
 f\sp{\bullet}(\pi) 
 &=& \sup\{  \langle \pi, x \rangle - f(x)  :  x \in \ZZ\sp{S} \}
\qquad ( \pi \in \ZZ\sp{S}), 
\label{conjvexZZ} \\
 h\sp{\circ}(\pi) 
 &=& \inf\{  \langle \pi, x \rangle - h(x)  :  x \in \ZZ\sp{S} \}
\qquad ( \pi \in \ZZ\sp{S}),
 \label{conjcavZZ}
\end{eqnarray}
where $\langle \pi, x \rangle$ means the (standard) inner product of vectors $\pi$ and $x$.
Note that both $x$ and $\pi$ are integer vectors.
Since the functions are integer-valued,
the supremum in (\ref{conjvexZZ}) is attained if it is finite-valued.
Similarly for the infimum in (\ref{conjcavZZ}).
Accordingly, we henceforth write ``$\max$'' and ``$\min$'' 
in place of ``$\sup$'' in (\ref{conjvexZZ}) and ``$\inf$'' in (\ref{conjcavZZ}),
respectively.

The Fenchel-type discrete duality is concerned with the relationship between
the minimum of 
$f(x) - h(x)$ over $x \in \ZZ\sp{S}$
and the maximum of 
$h\sp{\circ}(\pi)  - f\sp{\bullet}(\pi)$ over $\pi \in \ZZ\sp{S}$.
By the definition of the conjugate functions
in (\ref{conjvexZZ}) and (\ref{conjcavZZ}) we have 
inequalities (called the Fenchel--Young inequalities)
\begin{eqnarray}
 f(x) + f\sp{\bullet}(\pi) &\geq&  \langle \pi, x \rangle , 
 \label{youngineqf4}
\\ 
h(x) + h\sp{\circ}(\pi) &\leq&  \langle \pi, x \rangle 
 \label{youngineqh4}
\end{eqnarray}
for any $x$ and $\pi$, and hence
\begin{equation} \label{fencweakineq4}
 f(x) - h(x) \geq h\sp{\circ}(\pi) - f\sp{\bullet}(\pi) 
\end{equation}
for any $x$ and $\pi$.
Therefore we have {\bf weak duality}: 
\begin{equation} \label{fencminmaxMweak4}
  \min\{ f(x) - h(x)  :  x \in \ZZ\sp{S}  \}
 \geq   \max\{ h\sp{\circ}(\pi) - f\sp{\bullet}(\pi)  :  \pi \in \ZZ\sp{S} \} .
\end{equation}
It is noted, however, that in this expression using ``$\min$'' and ``$\max$'' 
we do not exclude the possibility of 
the unbounded case where $\min\{ \cdots \}$ and/or $\max\{ \cdots \}$ 
are equal to $-\infty$ or $+\infty$
(we avoid using ``$\inf$'' and ``$\sup$'' for wider audience). 
Here we note the following.
\begin{enumerate}
\item
If 
$\dom f \cap \dom h \not= \emptyset$
and 
$\dom f\sp{\bullet} \cap \dom h\sp{\circ} \not= \emptyset$,
both $\min\{ f(x) - h(x)  :  x \in \ZZ\sp{S}  \}$
and $\max\{ h\sp{\circ}(\pi) - f\sp{\bullet}(\pi)  :  \pi \in \ZZ\sp{S} \}$
are finite integers and the minimum and the maximum are attained 
by some $x$ and $\pi$ since the functions are integer-valued.

\item
If $\dom f \cap \dom h = \emptyset$,
we understand (by convention) that the minimum of $f - h$ is equal to $+\infty$,
that is, $\min\{ f(x) - h(x)  :  x \in \ZZ\sp{S}  \}=+\infty$.

\item
If 
$\dom f\sp{\bullet} \cap \dom h\sp{\circ} = \emptyset$,
we understand (by convention) that the maximum of $h\sp{\circ} - f\sp{\bullet}$ 
is equal to $-\infty$,
that is, $\max\{ h\sp{\circ}(\pi) - f\sp{\bullet}(\pi)  :  \pi \in \ZZ\sp{S} \} =-\infty$.
\end{enumerate}
We say that {\bf strong duality} holds if equality holds in \eqref{fencminmaxMweak4}.

The strong duality does hold for a pair of an M$\sp{\natural}$-convex function $f$
and an M$\sp{\natural}$-concave function $h$,
as the following theorem shows.
This is called the 
{\bf Fenchel-type discrete duality theorem} 
\cite{Mdca98,Mdcasiam}.
To be more precise, we need to assume that
at lease one of the following two conditions is satisfied:
\\ \quad
(i) 
there exists $x$ for which both $f(x)$ and $h(x)$ are finite 
({\bf primal feasibility}, $\dom f \cap \dom h \not= \emptyset$),
\\ \quad 
 (ii) 
there exists $\pi$ for which both
$f\sp{\bullet}(\pi)$ and $h\sp{\circ}(\pi)$ are finite
({\bf dual feasibility}, $\dom f\sp{\bullet} \cap \dom h\sp{\circ} \not= \emptyset$).
\\
Note that these two feasibility conditions, (i) and (ii),
are mutually independent, and there is an example
for which both conditions fail simultaneously
\cite[p.220, Note 8.18]{Mdcasiam}.

\begin{theorem}[Fenchel-type discrete duality theorem \cite{Mdca98,Mdcasiam}]     
\label{THmlfencdual}
Let $f: \ZZ\sp{S} \to \ZZ \cup \{ +\infty \}$
be an integer-valued M$\sp{\natural}$-convex function and 
$h: \ZZ\sp{S} \to \ZZ \cup \{ -\infty \}$
be an integer-valued M$\sp{\natural}$-concave function
such that 
$\dom f \cap \dom h \not= \emptyset$
or
$\dom f\sp{\bullet} \cap \dom h\sp{\circ} \not= \emptyset$.
Then we have
\begin{equation} \label{fencminmaxM}
  \min\{ f(x) - h(x)  :  x \in \ZZ\sp{S}  \}
 =   \max\{ h\sp{\circ}(\pi) - f\sp{\bullet}(\pi)  :  \pi \in \ZZ\sp{S} \} .
\end{equation}
This common value is finite
if and only if
$\dom f \cap \dom h \not= \emptyset$
and
$\dom f\sp{\bullet} \cap \dom h\sp{\circ} \not= \emptyset$,
and then the minimum and the maximum are attained.
\finbox
\end{theorem}

The essential content of the above theorem
may be expressed as follows:
If $f(x) - h(x)$ is bounded from below, then 
$h\sp{\circ}(\pi) - f\sp{\bullet}(\pi)$
is bounded from above, and 
the minimum of $f(x) - h(x)$ and
the maximum of $h\sp{\circ}(\pi) - f\sp{\bullet}(\pi)$ coincide.

\begin{remark} \rm  \label{RMdualityvar}
The Fenchel-type duality theorem is the central 
duality theorem in discrete convex analysis.
The duality phenomenon captured by this theorem
can be formulated in several different, mutually equivalent, forms
including 
the M-separation theorem \cite[Theorem 8.15]{Mdcasiam},
the L-separation theorem \cite[Theorem 8.16]{Mdcasiam}, 
and the M-convex intersection theorem \cite[Theorem 8.17]{Mdcasiam}. 
These duality theorems include a number of 
important results as special cases such as
Edmonds' intersection theorem,
Fujishige's Fenchel-type duality theorem 
\cite[Theorem 6.3]{Fuj05book}
for submodular set functions, 
the discrete separation theorem
%%%Frank's discrete separation theorem 
\cite[Theorem 12.2.1]{Fra11book}
for submodular/supermodular functions, and 
%%Frank's weight splitting theorem 
the weight splitting theorem 
\cite[Theorem 13.2.4]{Fra11book}
for the weighted matroid intersection problem.
See \cite[Section 8.2, Fig.8.2]{Mdcasiam} for this relationship.
\finbox
\end{remark}

\begin{remark} \rm  \label{RMfenccert1}
The Fenchel-type discrete duality theorem offers
an optimality certificate for the minimization problem of $f(x) - h(x)$.
Two cases are to be distinguished.

\begin{enumerate}
\item
If the explicit forms of the conjugate functions 
$f\sp{\bullet}(\pi)$ and $h\sp{\circ}(\pi)$ are known, 
we can easily evaluate the value of 
$h\sp{\circ}(\pi) - f\sp{\bullet}(\pi)$ for any integer vector $\pi$.
Given an integral vector 
$\pi$ as a certificate of optimality
for an allegedly optimal $x$,
we only have to compute the values of 
$f(x) - h(x)$ and $h\sp{\circ}(\pi) - f\sp{\bullet}(\pi)$
and compare the two values (integers) for their equality.
Thus the availability of explicit forms of the conjugate functions
is computationally convenient as well as intuitively appealing.
The min-max formula \eqref{minmaxSqSum-KM} for 
the square-sum minimization over an M-convex set
falls into this case.

\item
Even if explicit forms of the conjugate functions are not available, 
the Fenchel-type discrete duality theorem
offers a computationally efficient (polynomial-time) method
for verifying the optimality
if it is combined with the M-optimality criterion (Theorem \ref{THmopt}). 
We shall discuss this method in Section \ref{SCfencoptsetgen};
see Remark \ref{RMfenccert2}.
\finbox
\end{enumerate}
\end{remark}

The conjugate of an M$\sp{\natural}$-convex function is 
endowed with another kind of discrete convexity, called {\rm L}$\sp{\natural}$-convexity.
A function
$g: \ZZ\sp{S} \to \RR \cup \{ +\infty \}$
with $\dom g \not= \emptyset$
is called {\bf L$\sp{\natural}$-convex} if
it satisfies the inequality
\begin{equation} \label{disfnmidconvex}
 g(\pi) + g(\tau) \geq
   g \left(\left\lceil \frac{\pi+\tau}{2} \right\rceil\right) 
  + g \left(\left\lfloor \frac{\pi+\tau}{2} \right\rfloor\right) 
\qquad (\pi, \tau \in \ZZ\sp{S})   ,
\end{equation}
where, for $z \in \RR$ in general, 
$\left\lceil  z   \right\rceil$ 
denotes the smallest integer not smaller than $z$
(rounding-up to the nearest integer)
and $\left\lfloor  z  \right\rfloor$
the largest integer not larger than $z$
(rounding-down to the nearest integer),
and this operation is extended to a vector
by componentwise applications.
The property (\ref{disfnmidconvex}) is referred to 
as {\bf discrete midpoint convexity}.
A function $g$ is called {\bf L$\sp{\natural}$-concave} 
if $-g$ is {\rm L}$\sp{\natural}$-convex.

The following is a local characterization of global maximality
for L$\sp{\natural}$-concave functions, called the L-optimality criterion (concave version).

\begin{theorem}[{\cite[Theorem 7.14]{Mdcasiam}}] \label{THlopt}
Let $g: \ZZ\sp{S} \to \RR \cup \{ -\infty \}$ be an L$\sp{\natural}$-concave function,
and $\pi\sp{*}  \in \dom g$.
Then $\pi\sp{*} $ is a maximizer of $g$ if and only if
it is locally maximal in the sense that
\begin{align} 
& g(\pi\sp{*}) \geq g(\pi\sp{*} - \chi_{Y})  \quad  \mbox{\rm for all } \ Y \subseteq S ,
\label{lnatfnlocmin1}
\\
& g(\pi\sp{*}) \geq g(\pi\sp{*} + \chi_{Y})  \quad  \mbox{\rm for all } \ Y \subseteq S .
\label{lnatfnlocmin2}
\end{align}
\finbox
\end{theorem}

The reader is referred to 
\cite[Chapter 7]{Mdcasiam} for more properties of L$\sp{\natural}$-convex functions
and \cite[Chapter 8]{Mdcasiam} for the conjugacy between 
M$\sp{\natural}$-convexity and L$\sp{\natural}$-convexity. 
In particular, \cite[Figure 8.1]{Mdcasiam}
offers the whole picture of conjugacy relationship.

\begin{remark} \rm  \label{RMm2convex}
In Theorem~\ref{THmlfencdual}
the functions $f(x)$ and $-h(x)$ are both M$\sp{\natural}$-convex,
but the function $f(x) - h(x)$
to be minimized on the left-hand side of (\ref{fencminmaxM})
is not necessarily M$\sp{\natural}$-convex,
since the sum of M$\sp{\natural}$-convex functions may not be M$\sp{\natural}$-convex.
To see this, consider two M-convex sets $\odotZ{B}_{1}$ and $\odotZ{B}_{2}$ 
associated with integral base-polyhedra
$B_{1}$ and $B_{2}$, respectively, and for $i=1,2$,
let $f_{i}$ be the indicator function of $\odotZ{B}_{i}$
(i.e., $f_{i}(x) = 0$ if $x \in \odotZ{B}_{i}$, and 
$f_{i}(x) = +\infty$ if $x \in \ZZ\sp{S} \setminus \odotZ{B}_{i}$).
The function  $f_{1}+f_{2}$ is the indicator function of 
the set of integer points in the intersection 
$B_{1} \cap B_{2}$, which is not a base-polyhedron in general.
This argument also shows that the left-hand side of (\ref{fencminmaxM})
is a nonlinear generalization of the weighted polymatroid intersection problem;
see \cite[Section 8.2.3]{Mdcasiam} for details.
\finbox
\end{remark} 

\begin{remark} \rm  \label{RMlconvex}
Functions $h\sp{\circ}(\pi)$ and $f\sp{\bullet}(\pi)$ in Theorem~\ref{THmlfencdual}
are L$\sp{\natural}$-concave and L$\sp{\natural}$-convex, respectively.
Since the sum of L$\sp{\natural}$-concave functions is L$\sp{\natural}$-concave,
the function $h\sp{\circ}(\pi) - f\sp{\bullet}(\pi)$
to be maximized on the right-hand side of (\ref{fencminmaxM})
is an L$\sp{\natural}$-concave function.
In contrast, the function $f(x) - h(x)$
to be minimized on the left-hand side of (\ref{fencminmaxM})
is not an M$\sp{\natural}$-convex function, as explained in Remark \ref{RMm2convex} above.
In this sense, the left-hand side (minimization) and the right-hand side
(maximization) are not symmetric.
\finbox
\end{remark}

\subsection{Min-max formula for separable convex functions on an M-convex set}
\label{SCfencsepar}

In this section 
the Fenchel-type discrete duality theorem is tailored to 
the problem of minimizing a separable convex function over an M-convex set.
This special case deserves particular attention
as it is suitable and sufficient for our use in decreasing minimization.

Consider the problem of minimizing an integer-valued separable convex function
\begin{equation} \label{gensepar2}
  \Phi(x) = \sum [\varphi_{s}(x(s)):  s\in S] 
\end{equation}
over an M-convex set $\odotZ{B}$, where each
$\varphi_{s}: \ZZ \to \ZZ \cup \{ +\infty \}$
is an integer-valued discrete convex function in a single integer variable.
This problem is equivalent to minimizing 
$\Phi(x) + \delta(x)$,
where $\delta$ denotes the indicator function of 
$\odotZ{B}$ defined in (\ref{indicBdef}).

In Section \ref{SCseparDCA} we have regarded the function
$\Phi + \delta$
as an M-convex function
and applied the M-optimality criterion 
to derive some results 
obtained in Part~I \cite{FM18part1}. 
In contrast, we are now going to apply the Fenchel-type discrete duality theorem
to the minimization of the function 
$\Phi + \delta = \Phi - (-\delta)$.
In so doing we can 
separate the roles of the constraining M-convex set 
and the objective function $\Phi(x)$ itself.

With the choice
of $f = \Phi$ and $h= -\delta$ in the min-max relation 
$\min\{ f(x) - h(x)  \} =  \max\{ h\sp{\circ}(\pi) - f\sp{\bullet}(\pi)  \}$
in (\ref{fencminmaxM}),
the left-hand side
represents minimization of $\Phi$ over the M-convex set $\odotZ{B}$.
We denote the conjugate function of $\varphi_{s}$ by $\psi_{s}$,
which is a function $\psi_{s}: \ZZ \to \ZZ \cup \{  +\infty \}$ defined by
\begin{equation} \label{phiconjdef}
 \psi_{s}(\ell)  = \max\{ k \ell  -  \varphi_{s}(k)  :  k \in \ZZ \}
\qquad
(\ell \in \ZZ).
\end{equation}
Then the conjugate function of $f$ is given by
\begin{equation}
 f\sp{\bullet}(\pi)  = \sum [\psi_{s}(\pi(s)):  s\in S]
\qquad ( \pi \in \ZZ\sp{S}) .
\label{fconj}
\end{equation}
On the other hand, the conjugate function $h\sp{\circ}$ of $h$ is given by
\begin{equation}
 h\sp{\circ}(\pi) 
= \min\{  \langle \pi, x \rangle +  \delta(x)   : x \in \ZZ\sp{S} \}
=\min\{  \langle \pi, x \rangle  :  x \in \odotZ{B} \}
=\hat p(\pi)
\quad ( \pi \in \ZZ\sp{S})
\label{deltaBconj}
\end{equation}
in terms of 
the linear extension (Lov{\'a}sz extension) $\hat p$ of $p$.
Recall that, for any set function $p$, 
$\hat p$ is defined \cite[Part I, Section~\RefPartI{6.2}]{FM18part1} as
\begin{equation} \label{lovextdef}
\hat p(\pi ) = p(I_n)\pi (s_n) 
 + \sum _{j=1}\sp {n-1} p(I_j)[\pi (s_j)-\pi (s_{j+1})] ,
\end{equation}
where $n = |S|$, the elements of $S$ are indexed
in such a way that
$\pi (s_1)\geq \cdots \geq \pi (s_n)$, 
and $I_j=\{s_1,\dots ,s_j\}$ for $j=1,\dots ,n$.
If $p$ is supermodular, we have
\begin{equation}  \label{plovextmin}
\hat p(\pi) = \min \{ \pi x : x \in \odotZ{B} \} .
\end{equation}

Substituting (\ref{fconj}) and (\ref{deltaBconj}) into (\ref{fencminmaxM})
we obtain (\ref{minmaxgensep}) below.

\begin{theorem} \label{THminmaxgensep}
Assume that 
{\rm (i)} there exists $x \in \odotZ{B}$ such that
$\varphi_{s}(x(s)) < +\infty$ for all $s \in S$
(primal feasibility) or 
{\rm (ii)}  there exists $\pi \in \ZZ\sp{S}$ such that
$\hat p(\pi) > -\infty$ and $\psi_{s}(\pi(s)) < +\infty$ for all $s \in S$
(dual feasibility).
Then we have the min-max relation:
\begin{equation} \label{minmaxgensep}
 \min \{  \sum_{s \in S} \varphi_{s} ( x(s) ) :  x \in \odotZ{B}  \}
 = \max\{ \hat p(\pi) - \sum_{s \in S} \psi_{s}(\pi(s)) :  \pi \in \ZZ\sp{S} \} .
\end{equation}
The unbounded case with both sides 
being equal to $-\infty$ or $+\infty$ is also a possibility.
\finbox
\end{theorem}

Since $\hat p(\pi)$ is an L$\sp{\natural}$-concave function and 
$\sum [\psi_{s}(\pi(s)): s\in S]$ is an L$\sp{\natural}$-convex function,
the function 
$g(\pi) := \hat p(\pi) - \sum [\psi_{s}(\pi(s)): s\in S]$
to be maximized on the right-hand side of (\ref{minmaxgensep})
is an L$\sp{\natural}$-concave function (cf.~Remark \ref{RMlconvex}).
We state this as a proposition for later reference.

\begin{proposition} \label{PRdualobjlnat}
The function $g(\pi) = \hat p(\pi) - \sum [\psi_{s}(\pi(s)): s\in S]$
is L$\sp{\natural}$-concave.
\finbox
\end{proposition}

When specialized to a symmetric function $\Phi$, 
the min-max formula (\ref{minmaxgensep}) is simplified to 
\begin{equation} \label{minmaxsymsep}
  \min\{ \sum_{s\in S} \varphi(x(s))  :  x \in \odotZ{B}  \}
 = \max\{ \hat p(\pi) - \sum_{s\in S} \psi(\pi(s)) :  \pi \in \ZZ\sp{S} \} ,
\end{equation}
where $\varphi: \ZZ \to \ZZ \cup \{ +\infty \}$
is any integer-valued discrete convex function
and $\psi: \ZZ \to \ZZ \cup \{ +\infty \}$
is the conjugate of $\varphi$ defined as
$\psi(\ell)  = \max\{ k \ell  -  \varphi(k)  :  k \in \ZZ \}$
for $\ell \in \ZZ$.
With appropriate choices of $\varphi$ in (\ref{minmaxsymsep})
we shall derive the formulas
(\ref{minmaxSqSum-KM}), (\ref{(betamin-KM)}), and (\ref{r1-KM}).

In applications of 
(\ref{minmaxgensep}) 
(resp., (\ref{minmaxsymsep})) 
with concrete functions 
$\varphi_{s}$ (resp., $\varphi$), it is often the case 
that the conjugate functions 
$\psi_{s}$ (resp., $\psi$)  
can be given explicitly.
This is illustrated in Section \ref{SCexplicitconj}.

\subsection{DCA-based proof of the min-max formula for the square-sum}
\label{SCderminmaxsquaresum}

The min-max formula (\ref{minmaxSqSum-KM}) for the square-sum
can be derived immediately from our duality formula (\ref{minmaxsymsep}).
For $\varphi(k)=k\sp{2}$, the conjugate function $\psi(\ell)$
for $\ell \in \ZZ$ is given explicitly as
\begin{equation}
 \psi(\ell)  
= \max\{  k \ell  - k\sp{2}    :  k \in \ZZ \}
= \max \{  k \ell  - k\sp{2} 
     :  k \in \{   \left\lfloor {\ell}/{2} \right\rfloor ,
                    \left\lceil {\ell}/{2} \right\rceil   \}       \}
= \left\lfloor \frac{\ell}{2} \right\rfloor 
               \cdot \left\lceil \frac{\ell}{2} \right\rceil .
\label{squareconj}
\end{equation}
The substitution of  
(\ref{squareconj}) into 
(\ref{minmaxsymsep}) yields (\ref{minmaxSqSum-KM}).
Note that the primal feasibility is satisfied since
$\varphi(k)$ is finite for all $k$.

\begin{remark} \rm  \label{RMsquaresumdiscover}
In Part~I \cite[Section \RefPartI{6.2}]{FM18part1}
we provided a relatively simple algorithmic proof 
for the min-max formula (\ref{minmaxSqSum-KM}),
which did not use any tool from DCA.  
However, to figure out the min-max formula itself 
without the DCA background seems rather difficult.  
Indeed, the present authors first identified the formula (\ref{minmaxSqSum-KM})
via DCA as above, and then came up with the algorithmic proof.
This example demonstrates the role and effectiveness of DCA. 
\finbox
\end{remark}

\begin{remark} \rm  \label{RMnominmaxdiffsum}
In Part~I \cite[Section \RefPartI{6.1}]{FM18part1}
we have characterized a dec-min element
as a square-sum minimizer and also as a difference-sum minimizer.
Whereas a min-max formula can be obtained by DCA for the square-sum,
this is not the case with the difference-sum.
This is because the difference-sum is not M$\sp{\natural}$-convex 
(though it is  L-convex),
and therefore difference-sum minimization over an M-convex set
does not fit into the framework of the Fenchel-type discrete duality in DCA.
\finbox
\end{remark}

\subsection{DCA-based proof of the formula for $\beta_{1}$}
\label{SCderminmaxbeta1}

The formula (\ref{(betamin-KM)}) for the largest component $\beta_{1}$ of a
max-minimizer of $\odotZ{B}$
can also be derived from our duality formula (\ref{minmaxsymsep}).
With an integer parameter $\alpha$ we choose
\[
\varphi(k) = 
   \left\{  \begin{array}{ll}
    0         &   (k \leq \alpha) ,     \\
   +\infty    &   (k \geq \alpha + 1)   \\
                     \end{array}  \right.
\]
in (\ref{minmaxsymsep}). 
By the definition of $\beta_{1}$, the left-hand side of (\ref{minmaxsymsep}) is 
equal to zero if $\alpha \geq \beta_{1}$,
and equal to $+\infty$ if $\alpha \leq \beta_{1} -1$.
Hence $\beta_{1}$ is equal to the minimum of $\alpha$ for which the left-hand side is equal to zero.

The conjugate function $\psi$ of $\varphi$ is given by
\begin{equation}
\psi(\ell) = 
\max\{  k \ell  :  k \leq \alpha \} =
   \left\{  \begin{array}{ll}
   +\infty     &   (\ell \leq - 1) ,     \\
    0         &   (\ell = 0) ,     \\
    \alpha \ell          &   (\ell \geq 1) .     \\
                     \end{array}  \right.
\label{beta1conj}
\end{equation}
Both $\hat p(\pi)$ and $\psi(\ell)$ are positively homogeneous
(i.e., $\hat p(\lambda \pi) = \lambda \hat p(\pi)$
and
$\psi(\lambda \ell)= \lambda \psi(\ell)$
for nonnegative integers $\lambda$).
This implies, in particular, that the maximization problem on 
the right-hand side of (\ref{minmaxsymsep})
is feasible for all $\alpha$ and hence  
the identity (\ref{minmaxsymsep}) holds,
which reads either  $0 = 0$ or $+\infty= +\infty$. 
Since $\beta_{1}$ is the minimum of $\alpha$ for which the left-hand side is equal to $0$,
we can say that
$\beta_{1}$ is the minimum of $\alpha$ for which the right-hand side is equal to $0$.

Finally, we consider the condition that ensures $\pi\sp{*} = {\bf 0}$ to be a maximizer of 
the function $g(\pi):= \hat p (\pi) -   \sum_{s \in S}  \psi(\pi(s))$.
By the L$\sp{\natural}$-concavity of this function
we can make use of 
Theorem~\ref{THlopt} (L-optimality criterion).
The first condition (\ref{lnatfnlocmin1})
in Theorem~\ref{THlopt} 
is satisfied trivially by (\ref{beta1conj}), 
whereas the second condition (\ref{lnatfnlocmin2}) reads
$g(\pi\sp{*} + \chi_{Y}) = p(Y) -\alpha  |Y| \leq 0$.
Therefore, the right-hand side of (\ref{minmaxsymsep}) is equal to zero
if and only if $\max \{p(Y) - \alpha | Y | :  Y \subseteq S\} = 0$,
from which follows the formula
(\ref{(betamin-KM)}).

\subsection{DCA-based proof of the formula for $r_{1}$}
\label{SCderminmaxr1}

The formula (\ref{r1-KM}) for 
the minimum number $r_{1}$ of $\beta_{1}$-valued components of 
a $\beta_{1}$-covered member of $\odotZ{B}$
can also be derived from our duality formula (\ref{minmaxsymsep}).
We choose  
\begin{equation} \label{phir1}
\varphi(k) = 
   \left\{  \begin{array}{ll}
    0         &   (k \leq \beta_{1} - 1) ,     \\
    1         &   (k = \beta_{1}) ,     \\
   +\infty    &   (k \geq \beta_{1} + 1) ,  \\
                     \end{array}  \right.
\end{equation}
whose graph is given by the left of Fig.~\ref{FGphipsi}.
By the definitions of $\beta_{1}$ and $r_{1}$, 
the minimum in (\ref{minmaxsymsep}) is equal to $r_{1}$.
In particular, the primal problem is feasible, and hence the identity 
(\ref{minmaxsymsep}) holds.

%%%  FIGURE %%%%%%%%%%%%%%%%%%
%% \input{fgphipsiKMred}
\begin{figure}\begin{center}
\includegraphics[height=40mm]{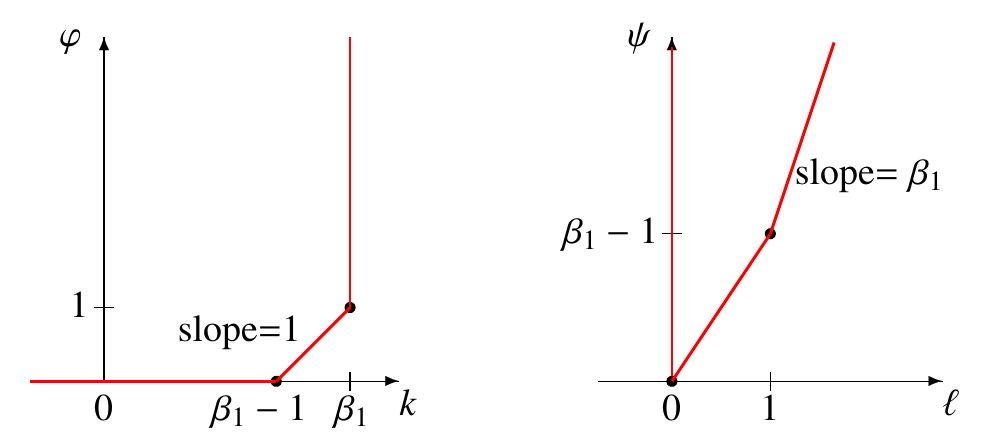}
%%\includegraphics[height=40mm]{fgphipsiRed2.pdf}
%% 2020-06-30

\caption{Mutually conjugate discrete convex functions $\varphi$ and $\psi$ 
 in (\ref{phir1}) and (\ref{psir1})}
\label{FGphipsi}
\end{center}\end{figure}
%%%  FIGURE %%%%%%%%%%%%%%%%%%

The conjugate function $\psi$ of $\varphi$ is given by
\begin{eqnarray}
\psi(\ell) &=&
\max \big\{ \  
 \max\{  k \ell  :  k \leq \beta_{1} -1  \},  \ \  \beta_{1} \ell -1 
    \  \big\} 
\nonumber \\
&=&
   \left\{  \begin{array}{ll}
   +\infty     &   (\ell \leq - 1) ,     \\
    0         &    (\ell = 0) ,     \\
    \beta_{1} \ell  -1         &   (\ell \geq 1) ,     \\
                     \end{array}  \right.
\label{psir1}
\end{eqnarray}
whose graph is given by the right of Fig.~\ref{FGphipsi}.
In considering the maximum of
$g(\pi):=\hat p (\pi) -   \sum_{s \in S}  \psi(\pi(s))$
over $\pi \in \ZZ\sp{S}$,
we may restrict $\pi$ to  $\{ 0,1 \}$-vectors,
as shown in Lemma~\ref{LM01dual} below.
For $\pi = \chi_{X} \in \{ 0, 1 \}\sp{S}$ with $X \subseteq S$, we have
 $\hat p(\pi) = \hat p(\chi_{X}) = p(X)$ and 
$\sum_{s \in S}  \psi(\pi(s))  = \sum_{s \in S}  \psi(\chi_{X}(s)) 
= \sum_{s \in X} \psi(1) = (\beta_{1} -1) |X|$,
and therefore,
the right-hand side of (\ref{minmaxsymsep}) is equal to
$\max \{p(X) - (\beta_{1} -1)\vert X\vert :  X\subseteq S\}$.
Thus the formula (\ref{r1-KM}) is derived.

\begin{lemma} \label{LM01dual}
There exists a $\{0,1 \}$-vector $\pi$ that attains the maximum of $g(\pi)$ over $\pi \in \ZZ\sp{S}$.
\end{lemma}
\begin{proof}
Note first that $g$ is an L$\sp{\natural}$-concave function,
and define $a = \beta_{1} - 1$.
Let $A \subseteq S$ be a maximizer of 
$p(X) - a \vert X\vert$ over all subsets of $S$,
and $\pi\sp{*} = \chi_{A}$.
Then $g(\pi\sp{*}) =  p(A) - a \vert A \vert $.
We will show that the conditions 
(\ref{lnatfnlocmin1}) and (\ref{lnatfnlocmin2}) in the L-optimality criterion (Theorem~\ref{THlopt})
are satisfied.

Proof of 
$g(\pi\sp{*}) \geq g(\pi\sp{*} - \chi_{Y})$
in (\ref{lnatfnlocmin1}):
We may assume $Y \subseteq A$, since, otherwise,
$\pi\sp{*} - \chi_{Y} \not\in \dom g$
by (\ref{psir1}).
If $Y \subseteq A$, we have
$\pi\sp{*} - \chi_{Y} = \chi_{A \setminus Y} = \chi_{Z}$,
where $Z = A \setminus Y$.
Hence,
\[
g(\pi\sp{*} - \chi_{Y})  = g(\chi_{Z}) = p(Z) - a \vert Z \vert
\leq  p(A) - a \vert A \vert = g(\pi\sp{*}).
\]

Proof of 
$g(\pi\sp{*}) \geq g(\pi\sp{*} + \chi_{Y})$
in (\ref{lnatfnlocmin2}):
Since 
\[
(\pi\sp{*} + \chi_{Y})(s) = (\chi_{A} + \chi_{Y})(s) =
   \left\{  \begin{array}{ll}
   2      &   (s \in A \cap Y) ,     \\
   1         &  (s \in (A \cup Y) \setminus (A \cap Y)) ,     \\
   0        &   (s \in S \setminus (A \cup Y)) ,     \\
                     \end{array}  \right.
\]
we have
\begin{align*}  
 \hat p(\pi\sp{*} + \chi_{Y})
&=  p( A \cap Y ) +  p( A \cup Y ),
\\ 
 \sum_{s\in S} \psi((\pi\sp{*} + \chi_{Y})(s))
& =
(2 \beta_{1} -1) |A \cap Y| + (\beta_{1} -1) |(A \cup Y) \setminus (A \cap Y)|
=
 \beta_{1} |A \cap Y| + a |A \cup Y| 
\end{align*}
by the definition \eqref{lovextdef} of $\hat p$ and 
the expression \eqref{psir1} of the conjugate function $\psi$.
Hence
\begin{align*} 
g(\pi\sp{*} + \chi_{Y}) 
 &= 
 \big( p( A \cap Y )  + p( A \cup Y ) \big)
   - \big( \beta_{1} |A \cap Y| + a |A \cup Y| \big)
\\
 &= 
 \big( p( A \cap Y )  - \beta_{1} |A \cap Y| \big)
  + \big( p( A \cup Y ) - a |A \cup Y| \big).
\end{align*}
Here we have
\begin{align*}
& p( A \cap Y )-  \beta_{1} |A \cap Y| \leq  0,  
\\ &
p( A \cup Y ) -  a |A \cup Y|  \leq  p( A ) - a |A| = g(\pi\sp{*}) ,
\end{align*}
since $(\beta_{1}, \beta_{1}, \ldots, \beta_{1})$ belongs 
to the supermodular polyhedra defined by $p$,
 $A$ is a maximizer of $p(X) - a \vert X\vert$, and 
$p( A ) - a |A| = g(\pi\sp{*})$.
Therefore,
$g(\pi\sp{*} + \chi_{Y})  \leq  g(\pi\sp{*})$.
\end{proof}

\subsection{Total $a$-excess and decreasing minimality}
\label{SCtotalexcess}

In this section, we explore a link between decreasing minimality and
the total $a$-excess announced 
at the beginning of Section \ref{SCminmaxformula}.
The minimization problem in 
\eqref{decmintruncsum0} (or \eqref{decmintruncsum} below) 
is most fundamental in the literature of majorization.  
Indeed, a least majorized element is characterized as a universal minimizer 
for all $a \in \ZZ$ (Proposition~\ref{PRmajorchar}).
On the other hand, the function 
$p(X) - a \vert X\vert$
to be maximized plays the pivotal role in characterizing
the canonical partition and the essential value-sequence
(cf., Section \ref{SCcanopat}).

As a preparation, we recall (\cite{Fra11book}, \cite{Sch03})
that, for a nonnegative and (fully) supermodular function $p_{0}$, 
the polyhedron 
$C= \{x:  \widetilde x\geq p_{0}\}$ 
is called a contra-polymatroid.  
Note that the nonnegativity and supermodularity
of $p_{0}$ imply that $p_{0}$ is monotone non-decreasing and 
that $C\subseteq \RR_{+}\sp{S}$,
that is, every member of $C$ is a nonnegative vector.
When $p_{0}$ is integer-valued, $C$ is an integer polyhedron.
The corresponding version of
Edmonds' greedy algorithm for polymatroids implies that $C$ 
uniquely determines $p_{0}$, namely,
\begin{equation} 
\label{(pmin)} 
 p_{0}(X)= \min \{ \ \widetilde z(X):  z\in C\}.  
\end{equation}
It is known that,
for a supermodular function $p_{1}$ with possibly negative values, 
the polyhedron
\begin{equation} 
C(p_{1}) := \{x:  x\geq \bm{0}, \  \widetilde x\geq p_{1}\} 
\label{(cpdef)}
\end{equation}
is a contra-polymatroid%
\footnote{%%%%
In the literature,  \eqref{(cpdef)} is used sometimes as the definition of a contra-polymatroid.
}, %% footnote %%%%
 for which the unique
nonnegative supermodular bounding function $p_{0}$ is given by
\begin{equation} 
p_{0}(X)= \max \{p_{1}(Y) :  Y\subseteq X\}.  
\label{(pp0)} 
\end{equation}
It follows from \eqref{(pmin)}, \eqref{(pp0)}, 
and the integrality of the polyhedron $C(p_{1})$ that
\begin{equation} 
 \min \{ \ \widetilde z(S) :  z \in \odotZ{C}(p_{1})\} 
= \max \{p_{1}(X) : X\subseteq S \},
\label{(minmaxp1)} 
\end{equation}
where $\odotZ{C}(p_{1})$ denotes the set of the integral members of $C(p_{1})$.

\begin{lemma}  \label{LMoldresult.2} 
Let $B=B'(p)$ be an (integral) base-polyhedron defined by 
an integer-valued supermodular function $p$.  
For a vector $g:S\rightarrow \ZZ$,
\begin{equation} 
\min \{ \sum [ (m(s) - g(s))\sp + :  s\in S] :  m\in \odotZ{B} \}
 = \max \{p(X) - \widetilde g(X):  X\subseteq S\}.  
\label{(gminmax)}
\end{equation}
\end{lemma}
\begin{proof}
It is known (for example, from the discrete separation theorem 
for submodular set functions
or from a version of Edmonds' polymatroid intersection
theorem) that, for a function $g':S\rightarrow \ZZ$, there is an
element $m\in \odotZ{B}$ for which $m\leq g'$ 
if and only if 
$p\leq \widetilde g'$.  
Therefore the minimization problem on the left-hand
side of \eqref{(gminmax)} is equivalent to finding a lowest lifting
$g':=g+ z$ 
of $g$ with $z\geq \bm{0}$ 
such that $p \leq \widetilde g'$.  That
is, the minimum on the left-hand side of \eqref{(gminmax)} is equal to
$\min \{ \ \widetilde z(S) : 
 z\geq \bm{0}, \  \widetilde z \geq p - \widetilde g \ \}$.  
By applying \eqref{(minmaxp1)} to $p_{1}:= p - \widetilde g$, 
we obtain that this latter minimum is indeed equal to the right-hand side
of \eqref{(gminmax)}. 
\end{proof}

\medskip

The following theorem reinforces the link 
between the present study and the theory of majorization.

\begin{theorem} \label{THtotalexcess}
Let $B$ be a base-polyhedron described 
by an integer-valued supermodular function $p$ 
and $\odotZ{B}$ the set of integral elements of $B$.  
For each integer $a$, we have the following min-max relation 
for the minimum of the total $a$-excess of the members of $\odotZ{B}$:
\begin{equation} \label{decmintruncsum}
 \min \{ \sum_{s \in S} (m(s) - a)\sp{+}  : m \in \odotZ{B} \}  
=  \max \{p(X) - a \vert X\vert :  X\subseteq S\}.
\end{equation}
Moreover, an element of $\odotZ{B}$ is a dec-min element of $\odotZ{B}$ 
if and only if it is a minimizer on the left-hand side 
for every $a \in \ZZ$.
\end{theorem}
\begin{proof}
The min-max formula (\ref{decmintruncsum}) 
follows from Lemma \ref{LMoldresult.2} 
as it is a special case of \eqref{(gminmax)} when $g=(a,a,\dots ,a)$.
Theorem~\ref{THdecminPhistrict2} shows that any dec-min element of $\odotZ{B}$ 
is a minimizer in (\ref{decmintruncsum}) for every $a \in \ZZ$.
The converse is also true, since
$\sum [(x(s) - a)\sp{+}: s \in S] = \sum [(y(s) - a)\sp{+}: s \in S]$
for every $a \in \ZZ$ implies $x{\downarrow}= y{\downarrow}$.
Therefore, an element of $\odotZ{B}$ is dec-min 
if and only if  
it is a universal minimizer for every $a \in \ZZ$.
\end{proof}

The established  formula (\ref{decmintruncsum})  generalizes 
the formula (\ref{r1-KM}) for $r_{1}$.
Indeed, the total $a$-excess 
for $a=\beta_{1} - 1$ is given as 
\[
\sum_{s \in S} (m(s) - a)\sp{+} = \sum_{s \in S} (m(s) - (\beta_{1} - 1))\sp{+} =
  | \{ s \in S  : m(s) = \beta_{1} \} | = r_{1} 
\]
for any dec-min element $m$ of $\odotZ{B}$.

For any dec-min element $m$ of $\odotZ{B}$
and for  $k=\beta_{1}, \beta_{1}-1, \beta_{1}-2, \ldots$,
let $\Theta(m,k)$ denote 
the number of components of $m$ whose value are equal to $k$,
that is,
\[
  \Theta(m,k)= | \{ s \in S : m(s) = k \} | .
\]
Note that $\Theta(m,\beta_{1}) = r_{1}$ and
$\Theta(m,k)$ does not depend on the choice of $m$. 
Since
\[
\sum_{s \in S} (m(s) - (\beta_{1} -i-1) )\sp{+} 
=\sum_{j=0}\sp{i} (i+1-j) \, \Theta(m,\beta_{1}-j)
\qquad (i=0,1,2,\ldots),
\]
the formula (\ref{decmintruncsum}) implies
\begin{equation}
 \sum_{j=0}\sp{i} (i-j+1) \, \Theta(m,\beta_{1}-j) 
=  \max \{p(X) - (\beta_{1} -i-1)\vert X\vert :  X\subseteq S\} 
 \label{thetabetaieqn2}
\qquad (i=0,1,2,\ldots).
\end{equation}
This formula gives a recurrence formula for
$\Theta(m,\beta_{1}), \Theta(m,\beta_{1}-1), \Theta(m,\beta_{1}-2), \ldots$ as
\begin{align}
\Theta(m,\beta_{1}) &=  \max \{p(X) - (\beta_{1} -1)\vert X\vert :  X\subseteq S\},
\notag \\
\Theta(m,\beta_{1} -1)  &=  \max \{p(X) - (\beta_{1} -2)\vert X\vert :  X\subseteq S\} 
   - 2 \, \Theta(m,\beta_{1}) ,
\label{thetamkrec}
\\
\Theta(m,\beta_{1} -2) &=  \max \{p(X) - (\beta_{1} -3)\vert X\vert :  X\subseteq S\} 
  - 3 \, \Theta(m,\beta_{1})  - 2 \, \Theta(m,\beta_{1} -1),
\notag \\
& \quad \cdots\cdots\cdots\cdots\cdots\cdots
\notag
\end{align}

\begin{remark} \rm  \label{RMdcaprftotalex}
A DCA-based proof of the formula (\ref{decmintruncsum}) is as follows.
In (\ref{minmaxsymsep}) we choose 
\begin{equation*}  %% \label{phirgen}
\varphi(k) = (k-a)\sp{+} =
   \left\{  \begin{array}{ll}
    0         &   (k \leq a) ,     \\
    k - a        &   (k \geq a + 1) .     \\
                     \end{array}  \right.
\end{equation*}
The left-hand side of (\ref{minmaxsymsep})
coincides with that of (\ref{decmintruncsum}).
The conjugate function $\psi$ is given by
\begin{equation*}   %% \label{psirgen}
\psi(\ell) =
   \left\{  \begin{array}{ll}
    0         &   (\ell = 0) ,     \\
    a        &   (\ell = 1) ,     \\
   +\infty     &   (\ell \not\in  \{ 0,1  \}) .     \\
                     \end{array}  \right.
\end{equation*}
Therefore,
we may restrict $\pi$ to  $\{ 0,1 \}$-vectors
in considering the maximum of
$g(\pi):=\hat p (\pi) -   \sum_{s \in S}  \psi(\pi(s))$
over $\pi \in \ZZ\sp{S}$.
For $\pi = \chi_{X} \in \{ 0, 1 \}\sp{S}$ with $X \subseteq S$, we have
 $\hat p(\pi) = \hat p(\chi_{X}) =  p(X)$ and 
$\sum_{s \in S}  \psi(\pi(s)) = \sum_{s \in S}  \psi(\chi_{X}(s)) =  \sum_{s \in X} \psi(1) = a |X|$,
and therefore,
the right-hand side of (\ref{minmaxsymsep}) is equal to
$\max \{p(X) - a \vert X\vert :  X\subseteq S\}$.
Thus the formula (\ref{decmintruncsum}) is derived. 
\finbox
\end{remark} 

%%% end file %%%

%% murota 2018-08-25 / 2019-05-22  / 2019-07-09

\section{Structure of optimal solutions to square-sum minimization}
\label{SCsetoptsols}

In this section we offer the DCA view
on the structure of optimal solutions of the min-max formula:
\begin{equation}  
\min \{ \sum [m(s)\sp{2}:  s\in S]:  m\in \odotZ{B} \} 
= \max \{\hat p(\pi ) - \sum _{s\in S} 
 \left\lfloor {\pi (s) \over 2}\right\rfloor 
 \left\lceil {\pi (s) \over 2}\right\rceil : 
 \pi \in {\bf Z}\sp{S} \} ,  
\label{minmaxSqSum-KM2} 
\end{equation}
to which a DCA-based proof has been given in Section \ref{SCderminmaxsquaresum}.

Concerning the optimal solutions to (\ref{minmaxSqSum-KM2}) 
the following results were obtained in  Part~I \cite{FM18part1}. 
Recall that
$\beta_{1}>\beta_{2}>\cdots >\beta_{q}$ denotes the essential value-sequence,
$C_{1} \subset  C_{2} \subset  \cdots \subset  C_{q}$
is the canonical chain, 
$\{S_{1},S_{2},\dots ,S_{q}\}$
is the canonical partition
($S_{i}= C_{i} - C_{i-1}$ and $C_{0}=\emptyset$), 
$\pi\sp{*}$ and $\Delta\sp{*}$ are integral vectors defined by
\[
\pi \sp{*}(s) =2\beta_{i}-1, \quad
\Delta \sp{*}(s) =\beta_{i}-1
\qquad (s\in S_i; \ i=1,2, \dots ,q) ,
\]
and $M\sp{*}$ denotes the direct sum of matroids
$M_{1}, M_{2}, \ldots, M_{q}$ constructed 
in Section \RefPartI{5.3} of Part~I \cite{FM18part1}.

\begin{proposition}[{\cite[Corollary \RefPartI{6.15}]{FM18part1}}] \label{PRoptdualsetlnatKM}
The set $\Pi$ of dual optimal integral vectors
$\pi$ in {\rm (\ref{minmaxSqSum-KM2})} 
is an {\rm L}$\sp {\natural}$-convex set.
The unique smallest element of $\Pi$ is $\pi\sp{*}$.
\finbox
\end{proposition}

\begin{theorem}[{\cite[Theorem \RefPartI{6.13}]{FM18part1}}] \label{THoptdualsetKM}
An integral vector $\pi$ 
is a dual optimal solution in 
{\rm (\ref{minmaxSqSum-KM2})} 
if and only if the following three conditions hold for each $i=1,2,\dots ,q:$
\begin{eqnarray} 
&&\hbox{\rm $\pi (s)=2\beta_{i}-1$ \ for every $s\in S_i-F_i,$}\ 
\label{Si-Fi-KM} 
\\&& 
\hbox{\rm $2\beta_{i}-1\leq \pi (s) \leq 2\beta_{i}+1$ \ for every $s\in F_i$,}\ 
\label{Fi-KM}
\\&& 
\hbox{\rm $\pi(s)-\pi (t) \geq 0$ \ whenever $s,t\in F_i$ and $(s,t) \in A_i$,  }\
\label{pispit-KM} 
\end{eqnarray} 
where 
$F_{i}$ is the largest member of
${\cal F}_{i} = \{ X \subseteq S_{i}: \beta_{i} |X| = p(C_{i-1} \cup X) - p(C_{i-1}) \}$
and $A_{i}$ is the set of pairs $(s,t)$ 
such that $s, t \in F_{i}$ and there is 
no set in ${\cal F}_{i}$ which contains $t$ and not $s$.
\finbox
\end{theorem}

\begin{theorem}[{\cite[Theorem \RefPartI{5.7}]{FM18part1}}] \label{THmatroideltoltKM}
The set of dec-min elements of $\odotZ{B}$ is a matroidal M-convex set.%
\footnote{%%%%%%%%%%%%%%%%%%
In Part I, we have defined a {\bf matroidal M-convex set} as 
the set of integral elements of a translated matroid base-polyhedron.
In other words, a matroidal M-convex set is an M-convex set
in which the $\ell_{\infty}$-distance of any two distinct members is equal to one.
} %%%%% footnote %%%%%%%%%%%%%%
More precisely, an element $m$ of $\odotZ{B}$ is decreasingly minimal
if and only if $m$ can be obtained in the form $m=\chi_L+ \Delta \sp{*}$, 
where $L$ is a basis of the matroid $M\sp{*}$.  
\finbox
\end{theorem}

The objective of this section is to shed the light of DCA 
on these results.
It will turn out that the general results in DCA
capture the structural essence of the above statements,
but do not provide the full statements with specific details.
We first present a summary of the relevant results from DCA 
in Sections \ref{SCfencoptsetgen} and \ref{SCfencoptsetsepar}.

\subsection{General results on the optimal solutions in the Fenchel-type discrete duality}
\label{SCfencoptsetgen}

We summarize the fundamental facts 
about the optimal solutions in the Fenchel-type min-max relation
\begin{equation} \label{fencminmaxM2}
  \min\{ f(x) - h(x)  :  x \in \ZZ\sp{S}  \}
 =   \max\{ h\sp{\circ}(\pi) - f\sp{\bullet}(\pi)  :  \pi \in \ZZ\sp{S} \} ,
\end{equation}
where $f$ is  an integer-valued M$\sp{\natural}$-convex function and
$h$ is an integer-valued M$\sp{\natural}$-concave function.
We assume that both
\ $\dom f \cap \dom h$ \ 
and 
\ $\dom f\sp{\bullet} \cap \dom h\sp{\circ}$ \ 
are nonempty, 
in which case the common value in \eqref{fencminmaxM2} is finite.
We denote the set of the minimizers by $\mathcal{P}$ and the set of the maximizers by $\mathcal{D}$.

To derive the optimality criteria we recall 
the Fenchel--Young inequalities
\begin{align}
 f(x) + f\sp{\bullet}(\pi) &\geq  \langle \pi, x \rangle , 
 \label{youngineqf}
\\ 
h(x) + h\sp{\circ}(\pi) &\leq  \langle \pi, x \rangle ,
 \label{youngineqh}
\end{align}
which hold for any $x \in \ZZ\sp{S}$ and $\pi \in \ZZ\sp{S}$.
These inequalities immediately imply the weak duality
\begin{equation} \label{fencminmaxMweak}
 f(x) - h(x) \geq h\sp{\circ}(\pi) - f\sp{\bullet}(\pi) .
\end{equation}
The inequality in  (\ref{fencminmaxMweak}) turns into an equality 
 if and only if
the inequalities in (\ref{youngineqf}) and (\ref{youngineqh}) are satisfied in equalities.
The former condition is equivalent to saying that 
$x \in \mathcal{P}$ and $\pi \in \mathcal{D}$.
The equality condition for (\ref{youngineqf}) can be rewritten as
\begin{align} 
 f(x) - \langle \pi, x \rangle  =  - f\sp{\bullet}(\pi) 
= - \max\{  \langle \pi, y \rangle - f(y)  :  y \in \ZZ\sp{S} \}
=  \min\{  f(y) - \langle \pi, y \rangle :  y \in \ZZ\sp{S} \} .
\label{argminfp}
\end{align}
Similarly,
the equality condition for (\ref{youngineqh}) can be rewritten as
\begin{align}
 h(x) - \langle \pi, x \rangle  =  - h\sp{\circ}(\pi) 
= - \min\{  \langle \pi, y \rangle - h(y)  :  y \in \ZZ\sp{S} \}
=  \max\{  h(y) - \langle \pi, y \rangle :  y \in \ZZ\sp{S} \} .
\label{argmaxhp}
\end{align}
Therefore we have
\begin{equation} \label{fencoptset1arg}
x \in \mathcal{P} \ \mbox{\rm and} \ \pi \in \mathcal{D}
\iff
x \in \argmin_{y} \{ f(y) - \langle \pi, y \rangle \} \cap
      \argmax_{y} \{ h(y) - \langle \pi, y \rangle \}.
\end{equation}
Furthermore, 
by the M-optimality criterion (Theorem~\ref{THmopt})
applied to $f(y) - \langle \pi, y \rangle$, we have
$x \in \argmin \{ f(y) - \langle \pi, y \rangle \}$
if and only if
\begin{align*}
    f(x) - \langle \pi, x \rangle & \leq
    f(x + \chi_{s} - \chi_{t}) - \langle \pi, x  + \chi_{s} - \chi_{t} \rangle  
 \qquad (\forall s, t \in S)  ,
\nonumber \\ 
    f(x) - \langle \pi, x \rangle & \leq
    f(x + \chi_{s}) - \langle \pi, x  + \chi_{s} \rangle  
 \qquad (\forall s \in S)  ,
\nonumber \\ 
    f(x) - \langle \pi, x \rangle & \leq
    f(x  - \chi_{t}) - \langle \pi, x  - \chi_{t} \rangle  
 \qquad  (\forall t \in S) , 
\end{align*}
that is, if and only if
\begin{align}
 \pi(s) -  \pi(t) & \leq  f(x + \chi_{s} - \chi_{t}) - f(x)
\qquad (\forall s, t \in S)  ,
\label{locoptfpi1}
 \\ 
   f(x) -  f(x - \chi_{s}) & \leq \pi(s) \leq f(x + \chi_{s}) - f(x)
 \qquad  (\forall s \in S)    .
\label{locoptfpi2}
\end{align}
Similarly, we have
$x \in \argmax \{ h(y) - \langle \pi, y \rangle \}$
if and only if
\begin{align}
 \pi(s) -  \pi(t) & \geq  h(x + \chi_{s} - \chi_{t}) - h(x)
\qquad (\forall s, t \in S)  ,
\label{locopthpi1}
 \\ 
   h(x) -  h(x - \chi_{s}) & \geq \pi(s) \geq h(x + \chi_{s}) - h(x)
 \qquad  (\forall s \in S)    .
\label{locopthpi2}
\end{align}
Therefore,
\begin{equation} \label{fencoptset1ineq}
x \in \mathcal{P} \ \mbox{\rm and} \ \pi \in \mathcal{D}
\iff
\mbox{\rm 
\eqref{locoptfpi1},
\eqref{locoptfpi2},
\eqref{locopthpi1}, 
\eqref{locopthpi2}
hold}.
\end{equation}

Using the integer biconjugacy $f\sp{\bullet\bullet}=f$ and  $h\sp{\circ\circ}=h$
for  M$\sp{\natural}$-convex/concave functions
with respect to the discrete conjugates in \eqref{conjvexZZ} and \eqref{conjcavZZ}
 (cf.~\cite[Theorem 8.12]{Mdcasiam}),
we can rewrite \eqref{argminfp} and \eqref{argmaxhp}, respectively, as
\begin{align*} 
 f\sp{\bullet}(\pi)  - \langle \pi, x \rangle  
 & =  -f(x) =  - f\sp{\bullet\bullet}(x)
=  \min\{  f\sp{\bullet}(\tau) - \langle \tau, x \rangle :  \tau \in \ZZ\sp{S} \} ,
\\
 h\sp{\circ}(\pi)  - \langle \pi, x \rangle  
 & =  -h(x) =  - h\sp{\circ\circ}(x)
=  \max\{  h\sp{\circ}(\tau) - \langle \tau, x \rangle :  \tau \in \ZZ\sp{S} \} .
\end{align*}
Hence the equivalence in \eqref{fencoptset1arg} can be rephrased in terms of the 
conjugate functions as
\begin{equation} \label{fencoptset2arg}
x \in \mathcal{P} \ \mbox{\rm and} \ \pi \in \mathcal{D}
\iff
\pi \in \argmin_{\tau} \{ f\sp{\bullet}(\tau) - \langle \tau,  x \rangle \} \cap
        \argmax_{\tau} \{ h\sp{\circ}(\tau) - \langle \tau, x \rangle \} .
\end{equation}
Furthermore, 
by the L-optimality criterion (Theorem~\ref{THlopt})
applied to the {\rm L}$\sp{\natural}$-convex function 
$f\sp{\bullet}(\tau) - \langle \tau,  x \rangle$, 
we have
$\pi \in \argmin \{ f\sp{\bullet}(\tau) - \langle \tau,  x \rangle \}$
if and only if
\begin{align*}
       f\sp{\bullet}(\pi) - \langle \pi, x \rangle  & \leq
       f\sp{\bullet}(\pi + \chi_{Y}) - \langle \pi + \chi_{Y}, x \rangle 
 \qquad  (\forall Y \subseteq S) ,
\nonumber \\ 
       f\sp{\bullet}(\pi) - \langle \pi, x \rangle & \leq
       f\sp{\bullet}(\pi - \chi_{Y}) - \langle \pi - \chi_{Y}, x \rangle 
 \qquad  (\forall Y \subseteq S) ,
\end{align*}
that is, if and only if
\begin{align}
 f\sp{\bullet}(\pi) - f\sp{\bullet}(\pi- \chi_{Y}) 
 \leq  \sum_{s \in Y}x(s) 
\leq    f\sp{\bullet}(\pi+ \chi_{Y}) - f\sp{\bullet}(\pi) 
 \qquad  (\forall Y \subseteq S) .
\label{locoptfconjx}
\end{align}
Similarly, we have
$\pi \in \argmax \{ h\sp{\circ}(\tau) - \langle \tau, x \rangle \}$
if and only if
\begin{align}
 h\sp{\circ}(\pi) - h\sp{\circ}(\pi- \chi_{Y}) 
 \geq  \sum_{s \in Y}x(s)
 \geq   h\sp{\circ}(\pi+ \chi_{Y}) - h\sp{\circ}(\pi) 
 \qquad  (\forall Y \subseteq S) .
\label{locopthconjx}
\end{align}
Therefore,
\begin{equation} \label{fencoptset2ineq}
x \in \mathcal{P} \ \mbox{\rm and} \ \pi \in \mathcal{D}
\iff
\mbox{\rm 
\eqref{locoptfconjx}, 
\eqref{locopthconjx}
hold}.
\end{equation}

From the above argument we can obtain the following optimality criteria.

\begin{theorem} \label{THfencoptprimaldual}
Let $f$ be  an integer-valued M$\sp{\natural}$-convex function and
$h$ be an integer-valued M$\sp{\natural}$-concave function
such that both
$\mathcal{P}_{0}:=\dom f \cap \dom h$ \ 
and 
$\mathcal{D}_{0} := \dom f\sp{\bullet} \cap \dom h\sp{\circ}$ \ 
are nonempty.

{\rm (1)}
Let $x \in \mathcal{P}_{0}$ and $\pi \in \mathcal{D}_{0}$.
Then the following three conditions are pairwise equivalent.

\quad
{\rm (a)} \ 
$x$ and $\pi$ are both optimal,
that is, 
$x \in \mathcal{P}$ and $\pi \in \mathcal{D}$.

\quad
{\rm (b)} \ 
The inequalities
\eqref{locoptfpi1},
\eqref{locoptfpi2},
\eqref{locopthpi1}, and 
\eqref{locopthpi2}
are satisfied by 
$x$ and $\pi$.

\quad
{\rm (c)} \ 
The inequalities
\eqref{locoptfconjx}
and 
\eqref{locopthconjx}
are satisfied by 
$x$ and $\pi$.

{\rm (2)}
Let $\hat \pi \in \mathcal{D}$ be an arbitrary dual optimal solution.
Then $x\sp{*} \in \mathcal{P}_{0}$ is a minimizer of $f(x) - h(x)$
if and only if it 
is a minimizer of
$f(x) - \langle \hat \pi, x \rangle$
and simultaneously a maximizer of
$h(x) - \langle \hat \pi, x \rangle$,
or equivalently, $x\sp{*}$ satisfies
\eqref{locoptfconjx}
and 
\eqref{locopthconjx}
for $\pi = \hat \pi$.
Namely,
\begin{align} 
\mathcal{P} 
%%& =  \partial f\sp{\bullet}(\hat \pi)  \cap \partial h\sp{\circ}(\hat \pi)
& = \argmin \{ f(x) - \langle \hat \pi, x \rangle \} \cap
  \argmax \{ h(x) - \langle \hat \pi, x \rangle \}
\label{fencprimaloptrep1}
 \\ &
= \{ x \in \ZZ\sp{S} : 
  \mbox{\rm 
  \eqref{locoptfpi1},
  \eqref{locoptfpi2},
  \eqref{locopthpi1}, 
  \eqref{locopthpi2}
  hold with $\pi = \hat \pi$} \}
\label{fencprimaloptrep2}
 \\ &
= \{ x \in \ZZ\sp{S} : 
\mbox{\rm \eqref{locoptfconjx} and \eqref{locopthconjx}
hold with $\pi = \hat \pi$} \}.
\label{fencprimaloptrep3}
\end{align}

{\rm (3)}
Let $\hat x \in \mathcal{P}$ be an arbitrary primal optimal solution.
Then $\pi\sp{*} \in \mathcal{D}_{0}$ is a maximizer of 
$h\sp{\circ}(\pi) - f\sp{\bullet}(\pi)$
if and only if it 
is a minimizer of
$f\sp{\bullet}(\pi) - \langle \pi, \hat x \rangle$
and simultaneously a maximizer of 
$h\sp{\circ}(\pi) - \langle \pi, \hat x \rangle$,
or equivalently, $\pi\sp{*}$ satisfies the inequalities
\eqref{locoptfpi1},
\eqref{locoptfpi2},
\eqref{locopthpi1}, and 
\eqref{locopthpi2}
for $x = \hat x$.
Namely,
\begin{align} 
\mathcal{D} 
&= \argmin \{ f\sp{\bullet}(\pi) - \langle \pi, \hat x \rangle \} \cap
\argmax \{ h\sp{\circ}(\pi) - \langle \pi, \hat x \rangle \}
\label{fencdualoptrep1}
 \\ &
= \{ \pi \in \ZZ\sp{S} : 
\mbox{\rm 
\eqref{locoptfpi1},
\eqref{locoptfpi2},
\eqref{locopthpi1}, 
\eqref{locopthpi2}
hold with $x = \hat x$} \}
\label{fencdualoptrep2}
\\ &
= \{ \pi \in \ZZ\sp{S} : 
\mbox{\rm \eqref{locoptfconjx} and \eqref{locopthconjx}
hold with $x = \hat x$} \}.
\label{fencdualoptrep3}
\end{align}
\finbox
\end{theorem}

It is emphasized that in the representation of $\mathcal{P}$,
each of 
$\argmin \{ f(x) - \langle \hat \pi, x \rangle \}$
and 
$\argmax \{ h(x) - \langle \hat \pi, x \rangle \}$
depends on the choice of $\hat \pi$, but their intersection 
is uniquely determined and equal to $\mathcal{P}$.
Similarly,  in the representation of $\mathcal{D}$,
each of 
$\argmin \{ f\sp{\bullet}(\pi) - \langle \pi, \hat x \rangle \}$
and $\argmax \{ h\sp{\circ}(\pi) - \langle \pi, \hat x \rangle \}$
depends on the choice of $\hat x$, but their intersection 
is uniquely determined and equal to $\mathcal{D}$.

The representation of $\mathcal{P}$ in 
\eqref{fencprimaloptrep1}  
(or \eqref{fencprimaloptrep3})  
shows that $\mathcal{P}$ is 
 the intersection of two {\rm M}$\sp{\natural}$-convex sets.
Such a set is called an 
{\bf {\rm M}$_{2}\sp{\natural}$-convex set}
\cite[Section 4.7]{Mdcasiam}.
Note that the intersection of {\rm M}$\sp{\natural}$-convex sets
is not always {\rm M}$\sp{\natural}$-convex.
The representation of $\mathcal{D}$ in 
\eqref{fencdualoptrep1} 
(or \eqref{fencdualoptrep2}) 
shows that $\mathcal{D}$ is 
the intersection of two {\rm L}$\sp{\natural}$-convex sets.
Since the intersection of two (or more) {\rm L}$\sp{\natural}$-convex sets is again
{\rm L}$\sp{\natural}$-convex, 
$\mathcal{D}$ is an {\rm L}$\sp{\natural}$-convex set.

\begin{proposition} \label{PRfencminmaxoptsetLM}
In the Fenchel-type min-max relation \eqref{fencminmaxM2}
for  M$\sp{\natural}$-convex/concave functions,
the set $\mathcal{P}$ of the minimizers
is an {M}$_{2}\sp{\natural}$-convex set
and the set $\mathcal{D}$ of the maximizers
is an {L}$\sp{\natural}$-convex set. 
\finbox
\end{proposition}

\begin{remark} \rm  \label{RMfenccert2}
In Remark \ref{RMfenccert1} we have discussed the role of
the Fenchel-type discrete duality theorem
for the certificate of optimality in minimizing $f(x) - h(x)$. 
We have distinguished two cases according to wheter 
the explicit forms of the conjugate functions 
$f\sp{\bullet}(\pi)$ and $h\sp{\circ}(\pi)$ are available or not.
If their explicit forms are known, 
we can verify the optimality of $x$ by simply computing 
the values of $f(x) - h(x)$ for $x$  and  
$h\sp{\circ}(\pi) - f\sp{\bullet}(\pi)$
for a given dual optimal $\pi$.
Even if the explicit forms of the conjugate functions are not known,
Theorem \ref{THfencoptprimaldual} (2) above enables us to verify the optimality of $x$
by checking the inequalities
\eqref{locoptfpi1},
\eqref{locoptfpi2},
\eqref{locopthpi1}, and
\eqref{locopthpi2}
for a given dual optimal $\pi$.
Note that we have $O(|S|\sp{2})$ inequalities in total.
We emphasize that 
Theorem \ref{THfencoptprimaldual} (2) is derived from a combination of 
the Fenchel-type discrete duality theorem (Theorem \ref{THmlfencdual})
with the M-optimality criterion (Theorem \ref{THmopt}). 
\finbox
\end{remark}

\begin{remark} \rm \label{RMsubdiff}
In convex analysis, as well as in discrete convex analysis, the optimality conditions 
such as those in Theorem \ref{THfencoptprimaldual}
are expressed usually in terms of 
subgradients and subdifferentials.
In this paper, however, we have intentionally avoided 
using these concepts  
for the sake of the audience from combinatorial optimization.
In this remark we will briefly indicate how the results in 
Theorem \ref{THfencoptprimaldual}
can be described and interpreted in terms of 
subgradients and subdifferentials.

Let 
$f: \ZZ\sp{S} \to \ZZ \cup \{ +\infty \}$
and
$h: \ZZ\sp{S} \to \ZZ \cup \{ -\infty \}$
be integer-valued functions defined on $\ZZ\sp{S}$.
The integral subdifferential of $f$ at $x \in \dom f$
and its concave version for $h$ at $x \in \dom h$ 
are the sets of integer vectors defined as
\begin{eqnarray}
 \partial f(x) 
& := & 
  \{ \pi \in  \ZZ\sp{S} : 
  f(y) - f(x)  \geq  \langle \pi, y - x \rangle  \ \ (\forall y \in \ZZ\sp{S}) \} ,
\label{subgrfdef}
\\
 \partial h(x) 
& := &   \{ \pi \in  \ZZ\sp{S} : 
  h(y) - h(x)  \leq  \langle \pi, y - x \rangle  \ \ (\forall y \in \ZZ\sp{S}) \} .
\label{subgrhdef}
\end{eqnarray}
A member of $\partial f(x)$ is called a subgradient of $f$ at $x$.
Accordingly, the integral subdifferentials 
of $f\sp{\bullet}$ and $h\sp{\circ}$ at $\pi$ are defined as
\begin{eqnarray}
 \partial f\sp{\bullet}(\pi) 
&:= &   \{ x \in  \ZZ\sp{S} : 
  f\sp{\bullet}(\tau) - f\sp{\bullet}(\pi)  \geq  \langle \tau - \pi, x \rangle 
  \ (\forall \tau \in \ZZ\sp{S}) \} ,
\label{subgrfconjdef}
\\
 \partial h\sp{\circ}(\pi) 
&:= &   \{ x \in  \ZZ\sp{S} : 
  h\sp{\circ}(\tau) - h\sp{\circ}(\pi)  \leq  \langle  \tau - \pi, x \rangle 
  \ (\forall \tau \in \ZZ\sp{S}) \} ,
\label{subgrhconjdef}
\end{eqnarray}
where $\partial f\sp{\bullet}(\pi)$ 
is defined for $\pi \in \dom f\sp{\bullet}$ 
and $\partial h\sp{\circ}(\pi)$ for $\pi \in \dom h\sp{\circ}$.
The following relations are straightforward 
translations of the corresponding results in
(ordinary) convex analysis to the discrete setting
 (cf., \cite{Mdca98}, \cite{Mdcasiam}):
\begin{align}
\pi \in \partial f(x) & 
\iff \mbox{equality holds in \eqref{youngineqf}} 
\iff x \in \partial f\sp{\bullet}(\pi),
 \label{youngsubgf}
\\ 
\pi \in \partial h(x) & 
\iff \mbox{equality holds in \eqref{youngineqh}} 
\iff x \in \partial h\sp{\circ}(\pi),
 \label{youngsubgh}
\\
\partial f(x) &= \argmin_{\pi} \{ f\sp{\bullet}(\pi) - \langle \pi, x \rangle \},
\label{subgfargfc}
\\
\partial h(x) &= \argmax_{\pi} \{ h\sp{\circ}(\pi) - \langle \pi, x \rangle \},
\label{subgharghc}
\\
 \partial f\sp{\bullet}(\pi) 
&=  \argmin_{x} \{ f(x) - \langle \pi, x \rangle \},
\label{subgfcargf}
\\
 \partial h\sp{\circ}(\pi)
&=  \argmax_{x} \{ h(x) - \langle \pi, x \rangle \} ,
\label{subghcargh}
\end{align}
where the integer biconjugacy ($f\sp{\bullet\bullet}=f$,  $h\sp{\circ\circ}=h$)
is assumed, which is true for M$\sp{\natural}$-convex/concave functions.
By using \eqref{subgfcargf}--\eqref{subghcargh}
in \eqref{fencprimaloptrep1}, 
and 
\eqref{subgfargfc}--\eqref{subgharghc}
in \eqref{fencdualoptrep1}, respectively, 
we obtain the following representations of optimal solutions
\begin{align} 
\mathcal{P} 
 & =  \partial f\sp{\bullet}(\hat \pi)  \cap \partial h\sp{\circ}(\hat \pi) ,
\label{fencprimaloptrep4}
\\
\mathcal{D} 
 & =  \partial f(\hat x) \cap \partial h(\hat x) 
\label{fencdualoptrep4}
\end{align} 
for any 
$\hat \pi \in \mathcal{D}$ and $\hat x \in \mathcal{P}$.
We also have optimality criteria
\begin{align} 
x \in \mathcal{P} 
 & \iff 
 \partial f(x) \cap \partial h(x) \not= \emptyset,
\label{fencprimalopt5}
\\
\pi \in \mathcal{D} 
 & \iff
 \partial f\sp{\bullet}(\pi)  \cap \partial h\sp{\circ}(\pi) \not= \emptyset.
\label{fencdualopt5}
\end{align} 
Finally it is worth mentioning that, by the M-L conjugacy 
\cite[Chapter 8]{Mdcasiam},
the subdifferential of 
an M$\sp{\natural}$-convex function $f$ 
(resp., an M$\sp{\natural}$-concave function $h$)  
is an L$\sp{\natural}$-convex set
and the subdifferential of 
an L$\sp{\natural}$-convex function $f\sp{\bullet}$ 
(resp., an L$\sp{\natural}$-concave function $h\sp{\circ}$)
is an M$\sp{\natural}$-convex set.
\finbox
\end{remark}

\subsection{Separable convex functions on an M-convex set}
\label{SCfencoptsetsepar}

In Theorem \ref{THminmaxgensep}
we have shown a min-max formula
\begin{align} 
 & \min \{  \sum_{s \in S} \varphi_{s} ( x(s) ) : 
   x \in \odotZ{B} \} 
= \max\{ \hat p(\pi)
 - \sum_{s  \in S}  \psi_{s} (\pi(s) ) :   \pi \in \ZZ\sp{S} \} 
\label{minmaxgensepBase}
\end{align}
for an integer-valued separable convex function
\begin{equation} \label{gensepar5}
  \Phi(x) = \sum [\varphi_{s}(x(s)):  s\in S] 
\end{equation}
on an M-convex set $\odotZ{B}$.
Here we introduce notations for the set of feasible points:
\begin{align} 
 \dom \Phi &=  \{ x \in \ZZ\sp{S} :  x(s) \in \dom \varphi_{s} \mbox{ for each } s \in S \},
\label{primaldomBase}
\\
 \mathcal{P}_{0} &= \odotZ{B} \cap \dom {\Phi} =
\{  x \in \odotZ{B}:
 x(s) \in \dom \varphi_{s} \mbox{ for each } s \in S \} ,
\label{primalfeasBase}
\\
 \mathcal{D}_{0} &= \{  \pi \in \ZZ\sp{S}:
\pi \in \dom \hat p, \  \pi(s) \in \dom \psi_{s}  \mbox{ for each } s \in S \} .
\label{dualfeasBase}
\end{align}
The min-max formula \eqref{minmaxgensepBase} holds under the assumption 
of primal feasibility ($\mathcal{P}_{0} \not= \emptyset$) 
or dual feasibility ($\mathcal{D}_{0} \not= \emptyset$).
The unbounded case with both sides of (\ref{minmaxgensepBase})
being equal to $-\infty$ or $+\infty$ is also a possibility in general,
but in this section we assume that the both sides are finite-valued
and denote the set of the minimizers $x$ by $\mathcal{P}$ 
and the set of the maximizers $\pi$ by $\mathcal{D}$.

We can obtain the optimality conditions for \eqref{minmaxgensepBase}  
by applying Theorem~\ref{THfencoptprimaldual} with 
\begin{align*}
& f(x) =\sum [\varphi_{s}(x(s)):  s\in S],
\qquad \ 
 h(x) = -\delta(x),
\\ &
 f\sp{\bullet}(\pi)  = \sum [ \psi_{s}(\pi(s)):  s\in S],
\qquad
 h\sp{\circ}(\pi) =\hat p(\pi) ,
\end{align*}
where $\delta$ is the indicator function 
of $\odotZ{B}$ defined in \eqref{indicBdef}.
However, 
we present a direct derivation from \eqref{minmaxgensepBase}
via weak duality ($\min \geq \max$), 
as it should be more informative and convenient for readers.

For each conjugate pair $(\varphi_{s}, \psi_{s})$,
it follows from the definition \eqref{phiconjdef} that
\begin{equation}  \label{youngineq5}
\varphi_{s}(k) + \psi_{s} ( \ell ) \geq k \ell 
\qquad (k,\ell \in \ZZ),
\end{equation}
which is known as the Fenchel--Young inequality,
where the equality holds 
if and only if
\begin{equation}  \label{youngsubgrad5}
\varphi_{s}(k) - \varphi_{s}(k-1) \leq \ell \leq \varphi_{s}(k+1) - \varphi_{s}(k).
\end{equation}
Let $x \in \mathcal{P}_{0}$ and $\pi \in \mathcal{D}_{0}$.
Then, using 
the Fenchel--Young inequality \eqref{youngineq5}
as well as \eqref{lovextdef} for $p$, we obtain the weak duality:
\begin{align}
 \sum_{s \in S} \varphi_{s} ( x(s) ) 
 -
  \left( \hat p(\pi)  
 - \sum_{s  \in S}  \psi_{s} (\pi(s) ) 
  \right)
  & =  \sum_{s \in S}   \big[ \varphi_{s} ( x(s) ) +  \psi_{s}(\pi(s)) \big] 
  \ - \hat p(\pi) 
\notag  \\ &  
\geq
  \sum_{s \in S}    x(s) \pi(s)  
  \ - \hat p(\pi) 
\label{weakd1Base} 
\\ &  
\geq
 \min \{ \pi z : z \in \odotZ{B} \} - \hat p(\pi) 
\ =  0.
\label{weakd2Base}
\end{align}

The optimality conditions can be obtained as the conditions
for the inequalities in \eqref{weakd1Base} and \eqref{weakd2Base}
to be equalities, as follows.

\begin{proposition} \label{PRminmaxgensepBaseOPT}
Assume that both $\mathcal{P}_{0}$ and $\mathcal{D}_{0}$ 
in \eqref{primalfeasBase}--\eqref{dualfeasBase} are nonempty.

{\rm (1)}
Let $x \in \mathcal{P}_{0}$ and  
$\pi \in \mathcal{D}_{0}$.
Then
$x \in \mathcal{P}$ and $\pi \in \mathcal{D}$
(that is, $x$ and $\pi$ are both optimal)
if and only if
the following two conditions are satisfied:
\begin{align}
 & 
\varphi_{s}(x(s)) - \varphi_{s}(x(s)-1) 
\leq \pi(s)  \leq 
\varphi_{s}(x(s)+1) - \varphi_{s}(x(s))
\qquad (s \in S),  
\label{pisubgradBase}
\\ 
 & \mbox{  \rm $\pi(s) \geq \pi(t)$ 
  \quad  for every $(s,t)$ \ with \ $x + \chi_{s} - \chi_{t} \in \odotZ{B}$}.
\label{piminzerBase}
\end{align}

{\rm (2)}
Let $\hat \pi \in \mathcal{D}$ be an arbitrary dual optimal solution.
Then $x\sp{*} \in \mathcal{P}_{0}$ is a minimizer of $\Phi(x)$ over $\odotZ{B}$
if and only if it satisfies
\eqref{pisubgradBase} and \eqref{piminzerBase}
for $\pi = \hat \pi$,
or equivalently, 
it is a minimizer of
$\sum [\varphi_{s}(x(s)) -  \hat \pi(s) x(s) :  s\in S]$
and simultaneously a $\hat\pi$-minimizer in $\odotZ{B}$.
Namely,
\begin{align}  
\mathcal{P} 
& =  \{ x \in \mathcal{P}_{0}: 
\mbox{\rm \eqref{pisubgradBase}, \eqref{piminzerBase} hold with $\pi = \hat\pi$}  \}
\label{primaloptsetBase1}
\\
& =\{ x \in \dom \Phi:  
\mbox{\rm \eqref{pisubgradBase} holds with $\pi = \hat\pi$}  \} 
\cap 
\{ x \in \odotZ{B}:
 \mbox{\rm $x$ is a $\hat\pi$-minimizer in $\odotZ{B}$ } \} .
 \label{primaloptsetBase2}
\end{align}

{\rm (3)}
Let $\hat x \in \mathcal{P}$ be an arbitrary primal optimal solution.
Then $\pi\sp{*} \in \mathcal{D}_{0}$ is a maximizer of 
$\hat p(\pi) - \sum_{s  \in S}  \psi_{s} (\pi(s) )$
if and only if 
it satisfies the inequalities
\eqref{pisubgradBase} and \eqref{piminzerBase}
for $x = \hat x$.
Namely,
\begin{equation}  \label{dualoptsetBase}
\mathcal{D} =  \{ \pi \in \mathcal{D}_{0} : 
\mbox{\rm \eqref{pisubgradBase}, \eqref{piminzerBase}
 hold with $x = \hat x$}  \} .
\end{equation}
\end{proposition}
\begin{proof}
The inequality \eqref{weakd1Base} turns into an equality 
if and only if,
for each $s \in S$, we have
$\varphi_{s} (k) + \psi_{s} ( \ell ) = k \ell$ 
for $k= x(s)$ and $\ell = \pi(s)$.
The latter condition is equivalent to
\eqref{pisubgradBase} by \eqref{youngsubgrad5}.
The other inequality \eqref{weakd2Base} turns into an equality 
if and only if
$x$ is a $\pi$-minimizer in $\odotZ{B}$,
which is equivalent to \eqref{piminzerBase}.
Finally, we see from \eqref{minmaxgensepBase} that
the two inequalities in \eqref{weakd1Base} and  \eqref{weakd2Base}
simultaneously turn into equality if $x \in \mathcal{P}$ and $\pi \in \mathcal{D}$.
\end{proof}

\begin{proposition} \label{PRminmaxgensepoptsetLM}
In the min-max relation \eqref{minmaxgensepBase}
for a separable convex function on an M-convex set,
the set $\mathcal{D}$ of the maximizers
is an L$\sp{\natural}$-convex set and 
the set $\mathcal{P}$ of the minimizers
is an M-convex set.
\end{proposition}
\begin{proof}
The representation \eqref{dualoptsetBase}
shows that $\mathcal{D}$ is described by the inequalities in 
\eqref{pisubgradBase} and \eqref{piminzerBase}.
Hence $\mathcal{D}$ is {\rm L}$\sp{\natural}$-convex.
(The {\rm L}$\sp{\natural}$-convexity of $\mathcal{D}$ can also be obtained from 
Proposition \ref{PRfencminmaxoptsetLM}.)
In the representation \eqref{primaloptsetBase2} of $\mathcal{P}$, the first set
\ $\{ x \in \dom \Phi: \mbox{\rm \eqref{pisubgradBase} holds with $\pi = \hat\pi$}  \}$ \ 
is a box of integers (the set of integers in an integral box),
while the set of $\hat\pi$-minimizers in $\odotZ{B}$
is an {\rm M}-convex set.
Therefore, $\mathcal{P}$ is an {\rm M}-convex set.
\end{proof}

\subsection{Dual optimal solutions to square-sum minimization}
\label{SCsqsumoptdual}

The min-max formula (\ref{minmaxSqSum-KM2}) for the  square-sum minimization
is a special case of the min-max formula (\ref{minmaxgensepBase}) with
$\varphi_{s}(k)=\varphi(k)=k\sp{2}$ and 
$\psi_{s}(\ell) = \psi(\ell) 
= \left\lfloor {\ell}/{2} \right\rfloor \cdot \left\lceil {\ell}/{2} \right\rceil$
for $k, \ell \in \ZZ$
(cf., \eqref{squareconj}).
Accordingly, we can apply the general results
(Proposition~\ref{PRminmaxgensepBaseOPT}, in particular) 
for the analysis of the optimal solutions in the min-max formula (\ref{minmaxSqSum-KM2}).
In this section we consider the dual solutions,
whereas the primal solutions are treated in Section \ref{SCsqsumoptprimal}.

\medskip

The function $g(\pi) = \hat p(\pi) - \sum [\psi(\pi(s)):  s\in S]$
to be maximized in  (\ref{minmaxSqSum-KM2})
is L$\sp{\natural}$-concave
by Proposition~\ref{PRdualobjlnat},
and the maximizers of an L$\sp{\natural}$-concave function
form an {\rm L}$\sp{\natural}$-convex set \cite[Theorem 7.17]{Mdcasiam}.
Therefore, the set $\Pi$ 
of dual optimal solutions is an {\rm L}$\sp {\natural}$-convex set,
which is the first statement of Proposition~\ref{PRoptdualsetlnatKM}.
The L$\sp{\natural}$-convexity of $\Pi$ implies that 
there exists a unique smallest element of $\Pi$.
The second statement of Proposition~\ref{PRoptdualsetlnatKM}
shows that this smallest element is given by $\pi\sp{*}$,
but this fact is not easily shown by general arguments 
from discrete convex analysis.

Next we consider Theorem~\ref{THoptdualsetKM},
which gives a representation of $\Pi$.
According to the general result stated in Proposition~\ref{PRminmaxgensepBaseOPT} (3),
we can obtain another representation of $\Pi$ 
by choosing any dec-min element $\hat x$ of $\odotZ{B}$,
which is a primal optimal solution for (\ref{minmaxSqSum-KM2}).
In this case the condition \eqref{pisubgradBase} reads
\begin{equation} \label{subgrfm}
 2 x(s) - 1 \leq \pi(s) \leq 2 x(s) + 1  
\qquad (s \in S) ,   
\end{equation}
since
$\varphi(k) - \varphi(k-1) =  k\sp{2} - (k-1)\sp{2} = 2k - 1$
and 
$\varphi(k+1) - \varphi(k) = (k+1)\sp{2} - k\sp{2} = 2k + 1$.

\begin{proposition} \label{PRstrD2}
Let $m$ be any dec-min element of $\odotZ{B}$.
The set $\Pi$ of dual optimal solutions to {\rm (\ref{minmaxSqSum-KM2})}
is represented as
$\Pi = I(m) \cap P(m)$, where
\begin{align*}
I(m) &= \{ \pi \in \ZZ\sp{S}  :  
 2 m(s) - 1 \leq \pi(s) \leq 2 m(s) + 1  \mbox{\rm \  for all $s \in S$} \},
\\
P(m) &= \{ \pi \in \ZZ\sp{S}  :  
\mbox{\rm $\pi(s) \geq \pi(t)$ 
\  \  for every $(s,t)$ \ with \ $x + \chi_{s} - \chi_{t} \in \odotZ{B}$}
\}.
\end{align*}
Hence $\Pi$ is an {\rm L}$\sp {\natural}$-convex set.
\finbox
\end{proposition}

Let us compare the representations of $\Pi$
in Proposition~\ref{PRstrD2} and Theorem~\ref{THoptdualsetKM}.
Roughly speaking, $I(m)$ corresponds to 
the first two conditions (\ref{Si-Fi-KM}) and (\ref{Fi-KM})
in Theorem~\ref{THoptdualsetKM} 
and $P(m)$ to the third condition (\ref{pispit-KM}). 
However, there is an essential difference between 
Proposition~\ref{PRstrD2} and Theorem~\ref{THoptdualsetKM}.
Namely, each of $I(m)$ and $P(m)$ 
varies with the choice of $m$, 
while their intersection is uniquely determined and equal to $\Pi$.
In this sense, the description of $\Pi$ 
in Proposition~\ref{PRstrD2} is not canonical.
Theorem~\ref{THoptdualsetKM}
is a much stronger statement, giving a canonical description of $\Pi$
without reference to a particular primal optimal solution.

\begin{remark} \rm \label{RMoptdualO1O2}
Proposition~\ref{PRstrD2} above is equivalent to Proposition~\RefPartI{6.11} of Part~I \cite{FM18part1},
though in a slightly different form.
Recall the optimality criteria there:%
\footnote{%%%%%%%%%%%%%%%%%%%%
For a given vector $\pi$ in ${\bf R}\sp S$, we call a nonempty set
$X\subseteq S$ a {\bf $\pi$-top set} 
if $\pi(u)\geq \pi(v)$ holds whenever $u\in X$ and $v\in S-X$.  
If $\pi(u)>\pi(v)$ holds whenever $u\in X$ and $v\in S-X$, we speak of a {\bf strict $\pi$-top set}.  
We call a subset $X\subseteq S$ 
{\bf $m$-tight} with respect to $p$ if $\widetilde m(X)=p(X)$.  
} %%%%%%%%%%%%%%% footnote %%%%%%%%%%%
\begin{align*}
{\rm (O1)} &  \qquad
m(s) \in \{ \left\lfloor \pi(s)/2 \right\rfloor,  \left\lceil \pi(s)/2 \right\rceil  \}
\mbox{\  for each \  } s \in S,
\\
{\rm (O2)} & \qquad
\mbox{each strict $\pi$-top-set is $m$-tight with respect to $p$.}
\end{align*}
The set $I(m)$ corresponds to the first optimality criterion (O1),
since 
$2 m(s) - 1 \leq \pi(s) \leq 2 m(s) + 1$
if and only if
$m(s) \in \{ \left\lfloor \pi(s)/2 \right\rfloor,  \left\lceil \pi(s)/2 \right\rceil  \}$.
The equivalence of $P(m)$ 
to the second criterion (O2)
is a well-known characterization of a minimum weight base. 
\finbox
\end{remark}

\subsection{Primal optimal solutions to square-sum minimization}
\label{SCsqsumoptprimal}

We now turn to the primal problem of (\ref{minmaxSqSum-KM2}), namely, the square-sum minimization.

Let  ${\rm dm}(\odotZ{B})$ denote the set of the dec-min elements
of $\odotZ{B}$.
By Theorem~\ref {THdecminPhistrict2},
${\rm dm}(\odotZ{B})$ coincides with
the set of primal optimal solutions for (\ref{minmaxSqSum-KM2}).
According to the general result in Proposition~\ref{PRminmaxgensepBaseOPT} (2),
a representation of ${\rm dm}(\odotZ{B})$ 
can be obtained by choosing any dual optimal solution $\hat \pi$.
In this case the condition \eqref{pisubgradBase}
is simplified to  \eqref{subgrfm}, which can be rewritten as 
\begin{equation} \label{subgrfconj}
x(s) \in \{ 
\left\lfloor \pi(s)/2 \right\rfloor,  \left\lceil \pi(s)/2 \right\rceil  \} 
\qquad (s \in S).  
\end{equation}

Thus the following representation of the set of dec-min elements is obtained.

\begin{proposition} \label{PRstrPdecmin}
Let $\hat \pi$ be any dual optimal solution to {\rm (\ref{minmaxSqSum-KM2})}.
The set ${\rm dm}(\odotZ{B})$ of dec-min elements of $\odotZ{B}$
is represented as
${\rm dm}(\odotZ{B}) = T(\hat \pi) \cap \odotZ{B\sp{\circ}}(\hat \pi)$, where
\begin{align*}
T(\hat \pi) &=  \{ m \in  \ZZ\sp{S} :  m(s) \in \{ 
   \left\lfloor \hat \pi(s)/2 \right\rfloor,
\left\lceil \hat \pi(s)/2 \right\rceil \} \ (s \in S)  \},
\\ 
\odotZ{B\sp{\circ}}(\hat \pi) &= 
\{ m \in \odotZ{B} :  \mbox{\rm $m$ is a minimum $\hat \pi$-weight element of $\odotZ{B}$} \}.
\end{align*}
Hence ${\rm dm}(\odotZ{B})$ is a matroidal M-convex set.
\finbox
\end{proposition}

Again, each of $T(\hat \pi)$ and $\odotZ{B\sp{\circ}}(\hat \pi)$  
varies with the choice of $\hat \pi$, but their intersection 
is uniquely determined and is equal to ${\rm dm}(\odotZ{B})$.
Here, $\odotZ{B\sp{\circ}}(\hat \pi)$ is  the integral elements of a face of $B$,
and is an M-convex set.
As for $T(\hat \pi)$, note that, for each $s \in S$,
the two numbers
$\left\lfloor \hat \pi(s)/2 \right\rfloor$
and
$\left\lceil \hat \pi(s)/2 \right\rceil$
are the same integer or consecutive integers.
Therefore, ${\rm dm}(\odotZ{B})$ is a matroidal M-convex set. 
In other words, there exist a matroid $\hat M$ 
and a translation vector $\hat \Delta \in \ZZ\sp{S}$
such that
\[
{\rm dm}(\odotZ{B}) = T(\hat \pi) \cap \odotZ{B\sp{\circ}}(\hat \pi)
= \{ \chi_L+ \hat \Delta  :  \mbox{$L$ is a basis of $\hat M$}  \}.
\]
In this construction both $\hat M$ and $\hat \Delta$
depend on the chosen $\hat \pi$; in particular,
$\hat \Delta = \left\lfloor \hat \pi/2 \right\rfloor$.

Theorem~\ref{THmatroideltoltKM} is significantly 
stronger than Proposition~\ref{PRstrPdecmin},
in that it gives a concrete description of the matroid $\hat M$
by referring to the canonical chain.
The translation vector $\Delta \sp{*}$ 
in Theorem~\ref{THmatroideltoltKM} 
corresponds to the choice of
$\hat \pi = \pi\sp{*}$; note that we indeed have the relation 
$\Delta \sp{*} = \left\lfloor \pi\sp{*}/2 \right\rfloor$.

\begin{remark} \rm  \label{RMnonMunitbox}
%%\memo{Addition, 2019-05-30, after arXiv, Ver.2}\\
Proposition~\ref{PRstrPdecmin} implies, in particular, that the dec-min elements
of an M-convex set is contained in a 
small box (unit box).
Note that such a property does not hold 
for an arbitrary integral polyhedron. 
To see this, consider the line segment $P$ in $\RR\sp{3}$
 connecting two points $(2,1,0)$ and $(1,0,2)$.
This $P$ is an integral polyhedron,
$\odotZ{P} = \{ (2,1,0), (1,0,2) \}$,
and $\odotZ{P}$ is not an M-convex set.
Both $(2,1,0)$ and $(1,0,2)$
are dec-min in $\odotZ{P}$, but there exists no 
small box (unit box) 
containing them,
since their third components differ by 2.
In Part III 
prove that this 
small box (unit box) 
property also holds for network flows.
\finbox
\end{remark} 

%%% end file %%%

%% murota 2018-08-25 / 2019-05-22 / 2019-07-09

\section{Comparison of continuous and discrete cases}
\label{SCcompRZ}

While our present study is focused on the discrete case for an M-convex set $\odotZ{B}$,
the continuous case for a base-polyhedron $B$ 
was investigated by Fujishige \cite{Fuj80} around 1980
under the name of lexicographically optimal bases,
as a generalization of lexicographically optimal maximal flows
considered by Megiddo \cite{Meg74}.
Lexicographically optimal bases are discussed in detail in  
\cite[Section 9]{Fuj05book}.
Later in game theory  Dutta--Ray \cite{DR89} treated majorization ordering 
in the continuous case
under the name of egalitarian allocation; see also Dutta \cite{Dut90}. 
See also the survey of related papers in Appendix \ref{SCprevwksurvey}.

Section~\ref{SCcompRZsummary} offers comparisons of major ingredients in discrete and continuous cases.
These comparisons show that the discrete case
is significantly different from the continuous case,
being endowed with a number of intriguing combinatorial structures 
on top of the geometric structures known in the continuous case. 
Section~\ref{SCprinpatR} is devoted to a review of
the principal partition (adapted to a supermodular function),
Section~\ref{SCcanopat} gives an alternative characterization of the canonical partition,
and Section~\ref{SCcanopatprinpat} clarifies their relationship.
Algorithmic implications are discussed in Section~\ref{SCalgRZ}.

\subsection{Summary of comparisons}
\label{SCcompRZsummary}

The continuous case is referred to as Case $\RR$
and the discrete case as Case $\ZZ$.
We use notation $m_{\RR}$ and $m_{\ZZ}$  
for the dec-min element in Case $\RR$ and Case $\ZZ$, respectively.

\paragraph{Underlying set}
In Case $\RR$ we consider a base-polyhedron $B$ described by 
a real-valued supermodular function $p$ or a submodular function $b$.
In Case $\ZZ$ we consider the set $\odotZ{B}$  of integral members of 
an integral base-polyhedron $B$ described by 
an integer-valued $p$ or $b$.

\paragraph{Terminology}

In Case $\RR$ the terminology of 
``lexicographically optimal base''
(or ``lexico-optimal base'')
is used in \cite{Fuj80,Fuj05book}.
A lexico-optimal base is the same as an inc-max element 
in our terminology,
whereas a dec-min element is called a
``co-lexicographically optimal base'' in \cite{Fuj05book}.

\paragraph{Weighting}
In Case $\RR$ a weight vector is introduced 
to define and analyze lexico-optimality,
while this is not the case in this paper for Case $\ZZ$.
In the following comparisons we always assume that no weighting is introduced 
in Cases $\RR$ and $\ZZ$.
In a forthcoming paper, 
%%\cite{FM18part5}, however, 
we consider
discrete decreasing minimality with respect to a weight vector.

\paragraph{Decreasing minimality and increasing maximality}
In Case $\ZZ$ decreasing minimality in $\odotZ{B}$ is equivalent to increasing maximality.
This statement is also true in Case $\RR$.
That is, an element of $B$ is dec-min in $B$ if and only if it is inc-max in $B$.
Moreover, a least majorized element exists in $\odotZ{B}$ (in Case $\ZZ$)
and in $B$ (in Case $\RR$).

\paragraph{Square-sum minimization}
In both Cases $\ZZ$ and $\RR$,
a dec-min element is characterized as a minimizer of
square-sum of the components
$W(x) = \sum [x(s)\sp{2}:  s\in S]$.
In Case $\RR$, 
the minimizer is unique, and is often referred to as 
the minimum norm point.

\paragraph{Uniqueness}
The structures of dec-min elements
have a striking difference in Cases $\RR$ and $\ZZ$. 
In Case $\RR$ the dec-min element of $B$
is uniquely determined, and is given by the minimum norm point of $B$.
In Case $\ZZ$ 
the dec-min elements of $\odotZ{B}$
are endowed with the structure of basis family of a matroid, as formulated in 
Theorem \ref{THmatroideltoltKM}.
The minimum norm point of $B$
can be expressed as a convex combination of 
the dec-min elements of $\odotZ{B}$ (cf., Theorem \ref{THmnormconvcombdmS}).

\paragraph{Proximity}
Every dec-min element $m_{\ZZ}$ of $\odotZ{B}$ is located near 
the minimum norm point $m_{\RR}$ of $B$, satisfying
$\left\lfloor m_{\RR} \right\rfloor \leq m_{\ZZ} \leq  \left\lceil m_{\RR} \right\rceil$
(cf., Theorem~\ref{THdecminproxS}).
However, not every integer vector $m_{\ZZ}$ in $B$ 
satisfying $\left\lfloor m_{\RR} \right\rfloor \leq m_{\ZZ} \leq  \left\lceil m_{\RR} \right\rceil$
is a dec-min element of $\odotZ{B}$,
which is demonstrated by the following example.

\begin{example} \rm \label{EX4dimC4}
Let $\odotZ{B}$ be an M-convex set consisting of five vectors%
\footnote{%%%%%%%%%%%%%%
$\odotZ{B}$ is obtained from  
$\{ (1,0,1,0), \ (1,0,0,1), \  (0,1,1,0), \ (0,1,0,1), \ (1,1,0,0) \}$ 
(basis family of rank 2 matroid) by a translation with $(1,1,0,0)$.
} %%%%%%%%% footnote %%%%%%%%%%
\[
 m_{1}=(2,1,1,0), \quad
 m_{2}=(2,1,0,1), \quad
 m_{3}=(1,2,1,0), \quad
 m_{4}=(1,2,0,1), \quad
 m_{5}=(2,2,0,0)
\]
and $B$ be its convex hull.
The dec-min elements of $\odotZ{B}$  are 
$m_{1}$, $m_{2}$, $m_{3}$, and $m_{4}$,
whereas 
$ m_{5}=(2,2,0,0)$ is not dec-min.
The minimum norm point of the base-polyhedron $B$ 
is $m_{\RR} = (3/2, 3/2, 1/2, 1/2 )$, for which 
$\left\lfloor m_{\RR} \right\rfloor = (1,1,0,0)$ and 
$\left\lceil m_{\RR} \right\rceil = (2,2,1,1)$.
The point $m_{5}=(2,2,0,0)$
satisfies $\left\lfloor m_{\RR} \right\rfloor \leq m_{5} \leq  \left\lceil m_{\RR} \right\rceil$
but it is not a dec-min element.
\finbox
\end{example}

\paragraph{Min-max formula}
In Case $\ZZ$ we have the min-max identity (\ref{minmaxSqSum-KM}):
\[ 
\min \{ \sum [m(s)\sp{2}:  s\in S]:  m\in \odotZ{B} \} 
= \max \{\hat p(\pi ) - \sum _{s\in S} 
 \left\lfloor {\pi (s) \over 2}\right\rfloor 
 \left\lceil {\pi (s) \over 2}\right\rceil : 
 \pi \in \ZZ\sp S \}.  
\]
In Case $\RR$ the corresponding formula is
\begin{equation}  \label{minmaxsqrsumR}
\min \{ \sum [x(s)\sp{2}:  s\in S]:  x\in B \} 
= \max \{\hat p(\pi ) - \sum _{s\in S} 
 \left( {\pi (s) \over 2}\right)\sp{2} : \pi \in \RR \sp S \} ,
\end{equation}
which may be regarded as an adaptation of the standard quadratic programming duality
to the case where the feasible region is a base-polyhedron.
To the best knowledge of the authors, 
the formula (\ref{minmaxsqrsumR}) has never been shown in the literature.

\paragraph{Principal partition vs canonical partition}
The canonical partition for Case $\ZZ$ is closely related to 
the principal partition for Case $\RR$.
The principal partition (adapted to a supermodular function) 
is described in Section~\ref{SCprinpatR} 
and the following relations are established in Sections \ref{SCcanopat} and \ref{SCcanopatprinpat}.
We denote the canonical partition by
$\{ S_{1}, S_{2}, \ldots, S_{q} \}$
and the principal partition by
$\{ \hat S_{1}, \hat S_{2}, \ldots, \hat S_{r} \}$.
They are constructed from the canonical chain 
$C_{1} \subset  C_{2} \subset  \cdots \subset  C_{q}$
and the principal chain 
$\hat C_{1} \subset \hat C_{2} \subset  \cdots \subset \hat C_{r}$,
respectively, as the families of difference sets:
$S_{j}=C_{j} -  C_{j-1}$ for $j=1,2,\ldots, q$ and
$\hat S_{i}=\hat C_{i} -  \hat C_{i-1}$ for $i=1,2,\ldots, r$,
where $C_{0} = \hat C_{0} = \emptyset$.
We denote the essential values by $\beta_{1} > \beta_{2} > \cdots > \beta_{q}$
and the critical values by
$\lambda_{1} > \lambda_{2} >  \cdots >  \lambda_{r}$.

\begin{itemize}
\item
An integer $\beta$ is an essential value for Case $\ZZ$ if and only if there exists 
a critical value $\lambda$ for Case $\RR$ 
satisfying $\beta  \geq \lambda > \beta -1$.
The essential values $\beta_{1} > \beta_{2} > \cdots > \beta_{q}$
are obtained 
from the critical values 
$\lambda_{1} > \lambda_{2} >  \cdots >  \lambda_{r}$
as the distinct members of the rounded-up integers
$\lceil \lambda_{1} \rceil \geq \lceil \lambda_{2} \rceil \geq \cdots \geq \lceil \lambda_{r} \rceil$.

\item
The canonical partition 
$\{ S_{1}, S_{2}, \ldots, S_{q} \}$
is obtained from the principal partition 
$\{ \hat S_{1}, \hat S_{2}, \ldots, \hat S_{r} \}$
as an aggregation; we have
$S_{j} = \bigcup_{i \in I(j)} \hat S_{i}$,
where $I(j) = \{ i :  \lceil \lambda_{i} \rceil  = \beta_{j} \}$.

\item
The canonical chain $\{ C_{j} \}$ is a subchain of the principal chain  $\{ \hat C_{i} \}$;
we have $C_{j} = \hat C_{i}$ for $i= \max I(j)$.

\item
In Case $\RR$, the dec-min element $m_{\RR}$ of $B$ is uniform 
on each member $\hat S_{i}$ of the principal partition, i.e., 
$m_{\RR}(s) = \lambda_{i}$ if $s \in  \hat S_{i}$, where $i=1,2,\ldots, r$
(cf., Proposition \ref{PRlexoptbaseR}).
In Case $\ZZ$, the dec-min element $m_{\ZZ}$ of $\odotZ{B}$ is near-uniform 
on each member $S_{j}$ of the canonical partition, i.e., 
$m_{\ZZ}(s) \in \{  \beta_{j}, \beta_{j} -1 \}$ 
if $s \in S_{j}$, where $j=1,2,\ldots,q$
(cf., Theorem \RefPartI{5.1} of Part~I \cite{FM18part1}).
\end{itemize}

\paragraph{Algorithm}
In Case $\ZZ$ we have developed a strongly polynomial algorithm 
for finding a dec-min element of $\odotZ{B}$ 
(Section~\RefPartI{7} of Part~I \cite{FM18part1}).
In Case $\RR$ the decomposition algorithm of Fujishige \cite{Fuj80}
finds the minimum norm point $m_{\RR}$
in strongly polynomial time.
Our proximity result (Theorem~\ref{THdecminproxS})
leads to the following ``continuous relaxation'' approach.
Let
$\ell = \left\lfloor m_{\RR} \right\rfloor$ and $u= \left\lceil m_{\RR} \right\rceil$,
and 
let $\odotZ{B_{\ell}\sp{u}}$
denote 
the intersection of $\odotZ{B}$ with  the 
box (interval) $[\ell, u]= [\ell, u]_{\RR}$ 
(or $T(\ell, u)$ in the notation of Part~I).
The dec-min element of $\odotZ{B_{\ell}\sp{u}}$
is also a dec-min element of $\odotZ{B}$,
since the box $[\ell, u]$ contains all dec-min elements of $\odotZ{B}$
by Theorem~\ref{THdecminproxS}.
Since $0 \leq u(s) - \ell(s) \leq 1$ for all $s \in S$,
$\odotZ{B_{\ell}\sp{u}}$ can be regarded as a matroid translated by $\ell$,
i.e.,
$\odotZ{B_{\ell}\sp{u}} = \{ \ell + \chi_{L} : L \mbox{ is a base of $M$} \}$,
where $M$ is a matroid.
Therefore,
the dec-min element of $\odotZ{B_{\ell}\sp{u}}$ 
can be computed as the minimum weight base  of matroid $M$
with respect to the weight vector $w$ defined by
$w(s) = u(s)\sp{2} - \ell(s)\sp{2}$  ($s \in S$).
By the greedy algorithm we can find the minimum weight base of $M$ in strongly polynomial time.
Thus the total running time of this algorithm is bounded by strongly polynomial time. 
Variants of such continuous relaxation algorithm are given in Section~\ref{SCalgRZ}.
In the literature 
\cite{Fuj05book,Gro91,Hoc07,KSI13}
we can find continuous relaxation algorithms
that are strongly polynomial for special classes of base-polyhedra;
see Appendix \ref{SCprevwksurvey} for details.

\subsection{Review of the principal partition}
\label{SCprinpatR}

As is pointed out by Fujishige \cite{Fuj80},
the dec-min element in the continuous case
is closely related to the principal partition.
The principal partition is 
the central concept 
in a structural theory for submodular functions developed mainly in Japan;
Iri gives an early survey in \cite{Iri79} and 
Fujishige provides a comprehensive historical and technical account
in \cite{Fuj09bonn}.
In this section we summarize the results 
that are relevant to the analysis of the dec-min element in the continuous case.
Originally \cite{Fuj80}, the results are stated for a real-valued submodular function,
and the present version is a translation for a real-valued supermodular function 
$p: 2\sp{S} \to \RR \cup \{ -\infty \}$.

For any real number $\lambda$,
let $\mathcal{L}(\lambda)$ denote the family of all maximizers of
$p(X) - \lambda |X|$.
Then
$\mathcal{L}(\lambda)$ is a ring family (lattice), and we denote 
its smallest member by $L(\lambda)$.
That is, $L(\lambda)$ denotes the smallest maximizer of $p(X) - \lambda |X|$.

The following is a  well-known basic fact.
The proof is included for completeness.

\begin{proposition} \label{PRbasicPrinPatR}
Let $\lambda > \lambda'$.
If $X \in \mathcal{L}(\lambda)$ and $Y \in \mathcal{L}(\lambda')$, 
then  $X \subseteq Y$.
In particular, $L(\lambda) \subseteq L(\lambda')$.
\end{proposition} 
\begin{proof}
Let $X \in \mathcal{L}(\lambda)$ and $Y \in \mathcal{L}(\lambda')$.
We have
\begin{align}
 p(X) +  p(Y)  & \leq p(X \cap Y) + p(X \cup Y),
\nonumber \\
 \lambda |X| + \lambda' |Y|  
& =  \lambda |X \cap Y| + \lambda' |X \cup Y| 
+ (\lambda - \lambda') |X - Y|
\nonumber \\ & 
\geq
  \lambda |X \cap Y| + \lambda' |X \cup Y| .
\label{pripatprf}
\end{align}
It follows from these inequalities that
\[
 ( p(X) - \lambda |X| ) +  ( p(Y) - \lambda' |Y| ) \leq 
 ( p(X \cap Y) - \lambda |X \cap Y| ) +  ( p(X \cup Y) - \lambda' |X \cup Y| ) .
\]
Here the reverse inequality $\geq$ is also true by 
$X \in \mathcal{L}(\lambda)$ and $Y \in \mathcal{L}(\lambda')$.
Therefore, we have equality in (\ref{pripatprf}),
which implies $|X - Y|=0$, i.e.,  $X \subseteq Y$.
\end{proof}

There are finitely many numbers $\lambda$ for which
$| \mathcal{L}(\lambda)| \geq 2$.
We denote such numbers as
$\lambda_{1} > \lambda_{2} > \cdots > \lambda_{r}$,
which are called the {\bf critical values}.
It is easy to see that $\lambda$ is a critical value if and only if
$L(\lambda) \not= L(\lambda - \varepsilon)$
for any $\varepsilon > 0$.

The {\bf principal partition}
$\{ \hat S_{1}, \hat S_{2}, \ldots, \hat S_{r} \}$
is defined by 
\begin{equation}  \label{prinpatRdef1}
 \hat S_{i} = \max \mathcal{L}(\lambda_{i}) - \min \mathcal{L}(\lambda_{i})
\qquad (i=1,2,\ldots, r),
\end{equation}
which says that $\hat S_{i}$ is the difference 
of the largest and the smallest element of $\mathcal{L}(\lambda_{i})$.
Alternatively, 
\begin{equation}  \label{prinpatRdef2}
 \hat S_{i}   = L(\lambda_{i} - \varepsilon) - L(\lambda_{i})
\end{equation}
for a sufficiently small $\varepsilon > 0$.

By defining 
$\hat C_{i} = \hat S_{1} \cup \hat S_{2} \cup \cdots \cup  \hat S_{i}$
for $i=1,2,\ldots, r$
we obtain a chain: 
$\hat C_{1} \subset  \hat C_{2} \subset  \cdots \subset \hat C_{r}$,
where $\hat C_{1} \not= \emptyset$ and $\hat C_{r} = S$;
we also define $\hat C_{0} = \emptyset$.
Then the chain
$(\emptyset =) \hat C_{0} \subset  \hat C_{1} \subset  \hat C_{2} 
\subset  \cdots \subset \hat C_{r} \ (=S)$
is a maximal chain of the lattice $\bigcup_{\lambda \in \RR} \mathcal{L}(\lambda)$.
In this paper we call this chain the {\bf principal chain}.
By slight abuse of terminology 
the principal chain sometime means the chain 
$\hat C_{1} \subset  \hat C_{2} \subset  \cdots \subset \hat C_{r} \ (=S)$
without $\hat C_{0}$ ($=\emptyset)$.

Let $m_{\RR} \in \RR\sp{S}$ be the  minimum norm point of $B$,
which is the unique dec-min element of $B$.
The critical values are exactly those numbers that appear as component values
of $m_{\RR}$.  Moreover,
the vector $m_{\RR}$ is uniform on each member $\hat S_{i}$.

\begin{proposition}[Fujishige \cite{Fuj80}] \label{PRlexoptbaseR}
$m_{\RR}(s) = \lambda_{i}$ if $s \in  \hat S_{i}$, where $i=1,2,\ldots, r$. 
\finbox
\end{proposition}

\subsection{New characterization of the canonical partition}
\label{SCcanopat}

For the discrete case,
the canonical partition describes the structure of dec-min elements.
In particular, a dec-min element is near-uniform on each member of the canonical partition.%
\footnote{%%%%%%%%%%%%%%%%%%%%%%%%%%%%%%%%%%
That is, $|m_{\ZZ}(s) - m_{\ZZ}(t)| \leq 1$ if $\{ s, t \} \subseteq S_{j}$ for some $S_{j}$ 
 (cf., Theorem \RefPartI{5.1} of Part~I \cite{FM18part1}).
} %%%%%%%%%%%%%%%%%%%%%%%%%%%%%%%%%%%%%
In Part~I \cite{FM18part1}, 
the canonical partition has been defined iteratively using contractions.
In this section we give a non-iterative construction of this canonical partition,
which reflects the underlying structure more directly.
This alternative construction enables us to reveal the precise relation 
between the discrete and continuous cases in Section \ref{SCcanopatprinpat}.

We first recall the iterative construction from Section~\RefPartI{5} of Part~I \cite{FM18part1}.
Let $p: 2\sp{S} \to \ZZ \cup \{ -\infty \}$
be an integer-valued supermodular function
with $p(\emptyset)=0$ and $p(S) > -\infty$, and $C_{0} = \emptyset$.
For $j=1,2,\ldots, q$,	define
\begin{align}
\beta_{j} &= \max \left\{  \left\lceil  \frac{p(X \cup C_{j-1}) - p(C_{j-1})}{|X|}  \right\rceil :  
         \emptyset \not= X \subseteq \overline{C_{j-1}}   \right\},
\label{betajdef}
\\
h_{j}(X) &= p(X \cup C_{j-1}) - (\beta_{j} - 1) |X| - p(C_{j-1})    
\qquad
(X \subseteq \overline{C_{j-1}}),
\label{hjdef}
\\
S_{j} &=  \mbox{smallest subset of $\overline{C_{j-1}}$ maximizing $h_{j}$},
\label{Sjdef}
\\
C_{j} &=  C_{j-1} \cup S_{j} ,
\label{Cjdef}
\end{align}
where $\overline{C_{j-1}} = S  - C_{j-1}$ and
 the index $q$ is determined by the condition that
$C_{q-1} \not= S$ and $C_{q} = S$.

According to the above definitions, we have that
\begin{equation} \label{CjsmallestmaxCj1}
\mbox{$C_{j}$ is the smallest maximizer of 
$p(X) - (\beta_{j}-1) |X|$
among all $Z \supseteq C_{j-1}$}.
\end{equation}
We will show in Proposition \ref{PRbasicitera} below 
that  $C_{j}$ is, in fact,  the smallest maximizer of 
$p(X) - (\beta_{j}-1) |X|$
among all subsets $X$ of $S$.

For any integer $\beta$,
let $\mathcal{L}(\beta)$ denote the family of all maximizers of $p(X) - \beta |X|$,
and  $L(\beta)$ be the smallest element of $\mathcal{L}(\beta)$,
where  the smallest element exists in $\mathcal{L}(\beta)$ since 
$\mathcal{L}(\beta)$ is a lattice (ring family).
(These notations are consistent with the ones introduced in Section~\ref{SCprinpatR}.) \

\begin{proposition} \label{PRbasicitera}
\quad

\noindent
{\rm (1)}
 $\beta_{1} > \beta_{2} > \cdots > \beta_{q}$.

\noindent
{\rm (2)}
For each $j$ with $1 \leq j \leq q$, $C_{j}$ is 
the smallest maximizer of $p(X) - (\beta_{j}-1) |X|$
among all subsets $X$ of $S$.
\end{proposition} 
\begin{proof}
(1)
The monotonicity of the $\beta$-values 
is already shown in Section~\RefPartI{5} of Part~I \cite{FM18part1},
but we give an alternative proof here.
Let $j \geq 2$.
By (\ref{betajdef}), we have
$\beta_{j-1} > \beta_{j}$ 
if and only if
\begin{equation} \label{prf1betaj1betaj}
\beta_{j-1} > \left\lceil \frac{p(X \cup C_{j-1}) - p(C_{j-1})}{|X|} \right\rceil 
\end{equation}
for every $X$ with $\emptyset \not= X \subseteq \overline{C_{j-1}}$.
Furthermore,
\begin{align*}
\eqref{prf1betaj1betaj}
& \iff
\beta_{j-1} -1  \geq   \frac{p(X \cup C_{j-1}) - p(C_{j-1})}{|X|}  
\\ & \iff
p(X \cup C_{j-1}) - p(C_{j-1}) \leq (\beta_{j-1} -1 ) |X|
\\ & \iff
p(X \cup C_{j-1}) - (\beta_{j-1} -1 ) |X  \cup C_{j-1}| \leq  p(C_{j-1}) - (\beta_{j-1} -1 ) |C_{j-1}|.
\end{align*}
The last inequality holds, since
the set $X \cup C_{j-1}$ contains $C_{j-2}$,
whereas $C_{j-1}$ is the (smallest) maximizer of 
$p(X) - (\beta_{j-1}-1) |X|$
among all $X$ containing $C_{j-2}$.
We have thus shown $\beta_{j-1} > \beta_{j}$.

(2)
We prove $C_{j}= L(\beta_{j} -1)$ for $j=1,2,\ldots,q$ by induction on $j$. 
This holds for $j=1$ by definition.
Let $j \geq 2$.
By Proposition \ref{PRbasicPrinPatR}
for $\lambda = \beta_{j-1}-1$ and $\lambda' = \beta_{j}-1$,
the smallest maximizer of 
$p(X) - (\beta_{j}-1) |X|$
is a superset of $L(\beta_{j-1} -1)$,
where $L(\beta_{j-1} -1) =C_{j-1}$
by the induction hypothesis.
Combining this with \eqref{CjsmallestmaxCj1}, we obtain $C_{j}= L(\beta_{j} -1)$.
\end{proof}

We now give an alternative characterization of the essential value-sequence 
$\beta_{1} > \beta_{2} > \cdots > \beta_{q}$
defined by (\ref{betajdef})--(\ref{Cjdef}).
We consider the family $\{ L(\beta) : \beta \in \ZZ \}$
of the smallest maximizers of $p(X) - \beta |X|$ for all integers $\beta$.
Each $C_{j}$ is a member of this family, since
$C_{j} = L(\beta_{j} - 1)$ ($j=1,2,\ldots, q$)
by Proposition~\ref{PRbasicitera}(2).

\begin{proposition} \label{PRsequenceLbeta}
As $\beta$ is decreased from $+\infty$ to $-\infty$
(or from $\beta_{1}$ to $\beta_{q}-1$), 
the smallest maximizer $L(\beta)$
is monotone nondecreasing.
We have $L(\beta) \not= L(\beta -1)$ if and only if
$\beta$ is equal to an essential value.
Therefore, the essential value-sequence 
$\beta_{1} > \beta_{2} > \cdots > \beta_{q}$
is characterized by the property%
\footnote{%%%%%%%%%%%%%%%%%%
Recall that \ ``$\subset$'' \  means \ ``$\subseteq$ \  and \ $\not=$.''
} %%%%%%%%%% footnote %%%%%%%%
\begin{align}
\emptyset= L(\beta_{1}) \subset L(\beta_{1}-1) 
= \cdots = L(\beta_{2}) \subset L(\beta_{2}-1)
= \cdots = L(\beta_{q}) \subset L(\beta_{q}-1) = S.
\end{align}
\end{proposition} 
\begin{proof}
The monotonicity of $L(\beta)$  follows from Proposition \ref{PRbasicPrinPatR}.
We will show 
(i) $L(\beta_{1}) = \emptyset$,
(ii) $L(\beta_{j-1} -1) = L(\beta_{j})$
for $j=2,\ldots, q$, and
(iii) $L(\beta_{j}) \subset L(\beta_{j}-1)$ for $j=1,2,\ldots, q$.

(i)
Since 
$\beta_{1} = \max \left\{  \left\lceil  p(X)/|X|  \right\rceil  :  X  \not = \emptyset  \right\}$,
we have
$p(X) - \beta_{1} |X| \leq  0$ for all $X \not= \emptyset$,
whereas 
$p(X) - \beta_{1} |X| =  0$ for $X = \emptyset$.
Therefore, $L(\beta_{1}) = \emptyset$.

(ii)
Let $2 \leq j \leq q$.
For short we write $C = C_{j-1}$.
Define $h(Y) = p(Y) - \beta_{j} |Y|$ for any subset $Y$ of $S$,
and let $A$ be the smallest  maximizer of $h$, which means 
$A = L(\beta_{j})$.
For any nonempty subset $X$ of $\overline{C} \ (= S - C)$
we have
\begin{align*}
& \beta_{j}
\geq    \left\lceil  \frac{p(X \cup C) - p(C)}{|X|}  \right\rceil 
\geq   \frac{p(X \cup C) - p(C)}{|X|} ,
\end{align*}
which implies
$p(X \cup C) - \beta_{j} |X  \cup C| \leq  p(C) - \beta_{j} |C|$,
that is,
\begin{align}
h(Y) \leq  h(C)
\qquad \mbox{for all \  $Y \supseteq C$}.
\label{prfhYhC}
\end{align}
By supermodularity of $p$ we have
$ h(A)  + h(C) \leq  h(A \cup C)  + h(A \cap C)$, 
whereas $h(C) \geq  h(A \cup C)$
by (\ref{prfhYhC}).
Therefore,
$ h(A)  \leq  h(A \cap C)$.
Since $A$ is the smallest maximizer of $h$,
this implies that $A = A \cap C$, i.e., $A \subseteq C$.
Recalling
$A = L(\beta_{j})$ and $C = C_{j-1}= L(\beta_{j-1} -1)$,
we obtain 
$L(\beta_{j}) \subseteq L(\beta_{j-1} -1)$.
We also have
$L(\beta_{j}) \supseteq L(\beta_{j-1} -1)$
by the monotonicity.
Therefore,
$L(\beta_{j}) = L(\beta_{j-1} -1)$.

(iii)
Let $1 \leq j \leq q$.
We continue to write $C = C_{j-1}$.
Take a nonempty subset $Z$ of $\overline{C}$ which
gives the maximum in the definition of $\beta_{j}$, i.e., 
\begin{align*}
& \beta_{j}
= \max \left\{  \left\lceil  \frac{p(X \cup C) - p(C)}{|X|}  \right\rceil    
 :  \emptyset \not= X \subseteq \overline{C}   \right\} 
 = \left\lceil  \frac{p(Z \cup C) - p(C)}{|Z|}  \right\rceil .
\end{align*}
Then we have
\begin{align*}
\frac{p(Z \cup C) - p(C)}{|Z|}  > \beta_{j} -1 ,
\end{align*}
which implies
\begin{align*}
p(Z \cup C) - (\beta_{j}-1) |Z \cup C| >  p(C) -  (\beta_{j}-1) |C|.
\end{align*}
This shows that 
$C = C_{j-1}= L(\beta_{j-1} -1)$
is not a maximizer of $p(Y) - (\beta_{j}-1) |Y|$,
and hence  $L(\beta_{j-1} -1) \not= L(\beta_{j}-1)$.
On the other hand, we have 
$L(\beta_{j-1} -1) = L(\beta_{j})$ by (ii) and 
$L(\beta_{j}) \subseteq L(\beta_{j}-1)$
by the monotonicity in Proposition \ref{PRbasicPrinPatR}.
Therefore,
$L(\beta_{j}) \subset L(\beta_{j}-1)$.
\end{proof}

Proposition \ref{PRsequenceLbeta} justifies the following alternative definition
of the essential value-sequence, the canonical chain, and the canonical partition:

\begin{quote}
Consider the smallest maximizer $L(\beta)$ of $p(X) - \beta |X|$
 for all integers $\beta$.
There are finitely many $\beta$ for which
$L(\beta) \not= L(\beta -1)$.
Denote such integers as
$\beta_{1} > \beta_{2} > \cdots > \beta_{q}$
and call them the {\bf essential value-sequence}.
Furthermore, define 
$C_{j} = L(\beta_{j}-1)$
for $j=1,2,\ldots, q$
to obtain a chain: 
$C_{1} \subset  C_{2} \subset  \cdots \subset  C_{q}$.
Call this  the {\bf canonical chain}.
Finally define a partition  $\{ S_{1},  S_{2}, \ldots, S_{q} \}$
of $S$ by $S_{j}=C_{j} -  C_{j-1}$ for $j=1,2,\ldots, q$,
where $C_{0} = \emptyset$,
and call this the {\bf canonical partition}. 
\end{quote}

This alternative construction clearly exhibits the parallelism between 
the canonical partition in Case $\ZZ$ and the principal partition in Case $\RR$.
In particular, the essential value-sequence is exactly
the discrete counterpart of the critical values.
This is discussed in the next section.

\subsection{Canonical partition from the principal partition}
\label{SCcanopatprinpat}

The characterization of the canonical partition shown in Section~\ref{SCcanopat}
enables us to obtain 
the canonical partition for Case $\ZZ$ 
from the principal partition for Case $\RR$ as follows.

\begin{theorem} \label{THrelRZpartition}
\quad

\noindent
{\rm (1)}
An integer $\beta$ is an essential value if and only if there exists 
a critical value $\lambda$ satisfying $\beta  \geq \lambda > \beta -1$.

\noindent
{\rm (2)}
The essential values $\beta_{1} > \beta_{2} > \cdots > \beta_{q}$
are obtained 
from the critical values 
$\lambda_{1} > \lambda_{2} >  \cdots >  \lambda_{r}$
as the distinct members of the rounded-up integers
$\lceil \lambda_{1} \rceil \geq \lceil \lambda_{2} \rceil \geq \cdots \geq \lceil \lambda_{r} \rceil$.
Let $I(j) = \{ i :  \lceil \lambda_{i} \rceil  = \beta_{j} \}$
for $j=1,2,\ldots, q$.

\noindent
{\rm (3)}
The canonical partition 
$\{ S_{1}, S_{2}, \ldots, S_{q} \}$
is obtained from the principal partition 
$\{ \hat S_{1}, \hat S_{2}, \ldots, \hat S_{r} \}$
as an aggregation; it is given as 
\begin{equation} \label{canpataggreprinpat}
 S_{j} = \bigcup_{i \in I(j)} \hat S_{i} 
\qquad (j=1,2,\ldots, q).
\end{equation}

\noindent
{\rm (4)}
The canonical chain
$\{ C_{j} \}$
is a subchain of the principal chain  $\{ \hat C_{i} \}$;
it is given as $C_{j} = \hat C_{i}$ for $i= \max I(j)$.
\finbox
\end{theorem}

In Case $\RR$, the dec-min element $m_{\RR}$ of $B$ is uniform 
on each member $\hat S_{i}$ of the principal partition, i.e., 
$m_{\RR}(s) = \lambda_{i}$ if $s \in  \hat S_{i}$, where $i=1,2,\ldots, r$
(cf., Proposition \ref{PRlexoptbaseR}).
In Case $\ZZ$, the dec-min element $m_{\ZZ}$ of $\odotZ{B}$ is near-uniform 
on each member $S_{j}$ of the canonical partition, i.e., 
$m_{\ZZ}(s) \in \{  \beta_{j}, \beta_{j} -1 \}$ 
if $s \in S_{j}$, where $j=1,2,\ldots,q$
(cf., Theorem \RefPartI{5.1} of Part~I \cite{FM18part1}).
Combining these results with Theorem~\ref{THrelRZpartition} above
we can obtain a (strong) proximity theorem for dec-min elements.

\begin{theorem}[Proximity] \label{THdecminproxS}
Let $m_{\RR}$ be the minimum norm point of $B$.
Then every dec-min element $m_{\ZZ}$ of $\odotZ{B}$ satisfies
$\left\lfloor m_{\RR} \right\rfloor \leq m_{\ZZ} \leq  \left\lceil m_{\RR} \right\rceil$.
\end{theorem} 
\begin{proof}
For $s \in S$ let $\hat S_{i}$ denote the member of the principal partition 
containing $s$, and $\lambda_{i}$ be the associated critical value.
We have $m_{\RR}(s) = \lambda_{i}$ by Proposition \ref{PRlexoptbaseR}.
Let $\beta_{j} = \lceil \lambda_{i} \rceil$.
This is an essential value, and the corresponding member $S_{j}$ of the canonical partition
contains the element $s$ by Theorem~\ref{THrelRZpartition}.
We have $m_{\ZZ}(s) \in \{  \beta_{j}, \beta_{j} -1 \}$ 
by Theorem \RefPartI{5.1} of Part~I \cite{FM18part1}.
Therefore, 
$m_{\ZZ} \leq  \left\lceil m_{\RR} \right\rceil$.

Next we apply the above argument to $-B$, which is an integral base-polyhedron.
Since $-m_{\RR}$ is the minimum norm point of $-B$ and 
$-m_{\ZZ}$ is a dec-min (=inc-max) element for $-\odotZ{B}$, we obtain
$-m_{\ZZ} \leq  \left\lceil -m_{\RR} \right\rceil$, which is equivalent to  
$m_{\ZZ} \geq  \left\lfloor  m_{\RR} \right\rfloor$.
\end{proof}

\begin{remark} \rm  \label{RMdecminproxW} 
Theorem~\ref{THdecminproxS} implies a weaker statement that 
\begin{equation} \label{decminproxW}
\mbox{There exists a dec-min element $m_{\ZZ}$ of $\odotZ{B}$ satisfying
$\left\lfloor m_{\RR} \right\rfloor \leq m_{\ZZ} \leq  \left\lceil m_{\RR} \right\rceil$,}
\end{equation}
where $m_{\RR}$ is the minimum norm point of $B$.
This statement (\ref{decminproxW}) 
should not be confused with 
Proposition \ref{PRdecminproxG} in Section~\ref{SCalgRZ},
which is another proximity statement referring 
to a minimizer of the piecewise extension of the quadratic function,
not to the minimum norm point (minimizer of the quadratic function itself). 
\finbox
\end{remark}

%%\memo{Addition, 2019-06-03, after arXiv, Ver.2}

\begin{theorem} \label{THmnormconvcombdmS}
The minimum norm point of $B$
can be represented as a convex combination of 
the dec-min elements of $\odotZ{B}$.
\end{theorem} 
\begin{proof}
On one hand, it was shown in Section \RefPartI{5.1} of Part~I \cite{FM18part1}
that the dec-min elements of $\odotZ{B}$
lie on the face $B\sp{\oplus}$ of $B$ 
defined by the canonical chain 
$C_{1} \subset  C_{2} \subset  \cdots \subset  C_{q}$.
This face is the intersection of $B$ with the hyperplanes 
$\{x\in {\bf R}\sp{S}:  \widetilde x(C_{j}) =  p(C_{j}) \}$
$(j=1,2,\ldots, q)$.
On the other hand, 
it is known (\cite{Fuj80}, \cite[Section 9.2]{Fuj05book}) that
the minimum norm point $m_{\RR}$ of $B$
lies on the face of $B$ 
defined by the principal chain 
$\hat C_{1} \subset  \hat C_{2} \subset  \cdots \subset \hat C_{r}$,
which is the intersection of $B$ with the hyperplanes 
$\{x\in {\bf R}\sp{S}:  \widetilde x(\hat C_{i}) =  p(\hat C_{i}) \}$
$(i=1,2,\ldots, r)$.
Since the principal chain is a refinement of
the canonical chain (Theorem~\ref{THrelRZpartition}),
the latter face is a face of $B\sp{\oplus}$.
Therefore, 
$m_{\RR}$ belongs to $B\sp{\oplus}$. 
The point $m_{\RR}$ also belongs to 
$T\sp{*}= \{x\in {\bf R}\sp{S}:  \ \beta_{j}-1\leq x(s)\leq \beta_{j} 
  \ \ \hbox{whenever}\ s\in S_{j} \ (j=1,2,\dots ,q)\}$,
since
$m_{\RR}(s) = \lambda_{i}$ for $s \in  \hat S_{i}$
(Proposition \ref{PRlexoptbaseR}) and 
$ S_{j}
 = \bigcup \{ \hat S_{i}:   \lceil \lambda_{i} \rceil  = \beta_{j} \}$ 
(Theorem~\ref{THrelRZpartition}).
Therefore,  $m_{\RR}$ is a member of $B\sp{\bullet}= B\sp{\oplus}\cap T\sp{*}$.
By recalling that $B\sp{\bullet}$ is an integral base-polyhedron
whose vertices are precisely the dec-min elements of $\odotZ{B}$,
we conclude that $m_{\RR}$
can be represented as a convex combination of the dec-min elements of $\odotZ{B}$.
\end{proof}

%%\hfill \memo{End of Addition, 2019-06-03, after arXiv, Ver.2}\\

The following two examples illustrate Theorem~\ref{THrelRZpartition}.

\begin{example} \rm \label{EX2dimB1ppRZ}
Let $S = \{ s_1 , s_2 \}$ and $\odotZ{B} = \{ (0,3), (1,2), (2,1) \}$,
where $B$ is the line segment connecting $(0,3)$ and $(2,1)$.
For $\odotZ{B}$ there are two dec-min elements:
$m_{\ZZ}\sp{(1)}=(1,2)$ and $m_{\ZZ}\sp{(2)}=(2,1)$.
The minimum norm point (dec-min element) of $B$ is
$m_{\RR}=(3/2,3/2)$.
The supermodular function $p$ is given by
\[
 p(\emptyset)= 0,  
\quad
 p(\{ s_{1} \})= 0,  
\quad
 p(\{ s_{2} \})= 1,  
\quad
 p(\{ s_{1},s_{2} \}) = 3,
\]
and we have
\begin{equation*} \label{}  %% cases/Cases
 p(X) - \lambda |X| = 
   \left\{  \begin{array}{ll}
    0   & (X = \emptyset),  \\
    -\lambda  & (X = \{ s_{1} \}),  \\
    1-\lambda  & (X = \{ s_{2} \}),  \\
    3- 2\lambda  & (X = \{ s_{1}, s_{2} \}).  \\
             \end{array}  \right.
\end{equation*}
There is only one ($r=1$) critical value $\lambda_{1} = 3/2$ and
the associated sublattice is $\mathcal{L}(\lambda_{1}) = \{ \emptyset, S \}$.
The principal partition is a trivial partition $\{ S \}$.
Since $\lceil \lambda_{1} \rceil = 2$, we have
$\beta_{1} = 2$ with $q=1$, and the (only) member $S_{1}$
in the canonical partition is given by 
$S_{1} = L(\beta_{1}-1)= L(1) = S$.
Accordingly,  the canonical chain consists of only one member ${C}_{1} = S$.
\finbox
\end{example}

\begin{example} \rm \label{EX4dimC4pat}
We consider Example \ref{EX4dimC4} again.
We have $S = \{ s_1, s_2, s_3, s_4  \}$ and
$\odotZ{B}$ consists of five vectors:
$ m_{1}=(2,1,1,0)$, \
$m_{2}=(2,1,0,1)$, \
$m_{3}=(1,2,1,0)$, \
$m_{4}=(1,2,0,1)$, and 
$m_{5}=(2,2,0,0)$,
of which the first four members, $m_{1}$ to $m_{4}$,
are the dec-min elements.
The supermodular function $p$ is given by
\begin{align*}
& p(\emptyset)= 0, \quad
p(\{ s_{1} \})= p(\{ s_{2} \})= 1, \quad  
p(\{ s_{3} \})= p(\{ s_{4} \})= 0,  
\\ & 
p(\{ s_{1},s_{2} \}) = 3, \quad  
p(\{ s_{3},s_{4} \}) = 0, \quad
p(\{ s_{1},s_{3} \}) =p(\{ s_{2},s_{3} \}) =p(\{ s_{1},s_{4} \}) =p(\{ s_{2},s_{4} \}) = 1,
\\ & 
p(\{ s_{1},s_{2}, s_{3} \}) =
p(\{ s_{1},s_{2}, s_{4} \}) = 3, \quad
p(\{ s_{1},s_{3}, s_{4} \}) =
p(\{ s_{2},s_{3}, s_{4} \}) = 2, \quad
\\ & 
p(\{ s_{1},s_{2},s_{3}, s_{4} \}) = 4.
\end{align*}
We have
\begin{equation*} \label{}  %% cases/Cases
 \max\{ p(X) - \lambda |X| : X \subseteq S \} = 
 \max\{ 0, \   1- \lambda, \  3- 2\lambda,  \   3- 3\lambda,  \   4- 4\lambda \}.  
\end{equation*}
There are two ($r=2$) critical values $\lambda_{1} = 3/2$ and $\lambda_{2} = 1/2$,
with the associated sublattices
$\mathcal{L}(\lambda_{1}) = \{ \emptyset, \{ s_{1},s_{2} \} \}$
and
$\mathcal{L}(\lambda_{2}) = \{\{ s_{1},s_{2} \}, S \}$.
The principal chain is given by
$\emptyset \subset \{ s_{1},s_{2} \} \subset  S$,
and the principal partition is a bipartition with 
$\hat S_{1} = \{ s_{1},s_{2} \}$
and
$\hat S_{2} = \{ s_{3},s_{4} \}$.
The minimum norm point of the base-polyhedron $B$ 
is given by $m_{\RR} = (3/2, 3/2, 1/2, 1/2 )$
by Proposition \ref{PRlexoptbaseR}.
Since
$\lceil \lambda_{1} \rceil = 2$ and
$\lceil \lambda_{2} \rceil = 1$, we have
$\beta_{1} = 2$ and 
$\beta_{2} = 1$ with $q=2$.
The canonical chain consists of two members 
$C_{1} = L(\beta_{1}-1)= L(1) = \{ s_{1},s_{2} \}$
and
$C_{2} = L(\beta_{2}-1)= L(0) = S$.
Accordingly,  the canonical partition is given by 
$S_{1} = \{ s_{1},s_{2} \}$
and
$S_{2} = \{ s_{3},s_{4} \}$.
\finbox
\end{example}

\subsection{Continuous relaxation algorithms}
\label{SCalgRZ}

In Section~\RefPartI{7} of Part~I \cite{FM18part1}, 
we have presented a strongly polynomial algorithm  
for finding a dec-min element of $\odotZ{B}$ as well as 
for finding the canonical partition.
This is based on an iterative approach
to construct a dec-min element along the canonical chain.

By making use of the relation between Case $\RR$ and Case $\ZZ$,
we can construct continuous relaxation algorithms,
which first compute a real (fractional) vector that is guaranteed to be close 
to an integral dec-min element,
and then find the integral dec-min element by solving 
a linearly weighted matroid optimization problem.

In our continuous relaxation algorithms,
we first apply some algorithm for Case $\RR$
to find two integer vectors 
$\ell$ and $u$ such that 
$\bm{0} \leq u - \ell \leq \bm{1}$,
(i.e., $0 \leq u(s) - \ell(s) \leq 1$ for all $s \in S$)
and the box $[\ell, u]$
contains at least one dec-min element of $\odotZ{B}$,
i.e., 
\begin{equation}  \label{smallboxdecmin}
 \ell \leq m_{\ZZ} \leq u
\end{equation}
for some dec-min element $m_{\ZZ}$ of $\odotZ{B}$.
We denote the intersection of $\odotZ{B}$ and $[\ell, u]$ by $\odotZ{B_{\ell}\sp{u}}$.
Then the dec-min element of $\odotZ{B_{\ell}\sp{u}}$ is a dec-min element of $\odotZ{B}$.
Since 
$\bm{0} \leq u - \ell \leq \bm{1}$,
$\odotZ{B_{\ell}\sp{u}}$ can be regarded as a matroid translated by $\ell$,
i.e.,
$\odotZ{B_{\ell}\sp{u}} = \{ \ell + \chi_{L} : L \mbox{ is a base of $M$} \}$
for some matroid $M$.
Therefore,
the dec-min element of $\odotZ{B_{\ell}\sp{u}}$ 
can be computed as the minimum weight base  of matroid $M$
with respect to the weight vector $w$ defined by
$w(s) = u(s)\sp{2} - \ell(s)\sp{2}$  ($s \in S$).
By the greedy algorithm we can find the minimum weight base of $M$ in strongly polynomial time.

We can conceive two different algorithms for finding vectors $\ell$ and $u$.

\subsubsection*{(a) Using the minimum norm point}

In Theorem~\ref{THdecminproxS} we have shown
that every dec-min element $m_{\ZZ}$ of $\odotZ{B}$ satisfies
$\left\lfloor m_{\RR} \right\rfloor \leq m_{\ZZ} \leq  \left\lceil m_{\RR} \right\rceil$
for the minimum norm point $m_{\RR}$ of $B$.
Therefore, we can choose
$\ell = \left\lfloor m_{\RR} \right\rfloor$ and $u= \left\lceil m_{\RR} \right\rceil$
in (\ref{smallboxdecmin}).
With this choice of $(\ell, u)$, 
$\odotZ{B_{\ell}\sp{u}}$ contains all dec-min elements of $\odotZ{B}$.
The decomposition algorithm of Fujishige \cite{Fuj80}
(see also \cite[Section 8.2]{Fuj05book})
finds the minimum norm point $m_{\RR}$ in strongly polynomial time.
Therefore, the continuous relaxation algorithm 
using the minimum norm point
is a strongly polynomial algorithm.

\begin{example} \rm \label{EX4dimC4relax}
We continue with Example \ref{EX4dimC4pat}, where
$\odotZ{B}$ consists of five vectors:
$ m_{1}=(2,1,1,0)$, \
$m_{2}=(2,1,0,1)$, \
$m_{3}=(1,2,1,0)$, \
$m_{4}=(1,2,0,1)$, and 
$m_{5}=(2,2,0,0)$.
From the minimum norm point  $m_{\RR} = (3/2, 3/2, 1/2, 1/2 )$,
we obtain $\ell = (1,1,0,0)$ and $u=(2,2,1,1)$, and hence $w = (3,3,1,1)$. 
Since 
$W(m_{i})=10$ for $i=1,\ldots,4$ and $W(m_{5})=12$,
the dec-min elements are given by $m_{1}$ to $m_{4}$.
\finbox
\end{example}

\subsubsection*{(b) Using the piecewise-linear extension}

The algorithm of Groenevelt \cite{Gro91}
(see also \cite[Section 8.3]{Fuj05book})
employs a piecewise-linear extension of the objective function.
For the quadratic function
$\varphi (k)=k\sp{2}$,
the piecewise-linear extension $\overline{\varphi}: \RR \to \RR$
is given by:
$\overline{\varphi}(t)
 = (2k-1) t - k (k-1)$
if $k-1 \leq |t| \leq k$ for $k \in \ZZ$.

The following proximity property is a special case of 
an observation of Groenevelt \cite{Gro91}
(see also \cite[Theorem 8.3]{Fuj05book}).
%% 2019-05-24

\begin{proposition}[Groenevelt \cite{Gro91}] \rm \label{PRdecminproxG}
For any minimizer $\overline{m}_{\RR} \in \RR\sp{S}$ of the function
$\overline{\Phi}(x) = \sum_{s \in S} \overline{\varphi} ( x(s) )$ over $B$,
there exists a minimizer $m_{\ZZ} \in \ZZ\sp{S}$ of 
$\Phi(x) = \sum_{s \in S} x(s)\sp{2}$ 
over $\odotZ{B}$ satisfying
$\lfloor \overline{m}_{\RR} \rfloor \leq m_{\ZZ}  \leq  \lceil \overline{m}_{\RR} \rceil$.
\end{proposition}
\begin{proof}
(We give a proof for completeness, though it is easy and standard.)
By the integrality of $B$,
we can express $\overline{m}_{\RR}$ as a convex combination of
integral member $z_{1}, z_{2}, \ldots, z_{k}$ of $B$
satisfying
$\lfloor \overline{m}_{\RR} \rfloor \leq z_{i}  \leq  \lceil \, \overline{m}_{\RR} \rceil$
$(i=1,2,\ldots,k)$, where
$\overline{m}_{\RR} = \sum_{i=1}\sp{k} \lambda_{i} z_{i}$
with
$\sum_{i=1}\sp{k} \lambda_{i} =1$
and $\lambda_{i} > 0 $ $(i=1,2,\ldots,k)$.
Since 
$\overline{\Phi}$ is piecewise-linear,
we have
$\overline{\Phi}(\overline{m}_{\RR}) = 
\sum_{i=1}\sp{k} \lambda_{i} \Phi(z_{i})$,
in which
$\Phi(z_{i}) = \overline{\Phi}(z_{i}) \geq \overline{\Phi}(\overline{m}_{\RR})$.
Therefore, 
$z_{1}, z_{2}, \ldots, z_{k}$ are the minimizers of $\Phi$ on $\odotZ{B}$.
We can take any $z_{i}$ as  $m_{\ZZ}$.
\end{proof}

By Proposition~\ref{PRdecminproxG} we can take
$\ell = \lfloor \overline{m}_{\RR} \rfloor$ and $u = \lceil \overline{m}_{\RR} \rceil$
in (\ref{smallboxdecmin}).
In this case, however,
$\odotZ{B_{\ell}\sp{u}}$ may not contain all dec-min elements of $\odotZ{B}$.
The complexity of computing $\overline{m}_{\RR}$ 
is not fully analyzed in the literature 
\cite{Fuj05book,Gro91,KSI13}.
See also  Remark \ref{RMdecminproxW}.

\begin{remark}\rm  \label{RMcontrelalgKSI} 
Minimization of a separable convex function on a base-polyhedron 
has been investigated in the literature of resource allocation
under the name of ``resource allocation problems under submodular constraints''
(Hochbaum \cite{Hoc07}, Ibaraki--Katoh \cite{IK88}, Katoh--Ibaraki \cite{KI98},
Katoh--Shioura--Ibaraki \cite{KSI13}).
The continuous relaxation approach for the case of discrete variables
is considered, e.g., by Hochbaum \cite{Hoc94} and Hochbaum--Hong \cite{HH95}.
A more recent paper by Moriguchi--Shioura--Tsuchimura \cite{MST11Mrelax}
discusses this approach in a more general context of M-convex function minimization
in discrete convex analysis.
It is known (\cite{HH95,MST11Mrelax}, \cite[Theorem 23]{KSI13})
that a convex quadratic function
$\sum a_{i} x_{i}\sp{2}$ 
in discrete variables can be minimized over an integral base-polyhedron
in strongly polynomial time
if the base-polyhedron has a special structure
like ``Nested'', ``Tree,'' or ``Network''
in the terminology of \cite{KSI13}.
\finbox
\end{remark}

%%% end file %%%

%% murota 2019-03-03 / 2019-05-22 / 2019-07-09 / 2019-08-18 / 2019-10-04 / 2020-06-30

\section{Min-max formulas for separable convex functions in DCA}
\label{SCdcaflowM2}

The objective of this section is to pave the way of DCA approach to 
discrete decreasing minimization on other discrete structures
such as the intersection of M-convex sets, network flows, submodular flows.
%%that we consider in Parts III and IV \cite{FM18part3,FM18part5}.
Min-max formulas for separable convex functions 
on the intersection of M-convex sets
and ordinary/submodular flows are presented.
%% in a way suitable for their use in Parts III and IV.
%% 2020-06-30

In Section \ref{SCfencsepar} we have considered  the min-max formula
\begin{equation} \label{minmaxgensep-NI}
 \min \{  \sum_{s \in S} \varphi_{s} ( x(s) ) :  x \in \odotZ{B}  \}
 = \max\{ \hat p(\pi) - \sum_{s \in S} \psi_{s}(\pi(s)) :  \pi \in \ZZ\sp{S} \} 
\end{equation}
for a separable convex function on an M-convex set.
Here, $p$ is an integer-valued (fully) supermodular function on $S$,
$B$ is the base-polyhedron defined by $p$,
$\odotZ{B}$ is the set of integral points of $B$,  and 
$\hat p$ is the linear extension (Lov{\'a}sz extension) of $p$.
For each $s \in S$,
$\varphi_{s}: \ZZ \to \ZZ \cup \{ +\infty \}$
is an integer-valued (discrete) convex function 
and $\psi_{s}$ is the conjugate function of $\varphi_{s}$.
Furthermore, the sets of primal and dual optimal solutions 
of (\ref{minmaxgensep-NI}) are described in Section \ref{SCfencoptsetsepar}.
These results have been used for the DCA-based proofs of some key results 
on decreasing minimization on an M-convex set
in Sections \ref{SCderminmaxsquaresum}, \ref{SCsqsumoptdual}, and \ref{SCsqsumoptprimal}.

The min-max formula \eqref{minmaxSqSum-KM} for the square-sum has been obtained as 
a special case of (\ref{minmaxgensep-NI}) where the conjugate functions
can be given explicitly.
To emphasize  the role of explicit forms of conjugate functions,
we offer in Section \ref{SCexplicitconj} several examples of (discrete) convex functions that admit
explicit expressions of conjugate functions.
These worked-out examples of conjugate functions
and min-max formulas 
will hopefully trigger other applications of discrete convex analysis.

\subsection{Examples of explicit conjugate functions}
\label{SCexplicitconj}

In this section we offer several examples of (discrete) convex functions 
whose conjugate functions can be given explicitly.
An explicit representation of the conjugate function
renders an easily checkable certificate of optimality 
in the min-max formulas such as (\ref{minmaxgensep-NI}).

For an integer-valued discrete convex function
$\varphi: \ZZ \to \ZZ \cup \{ +\infty \}$,
we denote its conjugate function 
$\varphi\sp{\bullet}$ by $\psi$.
That is, function 
$\psi: \ZZ \to \ZZ \cup \{ +\infty \}$
is defined by
\begin{equation} \label{phiconjdef-NI}
 \psi(\ell)  = \max\{ k \ell  -  \varphi(k)  :  k \in \ZZ \}
\qquad
(\ell \in \ZZ).
\end{equation}
Obviously, we have 
\begin{equation}  \label{youngineq}
\varphi (k) + \psi ( \ell ) \geq k \ell 
\qquad (k,\ell \in \ZZ),
\end{equation}
which is known as the Fenchel--Young inequality,
and the equality holds in \eqref{youngineq}
if and only if
\begin{equation}  \label{youngsubgrad}
\varphi(k) - \varphi(k-1) \leq \ell \leq \varphi(k+1) - \varphi(k).
\end{equation}
It is worth noting that,  for $a, b, c \in \ZZ$, 
the conjugate of the function 
$\varphi_{a,b,c}(k) = \varphi (k-a) + bk + c$
is given by
$\psi(\ell - b) + a (\ell -b) -c$. With abuse of notation we express this as
\begin{equation} \label{phitransconj}
\big( \varphi (k-a) + bk + c \big)\sp{\bullet} = \psi(\ell - b) + a (\ell -b) -c.
\end{equation}

In what follows we demonstrate 
how to calculate the conjugate functions 
for piecewise-linear functions, $\ell_{1}$-distances, quadratic functions, 
power products, and exponential functions.

\subsubsection{Piecewise-linear functions}

Let $a$ be an integer.
For a piecewise-linear function $\varphi$ defined by
\begin{equation} \label{phipos-NI}
\varphi(k) = (k-a)\sp{+} = \max \{ 0, k-a\}
\qquad (k \in \ZZ),
\end{equation}
the conjugate function $\psi$ is given by
\begin{eqnarray} \label{psipos-NI}
\psi(\ell) &=&
   \left\{  \begin{array}{ll}
    0         &   (\ell = 0) ,     \\
    a        &   (\ell = 1) ,     \\
   +\infty     &   (\ell \not\in  \{ 0,1  \}) .     \\
                     \end{array}  \right.
\end{eqnarray}
This explicit form can be used in the DCA-based proof of Theorem \ref{THtotalexcess}; 
see Remark \ref{RMdcaprftotalex}.

For another piecewise-linear function $\varphi$ defined by
\begin{equation} \label{phipiecelin2-NI}
\varphi(k) = 
   \left\{  \begin{array}{ll}
    0         &   (0 \leq k \leq a) ,     \\
    \lambda (k - a)        &   (a \leq k \leq b) , \\
   +\infty     &   (k \leq -1 \mbox{ or } k \geq b+1)      \\
                     \end{array}  \right.
\end{equation}
for $a, b, \lambda \in \ZZ$ with  $0 \leq a \leq b$ and $\lambda \geq 0$,
the conjugate function $\psi$ is given by
\begin{eqnarray} \label{psipiecelin2-NI}
\psi(\ell) &=&
   \left\{  \begin{array}{ll}
    0         &   (\ell \leq 0) ,     \\
    a \ell         &   (0 \leq \ell \leq \lambda ) ,     \\
    b \ell - (b - a) \lambda   &   (\ell \geq \lambda ) .     \\
                     \end{array}  \right.
\end{eqnarray}

\subsubsection{$\ell_{1}$-distances}

Let $a$ be an integer.
For function $\varphi$ defined by
\begin{equation} \label{phi=dist1}
\varphi(k) = | k  - a | 
\qquad (k \in \ZZ),
\end{equation}
the conjugate function $\psi$ is given by
\begin{equation} \label{conj=dist1}
\psi( \ell ) =
\left\{
\begin{array}{ll}
a \ell  & (\ell=-1,0,+1) , \\
+\infty & (\mbox{otherwise}) .
\end{array}
\right.
\end{equation}

A min-max relation for the minimum $\ell_{1}$-distance between an integer point of $B$ 
and a given integer point $c$
can be obtained from the min-max formula (\ref{minmaxgensep-NI}). 
Recall that $p$ and $b$ are, respectively, the supermodular and submodular functions 
associated with $B$, and our convention
$\widetilde c(X) =  \sum\{c(s): s \in X \}$.

\begin{proposition} \label{PRminmaxL1cent}
For $c \in \ZZ\sp{S}$,
\begin{align} \label{minmaxL1cent}
 & \min \{ \sum_{s \in S} |x(s) - c(s)|  : x \in \odotZ{B}  \}
\nonumber \\
 &= \max\{ p(X) - b(Y) 
- \widetilde c(X) + \widetilde c(Y) 
   :   X, Y \subseteq S; \  X \cap Y = \emptyset \} .
\end{align}
\end{proposition}

\begin{proof}
We choose $\varphi_{s}(k) = | k  - c(s)|$ in (\ref{minmaxgensep-NI}).
By \eqref{conj=dist1}, 
we may assume 
$\pi \in \{ -1,0,+1 \}\sp{S}$
on the right-hand side of (\ref{minmaxgensep-NI}).
On representing $\pi = \chi_{X} - \chi_{Y}$ with disjoint subsets $X$ and $Y$,
we obtain $\hat p(\pi) = p(X) - b(Y)$ and 
$\sum_{s \in S} \psi_{s} (\pi(s)) = \widetilde c(X) - \widetilde c(Y)$.
Therefore the right-hand side of 
(\ref{minmaxgensep-NI}) coincides with that of (\ref{minmaxL1cent}).
\end{proof}

Let $a$ and $b$ be integers with $a \leq b$,
and define $\varphi$ by
\begin{equation} \label{phi=dist2}
\varphi (k)
 = \min \{ | k - z | : a \leq z \leq b \}
 = \max \{ a - k, 0, k - b  \}
\qquad (k \in \ZZ).
\end{equation}
This function represents the distance from an integer $k$ 
to the integer interval $[a,b]_{\ZZ} := \{ z \in \ZZ : a \leq z  \leq b \}$.
The conjugate function $\psi$ is given by
\begin{equation} \label{conj=dist2}
\psi( \ell ) =
\left\{
\begin{array}{ll}
- a  & (\ell=-1), \\
0   & (\ell=0) ,\\
  b   & (\ell=+1), \\
+\infty & (\mbox{otherwise}).
\end{array}
\right.
\end{equation}

A min-max relation for the minimum $\ell_{1}$-distance between an integer point of $B$ 
and a given integer interval
$[c,d]_{\ZZ} := \{ y \in \ZZ\sp{S} : c(s) \leq y(s)  \leq d(s) \ (s \in S) \}$
can be obtained from the min-max formula (\ref{minmaxgensep-NI}),
where $c, d \in \ZZ\sp{S}$ and $c \leq d$.

\begin{proposition} \label{PRminmax=dist2}
For $c, d \in \ZZ\sp{S}$ with $c \leq d$,
\begin{align} \label{minmaxL1intvl}
 & \min \{ \| x - y \|_{1} : x \in \odotZ{B}, \ y \in [c,d]_{\ZZ}  \}
\nonumber \\ &
= \max\{ p(X) - b(Y) 
- \widetilde d(X) + \widetilde c(Y) 
   :   X, Y \subseteq S; \  X \cap Y = \emptyset \} .
\end{align}
\end{proposition}
\begin{proof}
With reference to  \eqref{phi=dist2}, 
we define $\varphi_{s}(k) = \min \{ | k - z | : c(d) \leq z \leq d(s) \}$.
Then
\begin{align*}
& \min \{ \ \| x - y \|_{1} : x \in \odotZ{B}, \ y \in [c,d]_{\ZZ}  \}
\\ & = 
 \min \{ \ \sum_{s \in S} | x(s) - y(s) |  : x \in \odotZ{B}, \ y \in [c,d]_{\ZZ}  \}
\\ & = 
 \min \{ \ \min \{ \sum_{s \in S}  | x(s) - y(s) | : c(d) \leq y(s) \leq d(s) \ (s \in S) \} 
         : x \in \odotZ{B}  \}\\ & = 
 \min \{ \ \sum_{s \in S} 
      \min \{ \ | x(s) - y(s) | : c(d) \leq y(s) \leq d(s) \} : x \in \odotZ{B}  \}
\\ & = 
 \min \{ \ \sum_{s \in S} \varphi_{s}(x(s)) : x \in \odotZ{B}  \}.
\end{align*}
Thus the left-hand side of \eqref{minmaxL1intvl} is in the form of
the left-hand side of the  min-max formula (\ref{minmaxgensep-NI}).
By \eqref{conj=dist2}, 
we may assume 
$\pi \in \{ -1,0,+1 \}\sp{S}$
on the right-hand side of (\ref{minmaxgensep-NI}).
On representing $\pi = \chi_{X} - \chi_{Y}$ with disjoint subsets $X$ and $Y$,
we obtain $\hat p(\pi) = p(X) - b(Y)$ and 
$\sum_{s \in S} \psi_{s} (\pi(s))
= \widetilde d(X) - \widetilde c(Y)$.
Therefore the right-hand side of 
(\ref{minmaxgensep-NI}) coincides with that of (\ref{minmaxL1intvl}).
\end{proof}

\subsubsection{Quadratic functions}

For a quadratic function $\varphi$ defined by
\begin{equation} \label{phiquadgen}
\varphi(k)= a k\sp{2}
\qquad (k \in \ZZ)
\end{equation}
with a positive integer $a$, 
the conjugate function $\psi$
is given (cf., Remark \ref{RMsquareconjWtd}) by 
\begin{equation} \label{squareconjWtd1}
 \psi(\ell)  =
\left\lfloor \frac{1}{2} \left( \frac{\ell}{a} + 1 \right) \right\rfloor 
\left( 
 \ell - a \left\lfloor \frac{1}{2} \left( \frac{\ell}{a} + 1 \right) \right\rfloor 
\right), \ 
\end{equation}
which admits the following alternative expressions:
\begin{align}
 \psi(\ell) & =
\left\lceil \frac{1}{2} \left( \frac{\ell}{a} - 1 \right) \right\rceil 
\left( 
 \ell - a \left\lceil \frac{1}{2} \left( \frac{\ell}{a} - 1 \right) \right\rceil 
\right), \ 
\label{squareconjWtd2}
\\
 \psi(\ell)  
 &= \max \left\{  \  
\left\lfloor \frac{\ell}{2 a} \right\rfloor 
\left( 
 \ell - a \left\lfloor \frac{\ell}{2 a} \right\rfloor 
\right), \ \  
\left\lceil \frac{\ell}{2 a } \right\rceil  
\left( 
 \ell - a \left\lceil \frac{\ell}{2 a} \right\rceil 
\right) \ 
\right\}.
\label{squareconjWtd3}
\end{align}
If $a=1$, these expressions reduce to 
$\psi(\ell) = \left\lfloor {\ell}/{2} \right\rfloor \cdot \left\lceil {\ell}/{2} \right\rceil$
in \eqref{squareconj}.

The min-max formula (\ref{minmaxSqSum-KM}) for the square-sum
can be extended for a nonsymmetric quadratic function
$\sum_{s \in S} c(s) x(s)\sp{2}$,
where $c(s)$ is a positive integer for each $s \in S$.

\begin{theorem} \label{THminmaxSqSumWtd}
For an integer vector $c \in \ZZ\sp{S}$ with $c(s) \geq 1$ for every $s \in S$,
\begin{align}  
& \min \{ \sum_{s \in S} c(s) x(s)\sp{2} : x \in \odotZ{B}  \}
\notag \\ &
= \max\{ \hat p(\pi) 
 - \sum_{s \in S}
\left\lfloor \frac{1}{2} \left( \frac{\pi(s)}{c(s) } + 1 \right) \right\rfloor 
\left( 
 \pi(s) - c(s)  \left\lfloor \frac{1}{2} \left( \frac{\pi(s)}{c(s) } + 1 \right) \right\rfloor 
\right) :
\pi \in \ZZ\sp{S} \} .
\label{minmaxSqSumWtd}
\end{align}
\vspace{-1.5\baselineskip} \\
\finbox
\end{theorem}

In the basic case where $c(s)=1$ for all $s \in S$,
we had a combinatorial constructive proof in Part~I \cite{FM18part1}.
Such a direct combinatorial proof,
not relying on the Fenchel-type discrete duality in DCA,
 for the general case of \eqref{minmaxSqSumWtd}
would be an interesting topic.

\begin{remark} \rm  \label{RMsquareconjWtd}
We derive \eqref{squareconjWtd1}, \eqref{squareconjWtd2}, and \eqref{squareconjWtd3}.
Since $\varphi$ is discrete convex, 
the maximum in the definition \eqref{phiconjdef-NI} of $\psi(\ell)$
is attained by $k$ satisfying 
\begin{equation} \label{squareconjWtdprf1}
 \varphi(k) - \varphi(k-1) \leq \ell \leq  \varphi(k+1) - \varphi(k) . 
\end{equation}
For $\varphi(k)= a k\sp{2}$ this condition reads
$a (2k-1) \leq \ell \leq  a (2k +1 )$, 
or equivalently
\[ 
 \frac{1}{2} \left( \frac{\ell}{a} - 1 \right)  \leq k \leq 
\frac{1}{2} \left( \frac{\ell}{a} + 1 \right) .
\] 
Therefore, the maximum in \eqref{phiconjdef-NI} is attained by 
$k= \left\lfloor \frac{1}{2} \left( \frac{\ell}{a} + 1 \right) \right\rfloor $
and also by
$k= \left\lceil \frac{1}{2} \left( \frac{\ell}{a} - 1 \right) \right\rceil $.
This gives \eqref{squareconjWtd1} and \eqref{squareconjWtd2}, respectively.

To derive \eqref{squareconjWtd3} 
we consider $\varphi(t)= a t\sp{2}$  in $t \in \RR$
and its derivative $\varphi'(t)= 2a t$.  
Let $k_{\ell}$ be the integer satisfying 
\begin{equation} \label{squareconjWtdprf4}
\varphi'(k_{\ell}) \leq \ell < \varphi'(k_{\ell} +1) .
\end{equation}
Then the maximum in the definition \eqref{phiconjdef-NI} of $\psi(\ell)$
is attained 
by $k=k_{\ell}$
if $\varphi'(k_{\ell}) = \ell$,
and otherwise by $k=k_{\ell}$ or $k_{\ell} +1$.
For $\varphi(k)= a k\sp{2}$,
we have
$k_{\ell}=\left\lfloor \frac{\ell}{2 a} \right\rfloor$
and 
the maximum is attained by
$k = \left\lfloor \frac{\ell}{2 a} \right\rfloor$
or 
$k = \left\lceil \frac{\ell}{2 a } \right\rceil$.  
Hence we have \eqref{squareconjWtd3}.
\finbox
\end{remark}

\subsubsection{Power products}

For function $\varphi$ defined by
\begin{equation} \label{phipower}
\varphi(k)= a \  k\sp{2b}
\qquad (k \in \ZZ)
\end{equation}
with positive integers $a$ and $b$, 
the conjugate function $\psi$
is given (cf., Remark \ref{RMpowerconj}) by 
\begin{align} \label{powerconj}
 \psi(\ell)  
 &= \max \left\{  \  
\ell
\left\lfloor K(\ell) \right\rfloor 
- a \left\lfloor K(\ell) \right\rfloor \sp{2b} 
, \ \ 
\ell
\left\lceil K(\ell) \right\rceil
- a \left\lceil K(\ell) \right\rceil \sp{2b} 
\right\},
\end{align}
where 
\[
 K(\ell) =  \left( \frac{\ell}{2 a b} \right)\sp{1/(2b-1)}.
\]
By choosing $a=1$ and $b=2$, for example, we 
obtain a min-max formula
\begin{align}  
& \min \{ \sum_{s \in S}  x(s)\sp{4} : x \in \odotZ{B}  \}
\notag \\ &
= \max\{ \hat p(\pi) 
 - \sum_{s \in S}
\max \left\{  \  
\pi(s)
\left\lfloor ( \pi(s)/4 )\sp{1/3} \right\rfloor 
-  \left\lfloor ( \pi(s)/4 )\sp{1/3} \right\rfloor \sp{4} 
, \ \
\right.
\notag \\ &
\phantom{AAAAAAAAAAAAAAA}
\left.
\pi(s)
\left\lceil ( \pi(s)/4 )\sp{1/3} \right\rceil
-  \left\lceil ( \pi(s)/4 )\sp{1/3} \right\rceil \sp{4} 
\right\}
: \pi \in \ZZ\sp{S} \} .
\label{minmaxQuartSum}
\end{align}

\begin{remark} \rm  \label{RMpowerconj}
We derive \eqref{powerconj} on the basis of \eqref{squareconjWtdprf4}
for $\varphi(t)= a \, t\sp{2b}$
and  $\varphi'(t)= 2ab \,  t\sp{2b-1}$.  
We have
$k_{\ell}=
\left\lfloor \left( \frac{\ell}{2 a b} \right)\sp{1/(2b-1)} \right\rfloor$,
and the maximum in \eqref{phiconjdef-NI} is attained by
$k = \left\lfloor \left( \frac{\ell}{2 a b} \right)\sp{1/(2b-1)} \right\rfloor$
or 
$k = \left\lceil \left( \frac{\ell}{2 a b} \right)\sp{1/(2b-1)} \right\rceil$.  
Hence follows \eqref{powerconj}.
\finbox
\end{remark}

\subsubsection{Exponential functions}

For an exponential function $\varphi$ defined by
\begin{equation} \label{phiexpon2}
\varphi(k)= 
\left\{
\begin{array}{ll}
2\sp{k}  & (k \geq 0) , \\
+\infty & (\mbox{otherwise}) ,
\end{array}
\right.
\end{equation}
the conjugate function $\psi$
is given (cf., Remark \ref{RMexponconj}) by 
\begin{equation} \label{expon2conj}
 \psi(\ell)  = \ell \left\lceil \log_{2} \ell \right\rceil
 - 2 \sp{ \lceil \log_{2}  \ell \rceil } .
\end{equation}
More generally, for function $\varphi$ defined by
\begin{equation} \label{phiexpongen}
\varphi(k)= 
\left\{
\begin{array}{ll}
a\, b\sp{k}  & (k \geq 0) , \\
+\infty & (\mbox{otherwise}) 
\end{array}
\right.
\end{equation}
with integers $a \geq 1$ and $b \geq 2$, 
the conjugate function $\psi$
is given (cf., Remark \ref{RMexponconj}) by 
\begin{equation} \label{exponconjgen}
 \psi(\ell)  = 
  \ell  \left\lceil \log_{b} \left( \frac{\ell}{a(b-1)} \right) \right\rceil
 - a\, b\sp{  \left\lceil \log_{b} \left( \frac{\ell}{a(b-1)} \right) \right\rceil  } .
\end{equation}

\begin{theorem} \label{THminmax=expon}
Assume that $B$ is contained in the nonnegative orthant $\ZZ_{+}\sp{S}$.
Then
\begin{align}  
& \min \{ \sum_{s \in S} 2\sp{x(s)} : x \in \odotZ{B}  \}
\notag \\ &
= \max\{ \hat p(\pi) 
 - \sum_{s \in S}
 \left(   \pi(s)  \left\lceil \log_{2}  \pi(s)  \right\rceil
 - 2 \sp{  \left\lceil \log_{2}  \pi(s) \right\rceil  } 
  \right)
: \pi \in \ZZ\sp{S} \} .
\label{minmax=expon2}
\end{align}
More generally, for an integer vector $c, d \in \ZZ\sp{S}$ with 
$c(s) \geq 1$, $d(s) \geq 2$ $(s \in S)$,
\begin{align}  
& \min \{ \sum_{s \in S} c(s) \, d(s)\sp{x(s)} : x \in \odotZ{B}  \}
\notag \\ &
= \max\{ \hat p(\pi) 
 - \sum_{s \in S} \left(
  \pi(s)  \left\lceil   K(\ell)  \right\rceil
 - c(s) \, d(s)\sp{  \left\lceil K(\ell) \right\rceil  } 
                  \right)
 : \pi \in \ZZ\sp{S} \} ,
\label{minmax=expongen}
\end{align}
where
\[
 K(\ell) =    \log_{d(s)} \left(  \frac{\pi(s)}{c(s)(d(s)-1)}   \right) .
\]
\vspace{-1.5\baselineskip} \\
\finbox
\end{theorem}

\begin{remark} \rm  \label{RMexponconj}
We derive \eqref{exponconjgen} on the basis of \eqref{squareconjWtdprf1}.
For $\varphi(k)= a \, b\sp{k}$, the condition \eqref{squareconjWtdprf1} reads
$ a (b-1) b\sp{k-1} \leq \ell \leq a (b-1) b\sp{k}$, 
or equivalently
\[ 
 k-1  \ \leq \ 
 \log_{b} \left( \frac{\ell}{a(b-1)} \right) 
 \ \leq \  k .
\] 
Therefore, the maximum in \eqref{phiconjdef-NI} is attained by 
$k=  \left\lceil \log_{b} \left( \frac{\ell}{a(b-1)} \right) \right\rceil$.
Hence follows \eqref{exponconjgen}.
By setting $a=1$ and $b=2$ in \eqref{exponconjgen},
we obtain \eqref{expon2conj}.
\finbox
\end{remark}

%% IIIIIIIIIIIIIIIIIIIIIIIIIIIIIIIIIIIIII

\subsection{Separable convex functions on the intersection of M-convex sets}
\label{SCfencseparInter}

The duality formula \eqref{minmaxgensep-NI} for separable convex functions
on an M-convex set
admits an extension to separable convex functions
on the intersection of two M-convex sets.
%%In  Part~IV \cite{FM18part4} 
This extension serves as a basis of the study of decreasing-minimality 
in the intersection of two M-convex sets (integral base-polyhedra).

Let $B_{1}$ and $B_{2}$ be two integral base-polyhedra,
and $p_{1}$ and $p_{2}$ be the associated (integer-valued) supermodular functions. 
For $i=1,2$, the set of integer points of $B_{i}$ is denoted as $\odotZ{B_{i}}$,
and the linear extension (Lov{\'a}sz extension) of $p_{i}$ as $\hat p_{i}$.
For each $s \in S$,
let $\varphi_{s}: \ZZ \to \ZZ \cup \{ +\infty \}$
be an integer-valued discrete convex function.
As before we denote the conjugate function of $\varphi_{s}$ 
by $\psi_{s}: \ZZ \to \ZZ \cup \{  +\infty \}$,
which is defined by (\ref{phiconjdef}).
Recall notation
\ $\dom \Phi =  \{ x  \in \ZZ\sp{S} :  x(s) \in \dom \varphi_{s} \mbox{ for each } s \in S \}$.

The following theorem gives
a duality formula for separable discrete convex functions on 
the intersection of two M-convex sets. 
We introduce notations for feasible vectors:
\begin{align} 
 \mathcal{P}_{0} &= \{  x \in \odotZ{B_{1}} \cap \odotZ{B_{2}} :
 x(s) \in \dom \varphi_{s} \mbox{ for each } s \in S \} 
\ = \odotZ{B_{1}} \cap \odotZ{B_{2}} \cap \dom \Phi ,
\label{primalfeasInter}
\\
 \mathcal{D}_{0} &= \{  (\pi_{1}, \pi_{2}) \in \ZZ\sp{S} \times \ZZ\sp{S}:
\pi_{i} \in \dom \hat p_{i} \ (i=1,2), \
\pi_{1}(s) + \pi_{2}(s) \in \dom \psi_{s}  \mbox{ for each } s \in S \} .
\label{dualfeasInter}
\end{align}

\begin{theorem} \label{THminmaxgensepInter}
Assume that $\mathcal{P}_{0} \not= \emptyset$ 
(primal feasibility) or 
$\mathcal{D}_{0} \not= \emptyset$ 
(dual feasibility) holds.
Then we have the min-max relation:%
\footnote{%%%%%%%%%%%%%%%%%%%%%%%
The unbounded case with both sides of (\ref{minmaxgensepInter})
being equal to $-\infty$ or $+\infty$ is also a possibility.
}%%%%%%%% footnote %%%%%%%%%%%%%%%%%%
\begin{align} 
 & \min \{  \sum_{s \in S} \varphi_{s} ( x(s) ) : 
   x \in \odotZ{B_{1}} \cap \odotZ{B_{2}} \} 
\notag \\ &
= \max\{ \hat p_{1}(\pi_{1}) + \hat p_{2}(\pi_{2}) 
 - \sum_{s  \in S}  \psi_{s} (\pi_{1}(s) + \pi_{2}(s) ) :  \pi_{1}, \pi_{2} \in \ZZ\sp{S} \} .
\label{minmaxgensepInter}
\end{align}
\end{theorem}
\begin{proof}
We give a proof based on an iterative application of the Fenchel duality
theorem (Theorem \ref{THmlfencdual}),
while the weak duality ($\min \geq \max$) is demonstrated in Remark \ref{RMinterWeakInter}.

We denote the indicator functions 
of $\odotZ{B}_{1}$ and $\odotZ{B}_{2}$ 
by $\delta_{1}$ and $\delta_{2}$,
respectively,
and use the notation
$\Phi(x) = \sum [\varphi_{s}(x(s)):  s\in S]$. 
In the Fenchel-type discrete duality
\begin{equation} \label{fencminmaxMinter}
\min\{ f(x) - h(x)  :  x \in \ZZ\sp{S}  \}
 = \max\{ h\sp{\circ}(\pi) - f\sp{\bullet}(\pi)  :  \pi \in \ZZ\sp{S} \}
\end{equation}
in  (\ref{fencminmaxM}),
we choose $f = \delta_{2} + \Phi$ and $h= -\delta_{1}$.
Since $f - h = \Phi + \delta_{1} + \delta_{2}$,
the left-hand side of (\ref{fencminmaxMinter})
coincides with the left-hand side of (\ref{minmaxgensepInter}).

The conjugate function $f\sp{\bullet}$ can be computed as follows.
For $\pi \in \ZZ\sp{S}$ we define 
$\varphi_{s}\sp{\pi}(k) =  \varphi_{s}(k) - \pi(s) k $
for $k \in \ZZ$ and $s \in S$.
Then the conjugate function $\psi_{s}\sp{\pi}$ of function $\varphi_{s}\sp{\pi}$
is given as
\begin{align*} 
 \psi_{s}\sp{\pi}(\ell)  
&= \max\{ k \ell  -  \varphi_{s}\sp{\pi}(k)  :  k \in \ZZ \}
\\ &
= \max\{ k ( \ell + \pi(s) )  -  \varphi_{s}(k) :  k \in \ZZ \}
\\ &
= \psi_{s}(\ell + \pi(s) )   
\qquad (\ell \in \ZZ).
\end{align*}
Using this expression  and 
the min-max formula (\ref{minmaxgensep-NI}) 
for $B_{2}$ and $\varphi_{s}\sp{\pi}$,
we obtain
\begin{align}
 f\sp{\bullet}(\pi) 
 &= \max\{  \langle \pi, x \rangle - \delta_{2}(x) - \sum_{s \in S} \varphi_{s}(x(s)) :  x \in \ZZ\sp{S} \}
\nonumber \\ 
 &= \max\{  \sum_{s \in S} \big[ \pi(s) x(s) -  \varphi_{s}(x(s)) \big] - \delta_{2}(x):  x \in \ZZ\sp{S} \}
\nonumber \\ 
 &= - \min\{  \sum_{s \in S} \varphi_{s}\sp{\pi}(x(s)) :  x \in \odotZ{B_{2}} \}
\nonumber \\ 
 &= - \max\{ \hat p_{2}(\pi') - \sum_{s \in S} \psi_{s}\sp{\pi}(\pi'(s)) :  \pi' \in \ZZ\sp{S} \} 
\nonumber \\ 
 &= - \max\{ \hat p_{2}(\pi') - \sum_{s \in S} \psi_{s}(\pi(s) + \pi'(s)) :  \pi' \in \ZZ\sp{S} \} 
\qquad ( \pi \in \ZZ\sp{S}) .
\label{fconjInter}
\end{align}
On the other hand,  the conjugate function $h\sp{\circ}$ of $h= -\delta_{1}$ is
equal to $\hat p_{1}$ by (\ref{deltaBconj}), i.e.,
\begin{equation}
 h\sp{\circ}(\pi) =\hat p_{1}(\pi)
\qquad ( \pi \in \ZZ\sp{S}).
\label{deltaBconj2}
\end{equation}
The substitution of (\ref{fconjInter}) and (\ref{deltaBconj2})
into $h\sp{\circ} - f\sp{\bullet}$
shows that 
the right-hand side of (\ref{fencminmaxMinter})
coincides with the right-hand side of (\ref{minmaxgensepInter}).
\end{proof}

\begin{remark} \rm  \label{RMinterWeakInter}
The weak duality ($\min \geq \max$) in \eqref{minmaxgensepInter}  is shown here.
Let $x \in \mathcal{P}_{0}$ and  
$(\pi_{1},\pi_{2}) \in \mathcal{D}_{0}$.
Then, using 
the Fenchel--Young inequality \eqref{youngineq}
for $(\varphi_{s}, \psi_{s})$ 
 as well as \eqref{plovextmin} for $p=p_{i}$ $(i=1,2)$, we obtain
\begin{align}
& \sum_{s \in S} \varphi_{s} ( x(s) ) 
 -
  \left( \hat p_{1}(\pi_{1}) + \hat p_{2}(\pi_{2}) 
 - \sum_{s  \in S}  \psi_{s} (\pi_{1}(s) + \pi_{2}(s) ) 
  \right)
\notag  \\ &  
  =  \sum_{s \in S}   \big[ \varphi_{s} ( x(s) ) +  \psi_{s}(\pi_{1}(s) + \pi_{2}(s) ) \big] 
  \ - \hat p_{1}(\pi_{1}) - \hat p_{2}(\pi_{2})
\notag  \\ &  
\geq
  \sum_{s \in S}    x(s) (\pi_{1}(s) + \pi_{2}(s) )  
  \ - \hat p_{1}(\pi_{1}) - \hat p_{2}(\pi_{2})
\label{weakd1Inter} 
\\ &  
=
  \sum_{s \in S}  \pi_{1}(s) x(s) +   \sum_{s \in S}  \pi_{2}(s) x(s) 
  \ - \hat p_{1}(\pi_{1}) - \hat p_{2}(\pi_{2})
\notag \\ &  
\geq
 \min \{ \pi_{1} z : z \in \odotZ{B_1} \}
 + \min \{ \pi_{2} z : z \in \odotZ{B_2} \}
  \ - \hat p_{1}(\pi_{1}) - \hat p_{2}(\pi_{2})
\ =  0,
\label{weakd2Inter}
\end{align}
showing the weak duality.
The optimality conditions can be obtained as the conditions
for the inequalities in 
\eqref{weakd1Inter}
and \eqref{weakd2Inter}
to be equalities, as stated in Proposition \ref{PRminmaxgensepInterOPT} below.
\finbox
\end{remark}

In the min-max formula \eqref{minmaxgensepInter} 
we  denote the set of the minimizers $x$  by $\mathcal{P}$ 
and the set of the maximizers $(\pi_{1},\pi_{2})$ by $\mathcal{D}$.
The following proposition follows from
the combination of 
Theorem \ref{THminmaxgensepInter} and Remark \ref{RMinterWeakInter}.
We remark that this proposition is a special case of Theorem~\ref{THfencoptprimaldual}.

\begin{proposition} \label{PRminmaxgensepInterOPT}
Assume that both $\mathcal{P}_{0}$ and $\mathcal{D}_{0}$ 
in \eqref{primalfeasInter}--\eqref{dualfeasInter} are nonempty.

%%Let $x \in \mathcal{P}_{0}$ and  $(\pi_{1},\pi_{2}) \in \mathcal{D}_{0}$.
%%Then $x \in \mathcal{P}$ and $(\pi_{1},\pi_{2}) \in \mathcal{D}$
%%if and only if the following three conditions are satisfied:

{\rm (1)}
Let $x \in \mathcal{P}_{0}$.
Then  $x \in \mathcal{P}$
if and only if 
there exists $(\pi_{1},\pi_{2}) \in \mathcal{D}_{0}$ such that
\begin{align}
 & 
\varphi_{s}(x(s)) - \varphi_{s}(x(s)-1) 
\leq \pi_{1}(s) + \pi_{2}(s) \leq 
\varphi_{s}(x(s)+1) - \varphi_{s}(x(s))
\qquad (s \in S),  
\label{pisubgradInter}
\\ 
 & \mbox{ 
  \rm $\pi_{1}(s) \geq \pi_{1}(t)$ 
  \quad  for every $(s,t)$ \ with \ $x + \chi_{s} - \chi_{t} \in \odotZ{B_{1}}$},
\label{pi1minzerInter}
\\
 & \mbox{ 
  \rm $\pi_{2}(s) \geq \pi_{2}(t)$ 
  \quad  for every $(s,t)$ \ with \ $x + \chi_{s} - \chi_{t} \in \odotZ{B_{2}}$}.
\label{pi2minzerInter}
\end{align}

{\rm (2)}
Let $(\pi_{1},\pi_{2}) \in \mathcal{D}_{0}$.
Then $(\pi_{1},\pi_{2}) \in \mathcal{D}$
if and only if 
there exists $x \in \mathcal{P}_{0}$ that satisfies
\eqref{pisubgradInter}, \eqref{pi1minzerInter}, and \eqref{pi2minzerInter}.

{\rm (3)}
For any $(\hat\pi_{1},\hat\pi_{2}) \in \mathcal{D}$, we have
\begin{align}  
\mathcal{P} &=  \{ x \in \mathcal{P}_{0}: 
\mbox{\rm \eqref{pisubgradInter}, \eqref{pi1minzerInter}, \eqref{pi2minzerInter}  
 hold with $(\pi_{1},\pi_{2}) = (\hat\pi_{1},\hat\pi_{2})$}  \} 
\label{primaloptsetInter1}
\\
& =  \{ x \in \dom \Phi:  
\mbox{\rm \eqref{pisubgradInter} holds with $(\pi_{1},\pi_{2}) = (\hat\pi_{1},\hat\pi_{2})$}  \} 
\notag \\  & \phantom{AA} 
 \cap 
\{ x \in \odotZ{B_{1}}:
 \mbox{\rm $x$ is a $\hat\pi_{1}$-minimizer in $\odotZ{B_{1}}$ } \} 
\notag \\  & \phantom{AA} 
 \cap 
\{ x \in \odotZ{B_{2}}:
 \mbox{\rm $x$ is a $\hat\pi_{2}$-minimizer in $\odotZ{B_{2}}$ } \} .
 \label{primaloptsetInter2}
\end{align}

{\rm (4)}
For any $\hat x \in \mathcal{P}$, we have
\begin{equation}  \label{dualoptsetInter}
\mathcal{D} =  \{ (\pi_{1},\pi_{2}) \in \mathcal{D}_{0} : 
\mbox{\rm \eqref{pisubgradInter}, \eqref{pi1minzerInter}, \eqref{pi2minzerInter} 
 hold with $x = \hat x$}  \} .
\end{equation}
\end{proposition}
\begin{proof}
The inequality \eqref{weakd1Inter}
turns into an equality 
if and only if,
for each $s \in S$, we have
$\varphi_{s} (k) + \psi_{s} ( \ell ) = k \ell$ 
for $k= x(s)$ and $\ell = \pi_{1}(s) + \pi_{2}(s)$.
The latter condition is equivalent to
\eqref{pisubgradInter} by \eqref{youngsubgrad}.
The other inequality \eqref{weakd2Inter}
turns into an equality 
if and only if
$x$ is a $\pi_{i}$-minimizer in $\odotZ{B_{i}}$
for $i=1,2$, 
that is, \eqref{pi1minzerInter} and \eqref{pi2minzerInter} hold.
Finally, we see from Theorem~\ref{THminmaxgensepInter} that
the two inequalities in \eqref{weakd1Inter} and  \eqref{weakd2Inter}
simultaneously turn into equality for some $x$ and $(\pi_{1},\pi_{2})$.
\end{proof}

\begin{proposition} \label{PRminmaxgensepoptsetInterLM}
In the min-max relation \eqref{minmaxgensepInter}
for a separable convex function on 
the intersection of two M-convex sets, the set 
\ $\mathcal{D}' := \{ (\pi_{1},-\pi_{2}) : (\pi_{1},\pi_{2}) \in \mathcal{D} \}$ \ 
corresponding to the maximizers is an L$\sp{\natural}$-convex set and 
the set $\mathcal{P}$ of the minimizers is an M$_{2}\sp{\natural}$-convex set.
\end{proposition}
\begin{proof}
The representation \eqref{dualoptsetInter}
shows that $\mathcal{D}$ is described by the inequalities in 
\eqref{pisubgradInter}, \eqref{pi1minzerInter}, and \eqref{pi2minzerInter}.
Hence $\mathcal{D}'$ is {\rm L}$\sp{\natural}$-convex.
In the representation \eqref{primaloptsetInter2} of $\mathcal{P}$, the first set
\ $\{ x \in \dom \Phi: \mbox{\rm \eqref{pisubgradInter} holds with 
$(\pi_{1},\pi_{2}) = (\hat\pi_{1},\hat\pi_{2})$}  \}$ \ 
is a 
box of integers (the set of integers in an integral box),
while for each $i=1,2$, the set of $\hat\pi_{i}$-minimizers
in $\odotZ{B_{i}}$  is an {\rm M}-convex set.
Therefore, $\mathcal{P}$ is an M$_{2}\sp{\natural}$-convex set.
\end{proof}

When specialized to a symmetric function, 
the min-max formula (\ref{minmaxgensepInter}) is simplified to 
\begin{align} 
 & \min \{  \sum_{s \in S} \varphi ( x(s) ) : 
   x \in \odotZ{B_{1}} \cap \odotZ{B_{2}} \} 
\notag \\ &
= \max\{ \hat p_{1}(\pi_{1}) + \hat p_{2}(\pi_{2}) 
 - \sum_{s  \in S}  \psi (\pi_{1}(s) + \pi_{2}(s) ) :  \pi_{1}, \pi_{2} \in \ZZ\sp{S} \} ,
\label{minmaxsymsepInter}
\end{align}
where $\varphi: \ZZ \to \ZZ \cup \{ +\infty \}$
is any integer-valued discrete convex function
and $\psi: \ZZ \to \ZZ \cup \{ +\infty \}$
is the conjugate of $\varphi$.
The identity (\ref{minmaxsymsepInter}) will play a key role
in the study of discrete decreasing minimization on the intersection of two M-convex sets,
just as  (\ref{minmaxsymsep}) did for an M-convex set.

As an example of (\ref{minmaxsymsepInter}) 
with explicit forms of $\varphi$ and $\psi$,
we mention a min-max formula
for the minimum square-sum of components on the intersection of two M-convex sets,
which is an extension of (\ref{minmaxSqSum-KM}) for an M-convex set.

\begin{theorem} \label{THintersquareminmax}
\begin{align} 
& \min \{ \sum_{s \in S} x(s)\sp{2} : x \in \odotZ{B_{1}} \cap \odotZ{B_{2}} \}
\notag \\ &
= \max\{ \hat p_{1}(\pi_{1}) + \hat p_{2}(\pi_{2}) -
 \sum_{s \in S}  \left\lfloor \frac{\pi_{1}(s) + \pi_{2}(s)}{2} \right\rfloor 
               \cdot \left\lceil \frac{\pi_{1}(s) + \pi_{2}(s)}{2} \right\rceil  
 : \pi_{1}, \pi_{2} \in \ZZ\sp{S} \} .
\label{minmaxintersquare}
\end{align}
\end{theorem} 
\begin{proof}
This is a special case of (\ref{minmaxsymsepInter}) with
$\varphi(k) = k \sp{2}$ and 
$\psi(\ell) = \left\lfloor {\ell}/{2} \right\rfloor 
               \cdot \left\lceil {\ell}/{2} \right\rceil$
(cf., (\ref{squareconj})).
\end{proof}

If $\odotZ{B_{1}} \cap \odotZ{B_{2}} \not= \emptyset$,
both sides of (\ref{minmaxintersquare}) are finite-valued, and the minimum and the maximum are attained.
If $\odotZ{B_{1}} \cap \odotZ{B_{2}} = \emptyset$,
the left-hand side of (\ref{minmaxintersquare}) is equal to $+\infty$ by convention
and the right-hand side is unbounded above (hence equal to $+\infty$). 
Note also that
$\odotZ{B_{1}} \cap \odotZ{B_{2}} \not= \emptyset$ if and only if $B_{1} \cap B_{2} \not= \emptyset$.

We can also formulate a min-max formula for a nonsymmetric quadratic function
$\sum_{s \in S} c(s) x(s)\sp{2}$,
where $c(s)$ is a positive integer for each $s \in S$.
On recalling the conjugate function in (\ref{squareconjWtd1}),
we obtain the following min-max formula.

\begin{theorem} \label{THminmaxSqSumWtdInter}
For an integer vector $c \in \ZZ\sp{S}$ with $c(s) \geq 1$ for all $s \in S$,
\begin{align} 
& \min \{ \sum_{s \in S} c(s) x(s)\sp{2} : x \in \odotZ{B_{1}} \cap \odotZ{B_{2}} \}
\notag \\ &
= \max\{ \hat p_{1}(\pi_{1}) + \hat p_{2}(\pi_{2})
 - \sum_{s \in S}
\left\lfloor \frac{1}{2} \left( \frac{\pi(s)}{c(s) } + 1 \right) \right\rfloor 
\left( 
 \pi(s)
 - c(s)  \left\lfloor \frac{1}{2} \left( \frac{\pi(s)}{c(s) } + 1 \right) \right\rfloor 
\right) :
\notag \\ & 
\phantom{ \max\{ \quad}
\pi = \pi_{1}+ \pi_{2}, \ \ 
\pi_{1}, \pi_{2} \in \ZZ\sp{S} \} .
\label{minmaxintersquareW}
\end{align}
\vspace{-1.5\baselineskip} \\
\finbox
\end{theorem}

We have obtained the min-max formulas \eqref{minmaxintersquare} and \eqref{minmaxintersquareW} 
as special cases of the Fenchel-type discrete duality in DCA.
Direct algorithmic proofs, not relying on the DCA machinery,
would be an interesting research topic.

%% FFFFFFFFFFFFFFFFFFFFFFFFFFFF

\subsection{Separable convex functions on network flows}
\label{SCmodflow}

Let $D=(V,A)$ be a digraph, and
suppose that we are given a finite integer-valued function $m$ on $V$ for which
$\widetilde m(V)=0$.  
A {\bf flow} means simply a function on $A$,
and we are interested in flow
$x$ that satisfies
\begin{equation} \label{modFflowdemand} 
  \varrho _x(v)-\delta _x(v)=m(v) 
\qquad \mbox{\rm for each node $v\in V$},
\end{equation}
where
\[
\varrho _x(v):= \sum [x(uv):uv\in A],
\qquad
\delta_x(v):= \sum [x(vu):  vu\in A].
\]
A flow $x$ satisfying \eqref{modFflowdemand} 
will be referred to as an {\bf $m$-flow}.

We consider a convex cost integer flow problem.
For each edge $e \in A$, 
an integer-valued (discrete) convex function
$\varphi_{e}: \ZZ \to \ZZ \cup \{ +\infty \}$
is given,
and we seek an integral flow $x$ that minimizes the sum of the edge costs
$\Phi(x) =  \sum_{e \in A} \varphi_{e} ( x(e) )$
subject to the constraint \eqref{modFflowdemand}. 
For the function value $\Phi(x)$ to be finite, we must have
\begin{equation} \label{modFdomphi1}
 x(e) \in \dom \varphi_{e} \quad \mbox{ for each edge $e \in A$},
\end{equation}
and therefore,
capacity constraints, if any, can be represented (implicitly)
in terms of the cost function $\varphi_{e}$.
A flow $x$ is called {\bf feasible}
if it satisfies the conditions \eqref{modFflowdemand} and \eqref{modFdomphi1}. 

\medskip

%%%%%%%%%%%%%%%%%%%%
\noindent
\qquad {\bf Convex cost flow problem (1)}:%
\begin{eqnarray}
\mbox{Minimize \ } & & 
 \Phi(x) =  \sum_{e \in A} \varphi_{e} ( x(e) )
\label{modFprimalobj1} \\
 \mbox{subject to \ } 
  &  &  \varrho _x(v)-\delta _x(v)=m(v)  \qquad (v \in V), 
 \label{modFflowdemand1} \\
  &  &  x(e) \in \ZZ
     \qquad (e \in A) .
 \label{modFflowint1} 
\end{eqnarray}
%%%%%%%%%%%%%%

\medskip

The dual problem, in its integer version, is as follows
(cf., e.g., \cite{AHO03}, \cite{Iri69book}, \cite{Mdcasiam}, \cite{Roc84}).
For each $e \in A$,  let
$\psi_{e}: \ZZ \to \ZZ \cup \{ +\infty \}$
denote the conjugate of $\varphi_{e}$,
that is,
\begin{equation} \label{modFphiconjdef}
 \psi_{e}(\ell)  = \max\{ k \ell  -  \varphi_{e}(k)  :  k \in \ZZ \}
\qquad
(\ell \in \ZZ),
\end{equation}
which is also an integer-valued (discrete) convex function.
This function $\psi_{e}$ represents the (dual) cost function associated with edge $e \in A$.
The decision variable in the dual problem is  an integer-valued potential $\pi: V \to \ZZ$
defined on the node-set $V$.
Recall the notation $\pi m = \sum_{v \in V}  \pi(v) m(v)$.

\medskip

%%%%%%%%%%%%%%%%%%%%
\noindent
\qquad {\bf Dual to the convex cost flow problem (1)}:%
\begin{eqnarray}
\mbox{Maximize \ } & & 
 \Psi(\pi) = \pi m
  -  \sum_{e =uv \in A}    \psi_{e} ( \pi(v) - \pi(u) ) 
\label{modFdualobj1} \\
 \mbox{subject to \ } 
  &  &  \pi(v) \in \ZZ
     \qquad (v \in V) .
\label{modFpotentialint1} 
\end{eqnarray}
%%%%%%%%%%%%%%

We introduce notations for feasible flows and potentials:
\begin{align} 
 \mathcal{P}_{0} &= \{  x \in \ZZ\sp{A} :
 \mbox{\rm $x$ satisfies \eqref{modFdomphi1} and  \eqref{modFflowdemand1}} \} ,
\label{primalfeasmodF1}
\\
 \mathcal{D}_{0} &= \{  \pi \in \ZZ\sp{V}:   
 \pi(v) - \pi(u) \in \dom \psi_{e} \ \mbox{\rm for each  } \  e=uv \in A
 \} .
\label{dualfeasmodF1}
\end{align}

\begin{theorem} \label{THminmaxmodF1}
Assume primal feasibility ($\mathcal{P}_{0} \not= \emptyset$) 
 or dual feasibility ($\mathcal{D}_{0} \not= \emptyset$).
Then we have the min-max relation:
\begin{equation} \label{modFminmax1}
 \min \{  \Phi(x) : 
 x \in \ZZ\sp{A} \ \mbox{\rm \  satisfies \eqref{modFflowdemand1}}   \}  
= \max\{ \Psi(\pi) :  \pi \in \ZZ\sp{V} \} .
\end{equation}
The unbounded case with both sides 
being equal to $-\infty$ or $+\infty$ is also a possibility.
\finbox
\end{theorem}

As an example of \eqref{modFminmax1} 
with explicit forms of $\Phi$ and $\Psi$,
we mention a min-max formula
for the minimum square-sum of components of an integral $m$-flow,
where no capacity constrains are imposed.

%% 2019-10-04 %% for vers. 4
\begin{proposition} \label{PRflowsquareminmax}
For a digraph $D=(V,A)$ and 
 an integer vector $m$ on $V$ with $\widetilde m(V)=0$, we have
\begin{align} 
& \min \{ \sum_{e \in A} x(e)\sp{2} : 
  \mbox{\rm $x$ is an integral $m$-flow}  \}
\notag \\ &
= \max\{ \sum_{v \in V} \pi(v) m(v) -
 \sum_{uv \in A}  \left\lfloor \frac{ \pi(v) - \pi(u) }{2} \right\rfloor 
               \cdot \left\lceil \frac{ \pi(v) - \pi(u) }{2} \right\rceil  
 : \pi \in \ZZ\sp{V} \} .
\label{minmaxflowsquare}
\end{align}
\end{proposition} 

\begin{proof}
This is a special case of (\ref{modFminmax1}) with
$\varphi_{e}(k) = k \sp{2}$ and 
$\psi_{e}(\ell) = \left\lfloor {\ell}/{2} \right\rfloor 
               \cdot \left\lceil {\ell}/{2} \right\rceil$
(cf., (\ref{squareconj})).
\end{proof}

In using this min-max relation in Part~III \cite{FM18part3}
it is convenient to introduce capacity constraints explicitly.
We denote 
the integer-valued lower and upper bound functions 
on $A$ by $f$ and $g$, for which $f\leq g$ is assumed,
and impose the capacity constraint
\begin{equation} \label{modFcap1} 
  f(e) \leq x(e) \leq g(e) 
\qquad \mbox{\rm for each edge $e \in A$} .
\end{equation}
With this explicit form of capacity constraints, 
a flow $x$ is called {\bf feasible}
if it satisfies the conditions 
\eqref{modFflowdemand},
\eqref{modFdomphi1} and \eqref{modFcap1}. 
The primal problem reads as follows.

\medskip

%%%%%%%%%%%%%%%%%%%%
\noindent
\qquad {\bf Convex cost flow problem (2)}:%
\begin{eqnarray}
\mbox{Minimize \ } & & 
 \Phi(x) =  \sum_{e \in A} \varphi_{e} ( x(e) )
\label{modFprimalobj2} \\
 \mbox{subject to \ } 
  &  &  \varrho _x(v)-\delta _x(v)=m(v)  \qquad (v \in V), 
 \label{modFflowdemand2} \\
  & &  f(e) \leq x(e) \leq g(e)  \qquad (e \in A),
 \label{modFflowcapconst2} \\
  &  &  x(e) \in \ZZ
     \qquad (e \in A) .
 \label{modFflowint2} 
\end{eqnarray}
%%%%%%%%%%%%%%

\medskip

The corresponding dual problem can be given as follows
(cf., Remark \ref{RMmodFdualderivation}),
where the decision variables consist of
an integer-valued potential $\pi: V \to \ZZ$ on $V$
and integer-valued functions $\tau_{1}, \tau_{2}: A \to \ZZ$ on $A$.
The constraint \eqref{modFpotdifftension2} below says that 
the tension (potential difference) is split into two parts $\tau_{1}$ and $\tau_{2}$.

\medskip

%%%%%%%%%%%%%%%%%%%%
\noindent
\qquad {\bf Dual to the convex cost flow problem (2)}:%
\begin{eqnarray}
\mbox{Maximize \ } & & 
 \Psi(\pi,\tau_{1}, \tau_{2}) = \pi m
  -  \sum_{e \in A}  \bigg( \  
      \psi_{e} ( \tau_{1}(e) ) + \max \{ f(e) \tau_{2}(e) , \ g(e) \tau_{2}(e) \}
           \  \bigg)
 \label{modFdualobj2} \\
 \mbox{subject to \ } 
  &  &  \pi(v) - \pi(u) = \tau_{1}(e) + \tau_{2}(e) 
  \qquad (e  = uv \in A),
 \label{modFpotdifftension2} \\
  &  &  \pi(v) \in \ZZ
     \qquad (v \in V) ,
 \label{modFpotentialint2}  \\
  &  &  \tau_{1}(e), \tau_{2}(e) \in \ZZ
     \qquad (e \in A) .
 \label{modFtensionint2} 
\end{eqnarray}
%%%%%%%%%%%%%%

We introduce notations for feasible flows and potentials/tensions:
\begin{align} 
 \mathcal{P}_{0} &= \{  x \in \ZZ\sp{A} :
 \mbox{\rm $x$ satisfies \eqref{modFdomphi1}, \eqref{modFflowdemand2}, \eqref{modFflowcapconst2} }
 \} ,
\label{primalfeasmodF2}
\\
 \mathcal{D}_{0} &= \{  (\pi, \tau_{1}, \tau_{2}) \in \ZZ\sp{V} \times \ZZ\sp{A}  \times \ZZ\sp{A}:  
 \eqref{modFpotdifftension2},   \  
 \tau_{1}(e) \in \dom \psi_{e} \ \mbox{\rm for each  } \  e \in A
 \} .
\label{dualfeasmodF2}
\end{align}

\begin{theorem} \label{THminmaxmodF2}
Assume primal feasibility ($\mathcal{P}_{0} \not= \emptyset$) 
 or dual feasibility ($\mathcal{D}_{0} \not= \emptyset$).
Then we have the min-max relation:
\begin{align}
& \min \{  \Phi(x) : 
 x \in \ZZ\sp{A} \ \mbox{\rm \  satisfies 
\eqref{modFflowdemand2} and \eqref{modFflowcapconst2}} \}  
\notag \\ 
 &= \max\{ \Psi(\pi, \tau_{1}, \tau_{2}) : 
\pi \in \ZZ\sp{V} \ \mbox{\rm and }  \tau_{1},  \tau_{2}  \in \ZZ\sp{A} \ 
\mbox{\rm \  satisfy \eqref{modFpotdifftension2}}  \} .
\label{modFminmax2}
\end{align}
The unbounded case with both sides 
being equal to $-\infty$ or $+\infty$ is also a possibility.
\finbox
\end{theorem}

As an example of \eqref{modFminmax2} 
with explicit forms of $\Phi$ and $\Psi$,
we can obtain the capacitated version of 
the min-max formula \eqref{minmaxflowsquare}
in Proposition \ref{PRflowsquareminmax}.

\begin{proposition} \label{PRflowsquareminmax3}
For a digraph $D=(V,A)$, 
an integer vector $m$ on $V$ with $\widetilde m(V)=0$, 
and integer-valued functions $f$ and $g$ on $A$ with $f\leq g$,
we have
\begin{align} 
& \min \{ \sum_{e \in A} x(e)\sp{2} : 
  \mbox{\rm $x$ is an integral $m$-flow satisfying $f \leq x \leq g$}  \}
\notag \\ &
= \max\{ \sum_{v \in V} \pi(v) m(v) -
 \sum_{e \in A} \left(  \left\lfloor \frac{ \tau_{1}(e) }{2} \right\rfloor 
               \cdot \left\lceil \frac{ \tau_{1}(e) }{2} \right\rceil  
 +  \max \{ f(e) \tau_{2}(e) , \ g(e) \tau_{2}(e) \} \right)
\notag \\ & \phantom{=\max\ }
 :  \pi(v) - \pi(u) = \tau_{1}(e) + \tau_{2}(e) \  (e=uv), \ 
   \pi \in \ZZ\sp{V}, \     \tau_{1}, \tau_{2} \in \ZZ\sp{A}
 \} .
\label{minmaxflowsquare3}
\end{align}
\end{proposition} 

\begin{proof}
For each $e \in A$, let $\varphi_{e}(k) = k \sp{2}$, whose conjugate function
is 
$\psi_{e}(\ell) = \left\lfloor {\ell}/{2} \right\rfloor 
               \cdot \left\lceil {\ell}/{2} \right\rceil$
by (\ref{squareconj}).
Then 
\eqref{minmaxflowsquare3} follows from \eqref{modFminmax2}. 
\end{proof}

As a special case of \eqref{minmaxflowsquare3} 
we can obtain a min-max formula for non-negative flows.
%% version of the min-max formula \eqref{minmaxflowsquare}
%%in Proposition \ref{PRflowsquareminmax}.

\begin{proposition} \label{PRflowsquareminmax2}
For a digraph $D=(V,A)$ and 
 an integer vector $m$ on $V$ with $\widetilde m(V)=0$, we have
\begin{align} 
& \min \{ \sum_{e \in A} x(e)\sp{2} : 
  \mbox{\rm $x$ is a non-negative integral $m$-flow}  \}
\notag \\ &
= \max\{ \sum_{v \in V} \pi(v) m(v) -
 \sum_{uv \in A}  \left\lfloor \frac{ (\pi(v) - \pi(u))\sp{+} }{2} \right\rfloor 
               \cdot \left\lceil \frac{ (\pi(v) - \pi(u))\sp{+} }{2} \right\rceil  
 : \pi \in \ZZ\sp{V} \} ,
\label{minmaxflowsquare2}
\end{align}
where $(\pi(v) - \pi(u))\sp{+} = \max \{ 0,  \pi(v) - \pi(u) \}$.
\end{proposition} 

\begin{proof}
Let $f(e)=0$ and $g(e)=+\infty$ for all $e \in A$.
%%Then the left-hand side of (\ref{modFminmax2}) 
%%coincides with that of \eqref{minmaxflowsquare2}.
In the maximization on the right-hand side of (\ref{minmaxflowsquare3}), 
we may assume $\tau_{2}(e) \leq 0$ for all $e \in A$,
since otherwise  
$\max \{ f(e) \tau_{2}(e) , \ g(e) \tau_{2}(e) \} = +\infty$.
Then $\tau_{1}(e) \geq  \pi(v) - \pi(u)$.
Since
\[
\min\{ 
 \left\lfloor \frac{ \tau_{1}(e) }{2} \right\rfloor 
               \cdot \left\lceil \frac{ \tau_{1}(e) }{2} \right\rceil  
: \tau_{1}(e) \geq  \pi(v) - \pi(u) \}
= \left\lfloor \frac{ (\pi(v) - \pi(u))\sp{+} }{2} \right\rfloor 
               \cdot \left\lceil \frac{ (\pi(v) - \pi(u))\sp{+} }{2} \right\rceil  ,
\]
the right-hand side of 
of (\ref{minmaxflowsquare3}) 
reduces to that of \eqref{minmaxflowsquare2}.
\end{proof}

For the min-max formula \eqref{modFminmax2}
we can obtain the following optimality criterion,
where we denote the set of the minimizers $x$  by $\mathcal{P}$ 
and the set of the maximizers $(\pi, \tau_{1},\tau_{2})$ by $\mathcal{D}$.

\begin{proposition} \label{PRminmaxmodF2OPT}
Assume that both $\mathcal{P}_{0}$ and $\mathcal{D}_{0}$ 
in \eqref{primalfeasmodF2}--\eqref{dualfeasmodF2} are nonempty.

{\rm (1)}
Let $x \in \mathcal{P}_{0}$.
Then 
$x \in \mathcal{P}$ 
if and only if
there exists
$(\pi, \tau_{1}, \tau_{2}) \in \mathcal{D}_{0}$
such that
\begin{align}
 & 
\varphi_{e}(x(e)) - \varphi_{e}(x(e)-1) 
\leq \tau_{1}(e) \leq 
\varphi_{e}(x(e)+1) - \varphi_{e}(x(e))
\qquad (e  \in A),  
\label{pisubgradmodF2}
\\
 & 
 \tau_{2}(e) 
   \left\{  \begin{array}{ll}
    =0   & \mbox{\rm if \ $f(e) +1 \leq x(e) \leq g(e) - 1$},  \\
    \leq 0   & \mbox{\rm if \ $x(e) =  f(e)$},  \\
    \geq 0   & \mbox{\rm if \  $x(e) =  g(e)$}  \\
             \end{array}  \right.
\qquad (e  \in A).
\label{capaslackmodF2}
\end{align}

{\rm (2)}
Let 
$(\pi, \tau_{1}, \tau_{2}) \in \mathcal{D}_{0}$.
Then 
$(\pi,\tau_{1},\tau_{2}) \in \mathcal{D}$
if and only if
there exists
$x \in \mathcal{P}_{0}$
that satisfies
\eqref{pisubgradmodF2} and \eqref{capaslackmodF2}.

{\rm (3)}
For any $(\hat\pi,\hat\tau_{1},\hat\tau_{2}) \in \mathcal{D}$, we have
\begin{equation}  \label{primaloptsetmodF2}
\mathcal{P} =  \{ x \in \mathcal{P}_{0}: 
 \mbox{\rm \eqref{pisubgradmodF2} and \eqref{capaslackmodF2} 
 hold with $(\pi,\tau_{1},\tau_{2}) = (\hat\pi,\hat\tau_{1},\hat\tau_{2})$}  \} ,
\end{equation}
where the conditions 
in \eqref{pisubgradmodF2} and \eqref{capaslackmodF2} 
can be rewritten as
\begin{align}
 & 
 x(e) \in \argmin_{k} \{ \varphi_{e}(k) - \tau_{1}(e) k \}  
\qquad (e  \in A),  
\label{pisubgradmodF2B}
\\
 & 
   \left\{ \begin{array}{ll}
    x(e) = f(e)   & \mbox{\rm if} \ \   \tau_{2}(e) < 0,  \\
    f(e) \leq x(e) \leq g(e)  & \mbox{\rm if} \ \   \tau_{2}(e) =0,  \\
    x(e) = g(e)   & \mbox{\rm if} \ \   \tau_{2}(e) > 0  \\
             \end{array}  \right.
\qquad (e  \in A).
\label{capaslackmodF2B}
\end{align}

{\rm (4)}
For any $\hat x \in \mathcal{P}$, we have
\begin{equation}  \label{dualoptsetmodF2}
\mathcal{D} =  \{ (\pi,\tau_{1},\tau_{2}) \in \mathcal{D}_{0} : 
 \mbox{\rm \eqref{pisubgradmodF2} and \eqref{capaslackmodF2} 
 hold with $x = \hat x$}  \} .
\end{equation}
\end{proposition}
\begin{proof}
This is a special case of Proposition \ref{PRminmaxsbmF2OPT} in Section \ref{SCsbmflow}.
\end{proof}

The condition
\eqref{capaslackmodF2}, or equivalently \eqref{capaslackmodF2B},
expresses the so-called kilter condition
for flow $x(e)$ and tension $\tau_{2}(e)$,
whereas the condition 
\eqref{pisubgradmodF2}, or equivalently \eqref{pisubgradmodF2B}, 
can be regarded as a nonlinear version thereof
for flow $x(e)$ and tension $\tau_{1}(e)$.

\begin{remark} \rm  \label{RMmodFdualderivation}
Here we derive the dual problem 
\eqref{modFdualobj2}--\eqref{modFtensionint2} 
from the basic case in 
\eqref{modFdualobj1}--\eqref{modFpotentialint1}. 
For each $e \in A$, 
let $\delta_{e}$ denote the indicator function of the integer interval 
$[f(e),g(e)]_{\ZZ}$,
define $\tilde \varphi_{e} :=  \varphi_{e} + \delta_{e}$,
and let $\tilde \psi_{e}$ be the conjugate function of $\tilde \varphi_{e}$.
By the claim below we obtain the following expression
\[
\tilde \psi_{e}( \pi(v) - \pi(u) )  =
\min \bigg\{
  \psi_{e} ( \ell_{1} ) + \max \{ f(e) \ell_2 , g(e) \ell_2 \}:
   \ell_{1}, \ell_{2} \in \ZZ, \    \ell_{1} + \ell_{2}  =  \pi(v) - \pi(u)  
  \bigg\}.
\]
The substitution of this expression into \eqref{modFdualobj1}
results in \eqref{modFdualobj2}--\eqref{modFtensionint2}.

{\bf Claim}: 
Let 
$\varphi: \ZZ \to \ZZ \cup \{ +\infty \}$
be a (discrete) convex function,
$\delta: \ZZ \to \ZZ \cup \{ +\infty \}$  
the indicator function of an integer interval $[a,b]_{\ZZ}$ with $a \leq b$.
Then the conjugate function 
$( \varphi + \delta )\sp{\bullet}$
of 
$\varphi + \delta $ is given by
\begin{equation} \label{modFconjsum}
 ( \varphi + \delta )\sp{\bullet}(\ell) = 
\min \bigg\{
  \varphi\sp{\bullet} ( \ell_1 ) + \max \{ a \ell_2 , b \ell_2 \}:
   \ell_{1}, \ell_{2} \in \ZZ, \    \ell_{1} + \ell_{2} = \ell  
  \bigg\}.
\end{equation}

\noindent
(Proof) 
By Theorem 8.36 of \cite{Mdcasiam},
$( \varphi + \delta )\sp{\bullet}$ is equal to
the infimum convolution of 
$ \varphi\sp{\bullet}$ and $\delta\sp{\bullet}$,
that is,
\begin{equation*} %%\label{modFconjsumprf1}
 ( \varphi + \delta )\sp{\bullet}(\ell) = 
\min \bigg\{
  \varphi\sp{\bullet} ( \ell_1 ) + \delta\sp{\bullet} (\ell_{2}): 
   \ell_{1}, \ell_{2} \in \ZZ, \ 
   \ell_{1} + \ell_{2} = \ell  
  \bigg\}.
\end{equation*}
Here  we have
\[
 \delta\sp{\bullet}(\ell) 
 = \max \{ k\ell - \delta(k) \} 
 = \max \{ k\ell : a \leq k \leq b \} 
 = \max \{ a \ell, b \ell \} .
\]
Hence follows \eqref{modFconjsum}.
\finbox
\end{remark}

\begin{remark} \rm  \label{RMmodFfeasible}
The feasibility of the primal problems
can be expressed by a variant of the Hoffman-condition. 
Denote the integer interval of $\dom \varphi_{e}$ by $[ f'(e), g'(e) ]_{\ZZ}$
with $f'(e) \in \ZZ \cup \{ -\infty \}$ and $g'(e) \in \ZZ \cup \{ +\infty \}$.
Then, by Hoffman's theorem,
there exists a feasible flow for the basic problem 
\eqref{modFprimalobj1}--\eqref{modFflowint1}
if and only if 
\begin{equation}
\varrho_{g'}(Z)-\delta_{f'}(Z) \geq \widetilde m(Z) 
\qquad 
\hbox{for all}
\quad  Z\subseteq V
\label{modFhoffman} 
\end{equation}
is satisfied.
For the problem 
\eqref{modFprimalobj2}--\eqref{modFflowint2}
with explicit capacity constraints,
we replace 
$f'(e)$ and $g'(e)$ by 
$\max \{ f(e), f'(e) \}$ and $\min \{ g(e), g'(e) \}$, respectively.
\finbox
\end{remark}

%%SSSSSSSSSSSSSSSSSSSSSSSSSSSSSSSSSSSSSSS

\subsection{Separable convex functions on submodular flows}
\label{SCsbmflow}

Let $D=(V,A)$ be a digraph, and
suppose that we are given an integral base-polyhedron $B$
with ground-set $V$.
We assume that $B$ is  described as $B = B'(p)$ in \eqref{basepolysupermod} 
by an integer-valued (fully) supermodular function 
$p: 2\sp{V} \to \ZZ \cup \{ -\infty \}$ with $p(V)=0$,
which is equivalent to saying that 
$B$ is  described as $B = B(b)$ in \eqref{basepolysubmod} 
by an integer-valued (fully) submodular function 
$b: 2\sp{V} \to \ZZ \cup \{ +\infty \}$ with $b(V)=0$,
where $b$ is the complementary function of $p$.

Here we are interested in an integral flow $x: A \to \ZZ$ 
such that the net-in-flow vector
$(\varrho _x(v)-\delta _x(v) : v \in V )$ belongs to $B$,
which we express  as 
\begin{equation} \label{sbmFflowdemand} 
  (\varrho _x(v)-\delta _x(v) : v \in V ) \in \odotZ{B} . 
\end{equation}
Such a flow $x$ is called a {\bf submodular flow}.
The constraint \eqref{modFflowdemand}
for the ordinary flow  problem in Section \ref{SCmodflow}
is a (very) special case of \eqref{sbmFflowdemand}
where the bounding submodular function $b$ 
(or the supermodular function $p$) is a  modular function $\widetilde m$
defined by the vector $m$.

We consider a convex cost integer submodular flow problem.
For each edge $e \in A$, 
an integer-valued (discrete) convex function
$\varphi_{e}: \ZZ \to \ZZ \cup \{ +\infty \}$
is given,
and we seek an integral flow $x$ that minimizes the sum of the edge costs
$ \Phi(x) =  \sum_{e \in A} \varphi_{e} ( x(e) )$
subject to the submodular constraint \eqref{sbmFflowdemand}. 
For the function value $\Phi(x)$ to be finite, we must have
\begin{equation} \label{sbmFdomphi1}
 x(e) \in \dom \varphi_{e} \quad \mbox{ for each edge $e \in A$},
\end{equation}
and therefore,
capacity constraints, if any, can be represented (implicitly)
in terms of the cost function $\varphi_{e}$.
A {\bf feasible submodular flow} means
a flow $x$ that satisfies the conditions \eqref{sbmFflowdemand} and \eqref{sbmFdomphi1}.

\medskip

%%%%%%%%%%%%%%%%%%%%
\noindent
\qquad {\bf Convex cost submodular flow problem (1)}:%
\begin{eqnarray}
\mbox{Minimize \ } & & 
 \Phi(x) =  \sum_{e \in A} \varphi_{e} ( x(e) )
\label{sbmFprimalobj1} \\
 \mbox{subject to \ } 
  &  &  (\varrho _x(v)-\delta _x(v) : v \in V ) \in \odotZ{B}, 
 \label{sbmFflowdemand1} \\
  &  &  x(e) \in \ZZ
     \qquad (e \in A) .
 \label{sbmFflowint1} 
\end{eqnarray}
%%%%%%%%%%%%%%

\medskip

In discrete convex analysis,
a systematic study of convex-cost submodular flows
has been conducted 
in a more general framework called the {\bf M-convex submodular flow problem},
where particular emphasis is laid on duality theorems (Murota \cite{Msbmfl99,Mdcasiam}).

The decision variable in the dual problem is  an integer-valued potential $\pi: V \to \ZZ$.
The objective function 
$\Psi(\pi)$ involves 
the linear extension (Lov{\'a}sz extension) 
 $\hat{p}(\pi)$ of the supermodular function $p$
defining $B$ as well as
the conjugate function $\psi_{e}$ of $\varphi_{e}$ 
for all $e \in A$.
It is worth noting that $\pi m$ in \eqref{modFdualobj1} 
is replaced by $\hat{p}(\pi)$ in \eqref{sbmFdualobj1}.

\medskip

%%%%%%%%%%%%%%%%%%%%
\noindent
\qquad {\bf Dual to the convex cost submodular flow problem (1)}:%
\begin{eqnarray}
\mbox{Maximize \ } & & 
 \Psi(\pi) = \hat{p}(\pi) 
  -  \sum_{e = uv \in A}    \psi_{e} ( \pi(v) - \pi(u) ) 
 \label{sbmFdualobj1} \\
 \mbox{subject to \ } 
  &  &  \pi(v) \in \ZZ
     \qquad (v \in V) .
 \label{sbmFpotentialint1}  
\end{eqnarray}
%%%%%%%%%%%%%%

The following min-max formula can be derived as a special case of
a min-max formula \cite[(9.83), page 270]{Mdcasiam} for M-convex submodular flows,
while the weak duality ($\min \geq \max$)  is demonstrated in Remark \ref{RMsmbFminmaxderivation} below.
We also mention that the min-max formula \eqref{sbmFminmax1} below can be regarded as 
being equivalent to the Fenchel-type discrete duality theorem (Theorem \ref{THmlfencdual});
see \cite[Section 9.1.4]{Mdcasiam} for the detail of this equivalence.
We introduce notations for feasible flows and potentials:
\begin{align} 
 \mathcal{P}_{0} &= \{  x \in \ZZ\sp{A} :
 \mbox{\rm $x$ satisfies \eqref{sbmFdomphi1} and  \eqref{sbmFflowdemand1}} \} ,
\label{primalfeassbmF1}
\\
 \mathcal{D}_{0} &= \{  \pi \in \ZZ\sp{V}:  \pi \in \dom \hat{p}, \ \  
 \pi(v) - \pi(u) \in \dom \psi_{e} \ \mbox{\rm for each  } \  e=uv \in A
 \} .
\label{dualfeassbmF1}
\end{align}

\begin{theorem} \label{THminmaxsbmF1}
Assume primal feasibility ($\mathcal{P}_{0} \not= \emptyset$) 
 or dual feasibility ($\mathcal{D}_{0} \not= \emptyset$).
Then we have the min-max relation:
\begin{equation} \label{sbmFminmax1}
 \min \{  \Phi(x) : 
 x \in \ZZ\sp{A}  \ \mbox{\rm \  satisfies \eqref{sbmFflowdemand1}} \}  
= \max\{ \Psi(\pi) :  \pi \in \ZZ\sp{V} 
 \} .
\end{equation}
The unbounded case with both sides 
being equal to $-\infty$ or $+\infty$ is also a possibility.
\finbox
\end{theorem}

\begin{remark} \rm  \label{RMsmbFminmaxderivation}
The weak duality 
$\Phi(x) \geq  \Psi(\pi)$
is shown here.
Let $x$ and $\pi$ be primal and dual feasible solutions. Then, using 
the Fenchel--Young inequality \eqref{youngineq}
for $(\varphi_{e}, \psi_{e})$ 
and the feasibility condition \eqref{sbmFflowdemand1}
as well as
the expression \eqref{plovextmin} for $\hat{p}(\pi)$, we obtain
\begin{align}
  \Phi(x) -  \Psi(\pi)
&  =  \sum_{e = uv \in A}   [ \varphi_{e} ( x(e) ) +  \psi_{e} ( \pi(v) - \pi(u) ) ] 
  \ -\hat{p}(\pi) 
\notag  \\ &  
\geq
  \sum_{e = uv \in A}    x(e) ( \pi(v) - \pi(u) )  \   -\hat{p}(\pi) 
\label{sbmF1weakd1} 
\\ &  
=
  \sum_{v \in V}    \pi(v) (\varrho _x(v)-\delta _x(v))
  \   -\hat{p}(\pi) 
\notag \\ &  
\geq
 \min \{ \pi z : z \in \odotZ{B} \}
  \   -\hat{p}(\pi) 
\ =  0.
\label{sbmF1weakd2}
\end{align}
This shows the weak duality.
The optimality conditions can be obtained as the conditions
for the inequalities in \eqref{sbmF1weakd1} and \eqref{sbmF1weakd2}
to be equalities.
See Proposition \ref{PRminmaxsbmF1OPT} below.
\finbox
\end{remark} 

In the min-max formula \eqref{sbmFminmax1}
we denote the set of the minimizers $x$ by $\mathcal{P}$ 
and the set of the maximizers $\pi$ by $\mathcal{D}$.

\begin{proposition} \label{PRminmaxsbmF1OPT}
Assume that both $\mathcal{P}_{0}$
and $\mathcal{D}_{0}$ 
in \eqref{primalfeassbmF1}--\eqref{dualfeassbmF1} are nonempty.

{\rm (1)}
Let $x \in \mathcal{P}_{0}$.
Then  $x \in \mathcal{P}$
if and only if 
there exists $\pi \in \mathcal{D}_{0}$ such that
\begin{align}
 &
\varphi_{e}(x(e)) - \varphi_{e}(x(e)-1) 
\leq \pi(v) - \pi(u) \leq 
\varphi_{e}(x(e)+1) - \varphi_{e}(x(e))
\qquad (e = uv \in A),  
\label{pisubgradsbmF1}
\\
 & 
\mbox{\rm Net-in-flow vector $(\varrho _x(v)-\delta _x(v) : v \in V )$ 
 is a $\pi$-minimizer in $\odotZ{B}$}.
\label{piminzersbmF1}
\end{align}

{\rm (2)}
Let 
$\pi \in \mathcal{D}_{0}$. 
Then $\pi \in \mathcal{D}$ 
if and only if 
there exists $x \in \mathcal{P}_{0}$
that satisfies
\eqref{pisubgradsbmF1} and \eqref{piminzersbmF1}.

{\rm (3)}
For any $\hat\pi \in \mathcal{D}$, we have
\begin{equation}  \label{primaloptsetsbmF1}
\mathcal{P} =  \{ x \in \mathcal{P}_{0}: 
\mbox{\rm  \eqref{pisubgradsbmF1} and \eqref{piminzersbmF1}   
 hold with $\pi = \hat\pi$}  \} ,
\end{equation}
\end{proposition}
where the condition in \eqref{pisubgradsbmF1} can be rewritten as
\begin{equation}
 x(e) \in \argmin_{k} \{ \varphi_{e}(k) - ( \pi(v) - \pi(u) ) k \}  
\qquad (e =(u,v)  \in A).  
\label{pisubgradsbmF1B}
\end{equation}

{\rm (4)}
For any $\hat x \in \mathcal{P}$, we have
\begin{equation}  \label{dualoptsetsbmF1}
\mathcal{D} =  \{ \pi \in \mathcal{D}_{0} : 
\mbox{\rm  \eqref{pisubgradsbmF1} and \eqref{piminzersbmF1}   
 hold with $x = \hat x$}  \} .
\end{equation}

\begin{proof}
The inequality \eqref{sbmF1weakd1}
turns into an equality 
if and only if,
for each $e = uv \in A$, we have
$\varphi_{e} (k) + \psi_{e} ( \ell ) = k \ell$ 
for $k= x(e)$ and $\ell = \pi(v) - \pi(u)$.
The latter condition is equivalent to
\eqref{pisubgradsbmF1} by \eqref{youngsubgrad}.
The other inequality \eqref{sbmF1weakd2} is an equality 
if and only if
\eqref{piminzersbmF1} holds.
\end{proof}

%% 2020-06-30 for Ver.4
%%In using the min-max relation in Part~IV \cite{FM18part4}
In applications it is often convenient to introduce capacity constraints explicitly as 
\begin{equation} \label{sbmFcap1} 
  f(e) \leq x(e) \leq g(e) 
\qquad \mbox{\rm for each edge $e \in A$}.
\end{equation}
With this explicit form of capacity constraints, 
a flow $x$ is called a {\bf feasible submodular flow}
if it satisfies the conditions
\eqref{sbmFflowdemand} and \eqref{sbmFcap1}
as well as \eqref{sbmFdomphi1}. 
The primal problem reads as follows.

\medskip

%%%%%%%%%%%%%%%%%%%%
\noindent
\qquad {\bf Convex cost submodular flow problem (2)}:%
\begin{eqnarray}
\mbox{Minimize \ } & & 
 \Phi(x) =  \sum_{e \in A} \varphi_{e} ( x(e) )
\label{sbmFprimalobj2} \\
 \mbox{subject to \ } 
  &  &  (\varrho _x(v)-\delta _x(v) : v \in V ) \in \odotZ{B}, 
 \label{sbmFflowdemand2} \\
  & &  f(e) \leq x(e) \leq g(e)  \qquad (e \in A),
 \label{sbmFflowcapconst2} \\
  &  &  x(e) \in \ZZ
     \qquad (e \in A) .
 \label{sbmFflowint2} 
\end{eqnarray}
%%%%%%%%%%%%%%

\medskip

The corresponding dual problem can be derived
from \eqref{sbmFdualobj1}--\eqref{sbmFpotentialint1}
 by the technique described in Remark \ref{RMmodFdualderivation}.
The decision variables of the resulting dual problem consist of
an integer-valued potential $\pi: V \to \ZZ$ on $V$
and integer-valued functions $\tau_{1}, \tau_{2}: A \to \ZZ$ on $A$.
The constraint \eqref{sbmFpotdifftension2} below says that 
the tension (potential difference) is split into two parts $\tau_{1}$ and $\tau_{2}$.

\medskip

%%%%%%%%%%%%%%%%%%%%
\noindent
\qquad {\bf Dual to the convex cost submodular flow problem (2)}:%
\begin{eqnarray}
\mbox{Maximize \ } & & 
 \Psi(\pi,\tau_{1}, \tau_{2}) = \hat{p}(\pi) 
  -  \sum_{e \in A}  \bigg( \  
      \psi_{e} ( \tau_{1}(e) ) + \max \{ f(e) \tau_{2}(e) , \ g(e) \tau_{2}(e) \}
           \  \bigg)
 \label{sbmFdualobj2} \\
 \mbox{subject to \ } 
  &  &  \pi(v) - \pi(u) = \tau_{1}(e) + \tau_{2}(e) 
  \qquad (e  = uv \in A),
 \label{sbmFpotdifftension2} \\
  &  &  \pi(v) \in \ZZ
     \qquad (v \in V) ,
 \label{sbmFpotentialint2}  \\
  &  &  \tau_{1}(e), \tau_{2}(e) \in \ZZ
     \qquad (e \in A) .
 \label{sbmFtensionint2} 
\end{eqnarray}
%%%%%%%%%%%%%%

We introduce notations for feasible flows and potentials/tensions:
\begin{align} 
 \mathcal{P}_{0} &= \{  x \in \ZZ\sp{A} :
 \mbox{\rm $x$ satisfies \eqref{sbmFdomphi1}, \eqref{sbmFflowdemand2}, \eqref{sbmFflowcapconst2} }
 \} ,
\label{primalfeassbmF2}
\\
 \mathcal{D}_{0} &= \{  (\pi, \tau_{1}, \tau_{2}) \in \ZZ\sp{V} \times \ZZ\sp{A}  \times \ZZ\sp{A}:  
 \eqref{sbmFpotdifftension2},  \ \pi \in \dom \hat{p}, \ \  
 \tau_{1}(e) \in \dom \psi_{e} \ \mbox{\rm for each } \  e \in A
 \} .
\label{dualfeassbmF2}
\end{align}

\begin{theorem} \label{THminmaxsbmF2}
Assume primal feasibility ($\mathcal{P}_{0} \not= \emptyset$) 
 or dual feasibility ($\mathcal{D}_{0} \not= \emptyset$).
Then we have the min-max relation:
\begin{align}
& \min \{  \Phi(x) : 
 x \in \ZZ\sp{A} \ \mbox{\rm \  satisfies \eqref{sbmFflowdemand2} and \eqref{sbmFflowcapconst2}} \}  
\notag \\ 
&= \max\{ \Psi(\pi, \tau_{1}, \tau_{2}) : 
\pi \in \ZZ\sp{V} \ \mbox{\rm and } \tau_{1},  \tau_{2}  \in \ZZ\sp{A} \ 
\mbox{\rm \  satisfy \eqref{sbmFpotdifftension2}} \} .
 \label{sbmFminmax2}
\end{align}
The unbounded case with both sides 
being equal to $-\infty$ or $+\infty$ is also a possibility.
\finbox
\end{theorem}

In the min-max formula \eqref{sbmFminmax2} 
we  denote the set of the minimizers $x$  by $\mathcal{P}$ 
and the set of the maximizers $(\pi, \tau_{1},\tau_{2})$ by $\mathcal{D}$.
The optimality criterion in Proposition \ref{PRminmaxsbmF1OPT} can be adapted for \eqref{sbmFminmax2}
as follows.

\begin{proposition} \label{PRminmaxsbmF2OPT}
Assume that both $\mathcal{P}_{0}$ and $\mathcal{D}_{0}$ 
in \eqref{primalfeassbmF2}--\eqref{dualfeassbmF2} are nonempty.

{\rm (1)}
Let $x \in \mathcal{P}_{0}$.
Then
$x \in \mathcal{P}$ 
if and only if there exists 
$(\pi, \tau_{1}, \tau_{2}) \in \mathcal{D}_{0}$
such that
\begin{align}
 & 
\varphi_{e}(x(e)) - \varphi_{e}(x(e)-1) 
\leq \tau_{1}(e) \leq 
\varphi_{e}(x(e)+1) - \varphi_{e}(x(e))
\qquad (e  \in A),  
\label{pisubgradsbmF2}
\\
 & 
 \tau_{2}(e) 
   \left\{  \begin{array}{ll}
    =0   & \mbox{\rm if \ $f(e) +1 \leq x(e) \leq g(e) - 1$},  \\
    \leq 0   & \mbox{\rm if \ $x(e) =  f(e)$},  \\
    \geq 0   & \mbox{\rm if \  $x(e) =  g(e)$}  \\
             \end{array}  \right.
\qquad (e  \in A),  
\label{capaslacksbmF2}
\\
 & 
\mbox{\rm Net-in-flow vector $(\varrho _x(v)-\delta _x(v) : v \in V )$ 
 is a $\pi$-minimizer in $\odotZ{B}$}.
\label{pi1minzersbmF2}
\end{align}

{\rm (2)}
Let 
$(\pi, \tau_{1}, \tau_{2}) \in \mathcal{D}_{0}$.
Then
$(\pi,\tau_{1},\tau_{2}) \in \mathcal{D}$
if and only if there exists
$x \in \mathcal{P}_{0}$ 
that satisfies
\eqref{pisubgradsbmF2}, \eqref{capaslacksbmF2}, and \eqref{pi1minzersbmF2}.

{\rm (3)}
For any $(\hat\pi,\hat\tau_{1},\hat\tau_{2}) \in \mathcal{D}$, we have
\begin{equation}  \label{primaloptsetsbmF2}
\mathcal{P} =  \{ x \in \mathcal{P}_{0}: 
 \mbox{\rm \eqref{pisubgradsbmF2}, \eqref{capaslacksbmF2}, \eqref{pi1minzersbmF2} 
 hold with $(\pi,\tau_{1},\tau_{2}) = (\hat\pi,\hat\tau_{1},\hat\tau_{2})$}  \} ,
\end{equation}
where the conditions 
in \eqref{pisubgradsbmF2} and \eqref{capaslacksbmF2} 
can be rewritten as
\begin{align}
 & 
 x(e) \in \argmin_{k} \{ \varphi_{e}(k) - \tau_{1}(e) k \}  
\qquad (e  \in A),  
\label{pisubgradsbmF2B}
\\
 & 
   \left\{ \begin{array}{ll}
    x(e) = f(e)   & \mbox{\rm if} \ \   \tau_{2}(e) < 0,  \\
    f(e) \leq x(e) \leq g(e)  & \mbox{\rm if} \ \   \tau_{2}(e)=0,  \\
    x(e) = g(e)   & \mbox{\rm if} \ \   \tau_{2}(e) > 0  \\
             \end{array}  \right.
\qquad (e  \in A).
\label{capaslacksbmF2B}
\end{align}

{\rm (4)}
For any $\hat x \in \mathcal{P}$, we have
\begin{equation}  \label{dualoptsetsbmF2}
\mathcal{D} =  \{ (\pi,\tau_{1},\tau_{2}) \in \mathcal{D}_{0} : 
\mbox{\rm \eqref{pisubgradsbmF2}, \eqref{capaslacksbmF2}, \eqref{pi1minzersbmF2} 
 hold with $x = \hat x$}  \} .
\end{equation}
\end{proposition}

\begin{proof}
Rather than translating the conditions in Proposition \ref{PRminmaxsbmF1OPT} 
for the present case,
we prove the claim by considering the weak duality 
$\Phi(x) \geq  \Psi(\pi,\tau_{1},\tau_{2})$ directly for this case.

Let $x$ and $(\pi, \tau_{1},\tau_{2})$ be primal and dual feasible solutions. Then, using 
the Fenchel--Young inequality \eqref{youngineq}
for $(\varphi_{e}, \psi_{e})$
and \eqref{sbmFpotdifftension2},
we obtain
\begin{align}
 & \Phi(x) -  \Psi(\pi,\tau_{1},\tau_{2})
\notag  \\ &  
  = 
 \sum_{e  \in A}   [ \varphi_{e} ( x(e) ) +  \psi_{e} ( \tau_{1}(e) ) ] 
 + \sum_{e  \in A}  \max \{ f(e) \tau_{2}(e) , \ g(e) \tau_{2}(e) \}
  \ -\hat{p}(\pi) 
\notag  \\ &  
\geq
  \sum_{e \in A}    x(e) \tau_{1}(e)
 + \sum_{e  \in A}  \max \{ f(e) \tau_{2}(e) , \ g(e) \tau_{2}(e) \}
  \   -\hat{p}(\pi) 
\label{sbmF2weakd1} 
\\ &  
=   \sum_{e=uv \in A}    x(e) ( \pi(v) - \pi(u) - \tau_{2}(e) )
+ \sum_{e  \in A}  \max \{ f(e) \tau_{2}(e) , \ g(e) \tau_{2}(e) \}
  \   -\hat{p}(\pi) 
\notag  \\ &  
=  \sum_{e = uv \in A}    x(e) ( \pi(v) - \pi(u) ) 
+ \sum_{e  \in A} \big[
   \max \{ f(e) \tau_{2}(e) , \ g(e) \tau_{2}(e) \} - x(e) \tau_{2}(e) 
                 \big]
  \   -\hat{p}(\pi) 
\notag \\ &  
=  \left( \sum_{v \in V}    \pi(v) (\varrho _x(v)-\delta _x(v))
  \   -\hat{p}(\pi)  \right) 
+ \sum_{e  \in A} \big[
   \max \{ f(e) \tau_{2}(e) , \ g(e) \tau_{2}(e) \} - x(e) \tau_{2}(e) 
                 \big] .
\notag 
\end{align}
For the former part of this expression we have
\begin{align}
   \sum_{v \in V}    \pi(v) (\varrho _x(v)-\delta _x(v))
  \   -\hat{p}(\pi)  
\geq 
\min \{ \pi z : z \in \odotZ{B} \}  \   -\hat{p}(\pi) = 0
\label{sbmF2weakd2}
\end{align}
by the feasibility condition \eqref{sbmFflowdemand1}
and the expression \eqref{plovextmin} for $\hat{p}(\pi)$,
whereas,  for each summand in the latter part we have
\begin{align}
   \max \{ f(e) \tau_{2}(e) , \ g(e) \tau_{2}(e) \} - x(e) \tau_{2}(e) \geq 0
\label{sbmF2weakd3}
\end{align}
since $f(e) \leq x(e) \leq g(e)$ 
by the capacity constraint \eqref{sbmFflowcapconst2}.
Thus the weak duality is established.
The optimality conditions can be obtained as the conditions
for the inequalities in \eqref{sbmF2weakd1},  \eqref{sbmF2weakd2} and \eqref{sbmF2weakd3}
to be equalities, as in the proof of Proposition \ref{PRminmaxsbmF1OPT}.
\end{proof}

\begin{remark} \rm  \label{RMsbmFfeasible}
The feasibility of the primal problems
can be expressed in terms of the submodular function $b$ 
as follows,
where we denote the integer interval of $\dom \varphi_{e}$ by $[ f'(e), g'(e) ]_{\ZZ}$
with $f'(e) \in \ZZ \cup \{ -\infty \}$ and $g'(e) \in \ZZ \cup \{ +\infty \}$.
Then there exists a feasible flow for the basic problem 
\eqref{sbmFprimalobj1}--\eqref{sbmFflowint1}
if and only if 
\begin{equation}
\varrho_{f'}(Z)-\delta_{g'}(Z) \leq b(Z)
\qquad 
\hbox{for all}
\quad  Z\subseteq V
\label{sbmFhoffman} 
\end{equation}
is satisfied.
For the problem 
\eqref{sbmFprimalobj2}--\eqref{sbmFflowint2}
with explicit capacity constraints,
we replace 
$f'(e)$ and $g'(e)$ by 
$\max \{ f(e), f'(e) \}$ and $\min \{ g(e), g'(e) \}$, respectively.
\finbox
\end{remark}

%%% end file %%%

%% murota 2018-08-25 / 2019-05-22 / 2019-07-09

\section{Survey of early papers}
\label{SCprevwksurvey}

This appendix offers a brief survey of earlier papers and  books
that deal with topics closely related to decreasing minimization
on base-polyhdera.
To be specific, we mention the following:
Veinott \cite{Vei71}  (1971),
Megiddo  \cite{Meg74} (1974),
Fujishige \cite{Fuj80} (1980),
Groenevelt \cite{Gro91} (1985, 1991),
Federgruen--Groenevelt \cite{FG86} (1986), 
Ibaraki--Katoh \cite{IK88} (1988),
Dutta--Ray \cite{DR89} (1989),
Fujishige \cite{Fuj05book}  (1991, 2005),
Hochbaum \cite{Hoc94} (1994),
and Tamir \cite{Tami95}  (1995).

Similar notions and terms are scattered in the literature such as 
``egalitarian,''
``lexicographically optimal,''
``least majorized,''
``least weakly submajorized,''
``decreasingly minimal (dec-min),'' and
``increasingly maximal (inc-max).''
Unfortunately, these notions are discussed often independently
in different context, without proper mutual recognition.
The term ``least majorized'' is used in Veinott \cite{Vei71} and
``Least weakly submajorized'' is used in Tamir \cite{Tami95}.
These terms are not used in Marshall--Olkin--Arnold \cite{MOA11}.
Dutta--Ray \cite{DR89} uses ``egalitarian'' and does not use ``majorization.''
The term ``lexicographically optimal'' in 
Veinott \cite{Vei71}, Megiddo \cite{Meg74,Meg77}, and Fujishige \cite{Fuj80,Fuj05book}
means ``increasingly maximal (inc-max).''

Three notions ``dec-min'',  ``inc-max'', and ``least majorized'' are different in general.
Generally, ``least majorized'' implies ``dec-min'' and  ``inc-max'', but the
converse is not true (see Section \ref{SCmajordecmin}).
In base-polyhedron (in $\RR$ and $\ZZ$), however,
the three notions coincide (see Section \ref{SCmajorbasepoly}).

Another important aspect in majorization is minimization of symmetric separable convex functions.
An element is least majorized if and only if
it simultaneously minimizes all symmetric separable convex functions
(see Proposition \ref{PRmajorchar}).
Therefore, if a least majorized is known to exist, then it can be computed
as a minimizer of the square-sum.

\subsubsection*{Veinott (1971) \cite{Vei71}}

This paper deals with a network flow problem. 
The ground set is a star of arcs, i.e., the set of arcs incident to a single node.
This amounts to considering a special case of a base-polyhedron.
The main result is the  unique  existence of a least majorized element in Case $\RR$. 

The computational aspect is also discussed.
The problem is reduced to separable quadratic network flow problem.
Then the paper describes an algorithm for nonlinear convex cost minimum flow problem.
It also defines the dual problem using the conjugate function.
Complexity of the algorithm is not discussed.

Case $\ZZ$ is also treated.  Theorem 2 (1)
shows the existence of an integral element that simultaneously minimizes
all symmetric separable convex functions.  
The proof is based on rounding argument (continuous relaxation). 
That is, for a discrete convex function in integers, 
its piecewise-linear extension is considered and the integrality theorem
is used to derive the existence of an integral minimizer.
Thus the existence of a least majorized element is shown 
for the network flow in Case $\ZZ$.

\subsubsection*{Megiddo (1974) \cite{Meg74}}
This paper deals with a network flow problem. 
The ground set is the set of multi-terminals.
This is more general than a star considered in Veinott \cite{Vei71}, 
but the difference is not really essential.
The paper defines the notions of ``sink-optimality'' and ``source-optimality,''
which are increasing-maximality for vectors on the sink and source terminals, respectively. 
This paper considers Case $\RR$ only.
The main result is the characterization of an inc-max element 
using a chain of cuts in the network (Theorem~4.6).
The computational aspect is discussed in the companion paper \cite{Meg77},
which gives an algorithm of complexity $O(n\sp{5}$).

\subsubsection*{Fujishige (1980) \cite{Fuj80}}

This is the first paper that deals with base-polyhedra, 
beyond network flows. 
It considers Case $\RR$ only.
Lexicographic optimality with respect to a weight vector is defined.
The lexicographically optimal base with respect to a uniform weight
coincides with the inc-max element of the base-polyhedron.
The relation to weighted square-sum  minimization is investigated in detail 
and the minimum norm point is highlighted.
The principal partition for base-polyhedra is introduced, 
as a generalization of the known construction for matroids.
The principal partition determines the lexico-optimal base.
The proposed decomposition algorithm
finds the lexico-optimal base 
as well as the principal partition in strongly polynomial time.
While this paper covers various aspects of the lexico-optimal base,
the majorization viewpoint is missing.
In particular, it is not stated that the minimum norm point 
is actually a minimizer of all symmetric separable convex functions.

\subsubsection*{Groenevelt (1985, 1991) \cite{Gro91}}

The technical report appeared in 1985, and the journal version in 1991.
Already the technical report was influential, 
cited by  \cite[1st ed.]{Fuj05book},  \cite{Hoc94}, and \cite{IK88}.

The main concern of this paper is separable convex minimization
(not restricted to symmetric separable convex functions)
on base-polyhedra.
Both continuous variables (Case $\RR$) and 
discrete variables (Case $\ZZ$) are treated.
In particular, this is the first paper that addressed minimization of 
separable convex functions
on base-polyhedra in discrete variables.
One of the results says that,
in any integral base-polyhedron,
there exists an integral element that 
is a (simultaneous) minimizer of all symmetric separable convex functions. 
This paper does not discuss implications of this result
to inc-maximality, dec-minimality, or majorization,
though the result does imply the existence of a least majorized element
by virtue of the well-known fact 
(Proposition \ref{PRmajorchar}) about majorization.

The paper presents two kinds of algorithms,
the marginal allocation algorithm (of incremental type)
and the decomposition algorithm (DA).
Concerning complexity, the author argues that 
the algorithms are polynomial if the  base-polyhedron 
are of some special types (tree-structured polymatroids, generalized symmetric polymatroids,
network polymatroids).
We quote the following statements from \cite[p.234, journal version]{Gro91},
where $E$ denotes the ground set of a base-polyhedron and $N$ is the associated submodular function,
which is integer-valued in Case $\ZZ$:
\begin{quote}
The total complexity of DA is thus
${\rm O}(|E| (\tau_{1} + \tau_{2} ))$,
where $\tau_{1}=$ the number of operations needed to solve a single
constraint problem, and $\tau_{2}=$ the number of operations
needed to perform one pass through Steps 2 and 3. 
It is well-known that in the discrete case 
$\tau_{1} = {\rm O}(|E| \log (N(E)/|E|))$
(see Frederickson and Johnson (1982)), and in the continuous case
$\tau_{1} = {\rm O}(|E| \log |E| +\chi)$,
where $\chi$ is the time needed to solve a certain type of
non-linear equation (see Zipkin, 1980).
\end{quote}

This paper was written in 1985 
and at that time, no strongly polynomial algorithm
for submodular function minimization was known;
the strongly polynomial algorithm (using the ellipsoid method)
first appeared in 1993  \cite[2nd edition]{GLS93}.

\subsubsection*{Federgruen--Groenevelt (1986) \cite{FG86}}
This paper deals with base-polyhedra in Case $\ZZ$.
Main concern of this paper is to offer a general framework in which a greedy procedure
called the marginal allocation algorithm (MAA) works.
The concept of concave order is introduced as a class of admissible
objective functions for which the greedy procedure works.
The main result (Corollary 1 in Sec.3) states, roughly,
that the MAA gives an optimal solution for every weakly concave order on polymatroids.

\subsubsection*{Ibaraki--Katoh (1988) \cite{IK88}}

This is the first comprehensive book for algorithmic aspects of
the resource allocation problem and its extensions.
Chapter~9, entitled ``Resource allocation problems under submodular constrains''
presents the fundamental and up-to-date results at that time,
including those by Fujishige \cite{Fuj80}, Groenevelt \cite{Gro91},
and Federgruen--Groenevelt \cite{FG86}.
In particular, Theorem 9.2.2 \cite[p.156]{IK88}
states that the decomposition algorithm runs in polynomial time in $|E|$ and $\log M$,
where 
$E$ is the ground set and
$M$ is an upper bound on $r(E)$
for the submodular function $r$ expressing the submodular constraint.

The contents of Chapter 9 of this book are updated in a handbook chapter
by Ibaraki--Katoh \cite{KI98} in 1998.
Its revised version by Katoh--Shioura--Ibaraki \cite{KSI13} in 2013
incorporates the views from discrete convex analysis.

\subsubsection*{Dutta--Ray (1989) \cite{DR89}}
This paper deals with base-polyhedra in the context of game theory. 
Recall that the core of a convex game is nothing but the base-polyhedron.
Naturally this paper deals exclusively with Case $\RR$. 
According to Tamir \cite{Tami95},
this is the first paper proving
the existence of a least majorized element in a base-polyhedron.
Technically speaking, 
this result could be obtained  from a combination of the results of Groenevelt \cite{Gro91}
(which was written in 1985 and published in 1991)
and a well-known fact 
``least majorized element $\Leftrightarrow$ simultaneous minimizer of all symmetric separable convex functions''
(see Proposition \ref{PRmajorchar}).
However, Dutta--Ray \cite{DR89} and Groenevelt \cite{Gro91} were unaware
of each other; see Table \ref{TBreferrelation} at the end of Appendix.
We also note that Fujishige \cite{Fuj80} deals with quadratic functions only,
and hence the results of \cite{Fuj80} do not imply the existence of a least majorized element.

\subsubsection*{Fujishige (1st ed., 1991; 2nd ed. 2005) \cite{Fuj05book}}

This book offers a comprehensive exposition of the results of Fujishige \cite{Fuj80} about
the lexico-optimal (inc-max) element of a base-polyhedron in Case $\RR$.
There is an explicit statement at the beginning of Section 9 
that the argument is not applicable to Case $\ZZ$.

For separable convex minimization, both Cases $\RR$ and $\ZZ$ are treated.
In particular, the results of Groenevelt \cite{Gro91} are described 
in a manner consistent with the other part of this book. 
It is stated that the decomposition algorithm works for Cases $\RR$ and $\ZZ$,
but complexity analysis is explicit only for Case $\RR$.
It is shown that the decomposition algorithm is strong polynomial for Case $\RR$.
As a natural consequence of the fact 
that lexico-optimal bases in Case $\ZZ$ are not considered in this book,
no connection is made between 
separable convex minimization and lexico-optimality (inc-max, dec-min)
in Case $\ZZ$.

Majorization concept is not treated in the first edition,
whereas in the second edition
the definition is given in Section 2.3 (p.~44)
and a reference to 
Dutta--Ray \cite{DR89} is added in Section 9.2 (p.~264).
%%2019-05-24

\subsubsection*{Hochbaum (1994) \cite{Hoc94}}

This paper shows that there exist no strongly polynomial time algorithms
to solve the resource allocation problem with a separable convex cost function.
Subsequently, Hochbaum and her coworkers made significant contributions
to resource allocation problems in discrete variables, 
dealing with important special cases and showing improved complexity bounds
for the special cases (e.g., Hochbaum--Hong \cite{HH95}).
The survey paper by Hochbaum \cite{Hoc07} is informative and useful.

\subsubsection*{Tamir  (1995) \cite{Tami95}}

This papers deals with g-polymatroids in Case $\RR$ and Case $\ZZ$.
The relationship between majorization and decreasing-minimality 
is discussed explicitly.

The main result is the existence of a least weakly submajorized element
in a g-polymatroid.
The following sentences concerning Case $\RR$ in pages 585--585 are informative:
\begin{quote}
 Fujishige (1980) extends the results of Megiddo to a general polymatroid and presents
an algorithm to find a lexicographically optimal base of the polymatroid 
with respect to an arbitrary positive weight vector $d$.
This weighted model is closely related to the concept of $d$-majorization
introduced by Veinott (1971).
 Neither Megiddo nor Fujishige relate their results on lexicographically optimal bases 
to the stronger concept of majorization.
 (From Proposition 2.1 we note that if an arbitrary set has a least majorized element 
it is clearly lexicographically optimal. 
However, every convex and compact set $S$ has a unique lexicographically maximum element, 
but might not have a least majorized element.) 
The fact that a polymatroid 
has a least majorized base is shown by Dutta and Ray (1989).
They consider the core of a convex game as defined by Shapley (1971), 
which corresponds to a polymatroid. (Strictly speaking the former is defined as a contra-polymatroid; 
see next section.) We will extend and unify the above results 
by proving that a bounded generalized polymatroid contains 
both least submajorized and least supermajorized elements. 
\end{quote}

For the complexity of finding the unique minimizer $x^{*} \in \RR\sp{n}$ 
of the square-sum over a g-polymatroid (Case $\RR$),
the following statement can be found in page 587:
\begin{quote}
$x^{*}$ can be found in strongly polynomial time by modifying the procedure in 
Fujishige (1980) and Groenevelt (1991) which is applicable to polymatroids.
The latter procedure can now be implemented to solve any convex separable quadratic over a polymatroid 
in a strongly polynomial time since its complexity is dominated by the efforts to minimize
a (strongly) polynomial number of submodular functions.
\end{quote}
There is no statement about the complexity in Case $\ZZ$.

\medskip

%%%%%%%%%% table %%%%%%%%%%
\begin{table}[h]
\begin{center}
\caption{Referencing relations between papers}
\label{TBreferrelation}

\medskip

\addtolength{\tabcolsep}{-3pt}%
\begin{tabular}{l|cccccccccc}
  & Vei & Meg & Fuj & Gro & F-G & I-K & D-R & Fuj & Hoc & Tam 
\\ 
& \cite{Vei71} & \cite{Meg74} & \cite{Fuj80}  & \cite{Gro91} & \cite{FG86} & \cite{IK88} & \cite{DR89} & \cite{Fuj05book} & \cite{Hoc94} & \cite{Tami95}  
\\ \hline
Veinott 1971 
& $\cdot$ & -- & -- & -- & -- & -- & -- & -- & -- & --  
\\
Megiddo 1974 
& -- & $\cdot$ & -- & -- & -- & -- & -- & -- & -- & --  
\\
Fujishige 1980 
& -- & R & $\cdot$ & -- & -- & -- & -- & -- & -- & --  
\\
Groenevelt 1985/91 
& -- & R & R & $\cdot$ &  R & -- & -- & -- & -- & --  
\\
{\small Federgruen--Groenevelt 1986}
& -- & R & R & -- & $\cdot$ & -- & -- & -- & -- & --  
\\ 
Ibaraki--Katoh 1988 
& -- & R & R & R & R & $\cdot$  & -- & -- & -- & --  
\\
Dutta--Ray 1989  
& -- & -- & -- & -- & -- & -- & $\cdot$ & -- & -- & --  
\\
Fujishige 1991 (1st ed.) 
& -- & R & R & R & -- & R  & R$\sp{\rm 2nd}$ & $\cdot$ & -- & --  
\\
Hochbaum 1994 
& -- & -- & -- & R & R & R  & -- & -- & $\cdot$ & --  
\\
Tamir 1995
& R & R & R & R & -- & -- & R & R & -- & $\cdot$   
\\
\hline
 \multicolumn{11}{l}{\qquad Paper at the left refers to papers marked \ R \  in the same row}
\\
\multicolumn{11}{l}{\qquad R$\sp{\rm 2nd}$ means that reference is made in the 2nd edition (2005) only}
\end{tabular}
\addtolength{\tabcolsep}{3pt}%
\\
\end{center}
\end{table}
%%%%%%%%%% table %%%%%%%%%%

%%% end file %%%

%% murota 2018-08-25 / 2019-02-19 / 2019-04-17 / 2019-07-09 / 2020-06-30

\newpage

%%%%%%%%%%%%%%%%%%%%%
%%\renewcommand{\refname}{References to Part II}

%%\addtocontents{toc}{\quad \\ }
%%\addtocontents{toc}{\bf References}

%\addcontentsline{toc}{section}{\refname}
%%\phantomsection

%%% end file %%%

%%\newpage
%%\tableofcontents 


\begin{thebibliography}{99}
\setlength{\itemsep}{0pt}
%%\setcounter{enumiv}{100}


\bibitem{AHO03} 
%% R. K. Ahuja, D. S. Hochbaum, and J. B. Orlin:
Ahuja, R.K., Hochbaum, D.S., Orlin, J.B.:
Solving the convex cost integer dual network flow problem.
Management Science {\bf 49}, 950--964 (2003)




\bibitem{And96} 
%%K. Ando,
Ando, K.:
Weak majorizations on finite jump systems. 
Mimeo (1996)
Available from author's home page.
%%http://coconut.sys.eng.shizuoka.ac.jp/ando/maj/maj11.dvi


\bibitem{AFN95} 
%%K. Ando, S. Fujishige, and T. Naitoh,
Ando, K., Fujishige, S., Naitoh, T.:
A greedy algorithm for minimizing a separable convex function
over a finite jump system.
Journal of the Operations Research Society of Japan {\bf 38}, 362--375 (1995)



\bibitem{AS18} 
%% B.C. Arnold,  J.M. Sarabia,
Arnold, B.C., Sarabia, J.M.:
Majorization and the Lorenz Order with Applications 
in Applied Mathematics and Economics.
Springer International Publishing, Cham (2018),
%%Springer International Publishing
%%Cham, Switzerland
%%https://www.palgrave.com/jp/book/9783319937724
(1st edn., 1987)




\bibitem{Dut90} 
%%B. Dutta,
Dutta, B.:
The egalitarian solution and reduced game properties in convex games. 
International Journal of Game Theory {\bf 19}, 153--169 (1990)



\bibitem{DR89} 
%%B. Dutta,    D. Ray, 
Dutta, B.,    Ray, D.:
A concept of egalitarianism under participation constraints.
Econometrica {\bf 57}, 615--635 (1989) 



\bibitem{FG86} 
%%A. Federgruen,  H. Groenevelt, 
Federgruen, A., Groenevelt, H.:
The greedy procedure for resource allocation problems:
necessary and sufficient conditions for optimality.
Operations Research {\bf 34}, 909--918 (1986) 

\bibitem{Fra11book} 
%%A. Frank, 
Frank, A.:
Connections in Combinatorial Optimization.
Oxford University Press, Oxford (2011)


\bibitem{FM18part1} 
%%A. Frank,    K.  Murota,
Frank, A.,  Murota, K.:
Discrete decreasing minimization, Part I:
Base-polyhedra with applications in network optimization.
arXiv: 1808.07600 (August 2018)
%%arXiv: http://arxiv.org/abs/1808.07600, August 2018.




\bibitem{FM18part3} 
%%A. Frank,    K.  Murota,
Frank, A.,  Murota, K.:
Discrete decreasing minimization, Part III: Network flows,
arXiv: 1907.02673 (July 2019)
%%arXiv: http://arxiv.org/abs/1907.02673.









\bibitem{Fuj80} 
%%S. Fujishige, 
Fujishige, S.:
Lexicographically optimal base of a polymatroid 
with respect to a weight vector.
Mathematics of Operations Research {\bf 5}, 186--196 (1980)

\bibitem{Fuj05book}
%%S. Fujishige, 
Fujishige, S.:
Submodular Functions and Optimization,
1st edn.
Annals of Discrete Mathematics {\bf 47}, 
North-Holland, Amsterdam (1991);
2nd edn.
Annals of Discrete Mathematics {\bf 58},
Elsevier, Amsterdam  (2005)



\bibitem{Fuj09bonn}
%%S. Fujishige, 
Fujishige, S.: 
Theory of principal partitions revisited.
In: Cook, W., Lov{\'a}sz, L., Vygen, J. (eds.)
Research Trends in Combinatorial Optimization,  pp.~127--162. 
Springer, Berlin (2009)



\bibitem{FKI88} 
%%S. Fujishige,   N. Katoh, N.,   T. Ichimori, 
Fujishige, S.,    Katoh, N.,  Ichimori, T.:
The fair resource allocation problem with submodular constraints.
Mathematics of Operations Research {\bf 13}, 164--173 (1988)





\bibitem{Gro91} 
%%H. Groenevelt,
Groenevelt, H.:
Two algorithms for maximizing a separable concave function
over a polymatroid feasible region.
European Journal of Operational Research {\bf 54}, 227--236 (1991);
The technical version appeared as 
Working Paper Series No. QM 8532,
Graduate School of Management,
University of Rochester (1985)



\bibitem{GLS93}
%%M. Gr{\"o}tschel,   L. Lov{\'a}sz,    A. Schrijver, 
Gr{\"o}tschel, M., Lov{\'a}sz, L., Schrijver, A.: 
Geometric Algorithms and Combinatorial Optimization, 2nd edn. 
Springer, Berlin (1993)
%%1st ed., 2nd. ed., 
%%Springer, Berlin (1988, 1993)




\bibitem{Hoc94}
%%D.S. Hochbaum,
Hochbaum, D.S.:
Lower and upper bounds for the allocation problem and other
nonlinear optimization problems.
Mathematics of Operations Research {\bf 19}, 390--409 (1994)






\bibitem{Hoc07}
%%D.S. Hochbaum,
Hochbaum, D.S.:
Complexity and algorithms for nonlinear optimization problems.
Annals of Operations Research {\bf 153}, 257--296 (2007)



\bibitem{HH95}
%%D.S. Hochbaum,  S.-P. Hong,
Hochbaum, D.S., Hong, S.-P.:
About strongly polynomial time algorithms for
quadratic optimization over submodular constraints.
Mathematical Programming {\bf 69}, 269--309 (1995)




\bibitem{IK88} 
%%T. Ibaraki,   N. Katoh, 
Ibaraki, T., Katoh, N.:
Resource Allocation Problems: Algorithmic Approaches.
MIT Press, Boston (1988)


\bibitem{Iri69book} 
%% M. Iri:
Iri, M.:
Network Flow, Transportation and Scheduling---Theory and Algorithms.
Academic Press, New York (1969)



\bibitem{Iri79}
%%M. Iri,
Iri, M. : 
A review of recent work in Japan on principal partitions of matroids and their applications. 
Annals of the New York Academy of Sciences {\bf 319}, 306--319  (1979) 



\bibitem{KI98} 
%%N. Katoh,   T. Ibaraki,
Katoh, N., Ibaraki, T.:
Resource allocation problems.
In: Du, D.-Z., Pardalos, P.M. (eds.) 
Handbook of Combinatorial Optimization, Vol.2,
pp.~159--260. Kluwer Academic Publishers, Boston (1998)



\bibitem{KSI13} 
%%N. Katoh,    A. Shioura,   T. Ibaraki, 
Katoh, N.,  Shioura, A.,  Ibaraki, T.:
Resource allocation problems.
In: Pardalos, P.M., Du, D.-Z.,   Graham, R.L. (eds.) 
Handbook of Combinatorial Optimization, 2nd ed., 
Vol. 5, 
pp.~2897-2988, Springer, Berlin (2013) 



\bibitem{LO16shift} 
%%A. Levin,   S.  Onn,
Levin, A.,  Onn, S.:
Shifted matroid optimization. 
Operations Research Letters {\bf 44}, 535--539 (2016)


\bibitem{MOA11}
%%A.W. Marshall,    I. Olkin,   B.C. Arnold, 
Marshall, A.W.,   Olkin, I.,  Arnold, B.C.:
Inequalities: Theory of Majorization and Its Applications,
2nd edn. 
Springer, New York (2011),
%%http://kryakin.org/nierow/Marshall-Olkin-inequalities.pdf
(1st edn., 1979)





\bibitem{Meg74}
%%N. Megiddo, 
Megiddo, N.:
Optimal flows in networks with multiple sources and sinks.
Mathematical Programming {\bf 7}, 97--107 (1974)


\bibitem{Meg77}
%%N. Megiddo, 
Megiddo, N.:
A good algorithm for lexicographically optimal flows in multi-terminal networks.
Bulletin of the American Mathematical Society {\bf 83}, 407--409 (1977)





\bibitem{MST11Mrelax} 
%%S. Moriguchi,   A. Shioura,   N. Tsuchimura,
Moriguchi, S., Shioura, A., Tsuchimura, N.:
M-convex function minimization by continuous relaxation 
approach---Proximity theorem and algorithm. 
SIAM Journal on Optimization {\bf 21}, 633--668 (2011)


\bibitem{Mstein96} 
%%K. Murota, 
Murota, K.:
Convexity and Steinitz's exchange property. 
Advances in Mathematics {\bf 124}, 272--311 (1996)


\bibitem{Mdca98} 
%%K. Murota, 
Murota, K.:
Discrete convex analysis. 
Mathematical Programming  {\bf 83}, 313--371 (1998)

\bibitem{Msbmfl99}
%% K. Murota:
Murota, K.:
Submodular flow problem with a nonseparable cost function.
Combinatorica {\bf 19}, 87--109 (1999)


\bibitem{Mdcasiam} 
%%K. Murota, 
Murota, K.:
Discrete Convex Analysis.
%%SIAM Monographs on Discrete Mathematics and Applications, Vol.~10, 
Society for Industrial and Applied Mathematics, Philadelphia (2003)



\bibitem{Mbonn09} 
%%K. Murota, 
Murota, K.:
Recent developments in discrete convex analysis.
In: Cook, W., Lov{\'a}sz, L., Vygen, J. (eds.)
Research Trends in Combinatorial Optimization,
Chapter 11, pp.~219--260. Springer, Berlin (2009) 






\bibitem{Onn10book} 
%% S. Onn:
Onn, S.: 
Nonlinear Discrete Optimization: An Algorithmic Theory.
European Mathematical Society, Zurich (2010)


\bibitem{Roc84}
%% R. T. Rockafellar:
Rockafellar, R.T.:
Network Flows and Monotropic Optimization.
Wiley, New York (1984)


\bibitem{Sch03} 
%% A. Schrijver:
Schrijver, A.:
Combinatorial Optimization---Polyhedra and Efficiency.
Springer, Heidelberg (2003)

\bibitem{Tami95}
%%A. Tamir, 
Tamir, A.: 
Least majorized elements and generalized polymatroids.
Mathematics of Operations Research {\bf 20}, 583--589 (1995)



\bibitem{Vei71}
%%A.F. Veinott, Jr., 
Veinott, Jr., A.F.: 
Least $d$-majorized network flows with inventory and statistical applications.
Management Science {\bf 17}, 547--567 (1971)




\end{thebibliography}
\end{document}